\documentclass[opre,nonblindrev]{informs3} % current default for manuscript submission
\SingleSpacedXI % current default line spacing

\usepackage[T1]{fontenc}

\usepackage{color}              % Need the color package
\usepackage{color-edits}[suppress]
\addauthor{bs}{cyan}
\addauthor{hk}{purple}
\usepackage{colortbl}
\usepackage[table]{xcolor}
\definecolor{lightyellow}{RGB}{255, 255, 204} % Soft yellow
\definecolor{darkred}{RGB}{139,0,0} 
\usepackage{natbib}
 \bibpunct[, ]{(}{)}{,}{a}{}{,}%
 \def\bibfont{\small}%
\TheoremsNumberedThrough     % Preferred (Theorem 1, Lemma 1, Theorem 2)
\ECRepeatTheorems

 \usepackage{float}
\newfloat{algorithm}{t}{lop}
\floatname{algorithm}{Algorithm}

\usepackage{afterpage}

\usepackage[colorlinks,citecolor=black,urlcolor=blue,linkcolor=blue]{hyperref} 

\EquationsNumberedThrough    % Default: (1), (2), ..

\usepackage{booktabs} % For professional looking tables

\usepackage{mathrsfs}
\usepackage{bm,natbib,soul}

\usepackage[dvipsnames]{xcolor}
\usepackage{color}  % Need the color package
\usepackage{color-edits}
\addauthor{bs}{purple}
\addauthor{hk}{ForestGreen}
\usepackage{float}

\RequirePackage[shortlabels]{enumitem}
\usepackage{mathtools} % for coloneqq
\usepackage{enumitem} % for enumeration
\usepackage{dsfont} % for indicator function
\usepackage{mathrsfs} 
\usepackage{verbatim}
\usepackage{romannum}

\AtBeginDocument{\pagenumbering{arabic}}
\newcommand{\bq}{{\bm q}}

\newcommand{\bx}{{\bm x}}
\newcommand{\by}{{\bm y}}
\newcommand{\bz}{{\mathbf{z}}}

\newcommand{\balpha}{{\boldsymbol{\alpha}}}
\newcommand{\bbeta}{{\boldsymbol{\beta}}}

\newcommand{\bdelta}{{\boldsymbol{\delta}}}

\newcommand{\Exp}{\mathbb{E}}

\let\R\Real

\usepackage{graphicx}
\usepackage{bbm}
\usepackage{tikz,ifthen}
 \usetikzlibrary{math}
 \usetikzlibrary{shapes.misc}
 \usetikzlibrary{shapes.multipart,positioning,patterns,backgrounds}

\usepackage{multirow}
\usepackage{array}
 \usepackage[caption=false]{subfig}

\begin{document}
%%%%%%%%%%%%%%%%

%\setcounter{page}{1}
%\DOI{10.1287/xxxx.0000.0000}%

\RUNAUTHOR{Khalid and Sturt}

\RUNTITLE{Assortment Optimization and the Sample Average Approximation}
% \TITLE{Bundling Information Goods of Decreasing Value}
\TITLE{Assortment Optimization and the Sample Average Approximation}

\ARTICLEAUTHORS{%
\AUTHOR{Hassaan Khalid}
\AFF{Department of Information and Decision Sciences\\
University of Illinois Chicago, \EMAIL{hkhali23@uic.edu}} %, \URL{}}
\AUTHOR{Bradley Sturt}
\AFF{Department of Information and Decision Sciences\\
University of Illinois Chicago, \EMAIL{bsturt@uic.edu}}
} % end of the block

     \ABSTRACT{{\looseness=-1 \spaceskip= 2pt plus 1pt minus 1.5pt  \spaceskip= 3pt plus 2pt minus 2pt 
     We consider a simple approach to solving assortment optimization under the random utility maximization model. The approach uses Monte-Carlo simulation to  construct a ranking-based choice model that serves as a proxy for the true choice model, followed by finding an assortment that is optimal with respect to that proxy.  In this paper, we make that  approach more viable  by developing  faster algorithms for finding assortments that are optimal under ranking-based choice models. Our algorithms are based on mixed-integer programming and  consist of stronger formulations as well as  new structural and algorithmic results related to  Benders cuts.  We demonstrate that our algorithms---without any heuristics   or parameter tuning---can offer more than a 20x speedup in real-world settings with thousands of products and samples. Equipped with  our algorithms, we  showcase the value of using the sample average approximation to solve assortment optimization problems for which no practically efficient algorithms are known. 
     \looseness=-1  
     }}

\HISTORY{First version: Sept.\,30, 2025.  }
\maketitle

\vspace{-0.7em}

\section{Introduction} \label{sec:intro}

The sample average approximation is one of the most celebrated and effective methods for solving \emph{stochastic programs}---optimization problems in which the goal is to maximize the expectation of a reward function. This method has been deployed by organizations in nearly every subfield of operations research and management science, ranging from inventory control to energy planning.  The attractiveness of the sample average approximation can be attributed to its versatility: it can be used to obtain near-optimal solutions for stochastic programs with virtually {any} probability distribution, as long as it is possible using Monte-Carlo simulation to obtain a large number of samples from the underlying probability distribution~\citep{shapiro2021lectures}. 

Assortment optimization has emerged over the past two decades 
as a fundamental and practically significant class of stochastic programs. Widely studied  in revenue management and marketing, assortment optimization refers to a class of stochastic programs in which the goal is to select a subset of products for a firm to offer to its customers that maximizes the firm's expected revenue. In the most common setting, the customer preferences are captured in the stochastic program by a random utility vector that is drawn from a joint probability distribution;  when presented with a subset of products,  the customer will purchase the product from the subset with the highest utility. 
Tailoring the probability distribution of the random utility vector gives rise to assortment optimization problems under canonical choice models such as the multinomial logit, nested logit, and mixed multinomial logit 
as well as more modern variants like neural network-based choice models.  
\looseness=-1

In contrast to many other classes of stochastic programs, assortment optimization problems are \emph{not} typically solved using the sample average approximation. 
Instead, most of the algorithms that have been developed in the past two decades  for solving assortment optimization problems have been tailored to {parametric}  probability distributions for the random utility vector. For example, in the case where the random utility vector is comprised of independent Gumbel random variables---in which the demand model is known as the multinomial logit model---a number of  algorithms have been {developed} for solving the assortment optimization problem. Exact and approximation  algorithms that run in practical computation times have {also} been developed for numerous other parametric probability distributions such as those corresponding to the mixed MNL choice model \citep{rusmevichientong2014assortment}, the Markov-chain choice model \citep{blanchet2016markov}, and many others. 
These so-called \emph{distribution-specific} algorithms  usually rely on the probability distribution generating an expected revenue for each assortment that has a simple closed-form representation.

Yet there are many probability distributions for which no {practically efficient} distribution-specific algorithms for solving the  assortment optimization problem are known. Examples  range from random utility vectors with simple parametric probability distributions (e.g., when the random utilities are a mixture of exponential random variables~\citep[Section 8.2]{aouad2023}) to random utility vectors with complex probability distributions (e.g., when the random utilities are generated using an artificial neural network \citep{aouad2022representing}). The prevalence of complex probability distributions for the random utility vector is likely to accelerate as deep learning and artificial intelligence become increasingly used in the construction of probability distributions for the random utility vector. In such situations, the lack of practically efficient algorithms is a key bottleneck for firms to operationalizing new probability distributions that accurately represent customer behavior.  

In view of the above, the sample average approximation has the potential to serve as a valuable tool for solving assortment optimization problems for which no practically efficient distribution-specific algorithms are known. 
Indeed, the combination of the sample average approximation and Monte-Carlo simulation can be used to solve assortment optimization problems under \emph{any}  probability distribution. This  combination can hence  provide firms a means to decouple the tasks of \emph{estimating} a probability distribution that accurately captures customer behavior from the task of \emph{optimizing} the expected revenue. This decoupling of estimation  and optimization, in turn, can 
 offer firms the flexibility to estimate probability distributions that most accurately capture their customer behavior.

The sample average approximation can also play an important role in designing new distribution-specific algorithms for assortment optimization. 
In many real-world applications---such as recommendation systems and online retail---it is crucial to develop algorithms  that can scale to problems involving tens or even hundreds of thousands of products. A central challenge firms face when considering a new choice model is determining whether it is promising enough to justify the investment in building a distribution-specific algorithm. This challenge is especially salient because, counterintuitively, a new choice model may not improve predictive accuracy but can still lead to substantially better assortments \cite[p.2830]{aouad2023}. In such cases, the sample average approximation can provide a practical way for firms to test the performance of a new choice model on smaller instances (e.g., with hundreds of products) before committing resources to developing a scalable approximation algorithm (e.g., a PTAS or FPRAS) for real-world use with  hundreds of thousands of products.\looseness=-1

Thus motivated,  our goal of the present paper   is to make the sample average approximation a more viable approach to solving assortment optimization problems. We pursue this goal by focusing on the task of \emph{algorithm design}. On one hand, the combination of the sample average approximation and Monte-Carlo simulation can, at least in principle, solve any assortment optimization problem to any desired level of accuracy.
     On the other hand, the viability of this combination  hinges on being able to solve the sample average approximation  in practical computation times when the {number of samples and the number of products are large}.  This is problematic because  solving the sample average approximation is equivalent to solving assortment optimization {under the ranking-based choice model}, a  class of problems that is known to be computationally demanding. The need for \emph{practically efficient} algorithms for solving assortment optimization  under the ranking-based choice model  
has thus been the key roadblock to deploying the sample average approximation to solve assortment optimization problems in practice.\looseness=-1 

Our main contribution of this paper contends with that challenge by developing faster algorithms for solving assortment optimization under the ranking-based choice model. We do this in two ways. First, we introduce a new mixed-integer programming formulation of these problems, the relative strength  of which improves as  the number of samples increases. Second, we develop an accelerated Benders decomposition method based on novel structural insights and algorithmic results related to  Pareto-optimal cuts in the context of assortment optimization under the ranking-based choice model. Theoretically, we  prove that our solution methods are strict improvements over the state-of-the-art solution methods from the literature, with the improvements being most significant when the numbers of products and samples are large. Empirically, we show across  numerical experiments with synthetic and real-world data that our  solution methods can lead to substantial speedups in solve times for the sample average approximation, thereby offering the state-of-the-art solution methods for several classes of assortment optimization problems from the literature. Our developments thus make the sample average approximation a more viable approach to solving  assortment optimization problems  for which no practically efficient distribution-specific algorithms are known.\looseness=-1

\subsection{Related Literature} \label{sec:lit}
The study of assortment optimization under the ranking-based choice model has a rich history.\footnote{The equivalence of random utility models and distributions over rankings, as well as the construction of rankings by sampling, was established by \citet[Sections III and VII]{blockmarschak}.} 
As best as we can tell,  \cite{mcbride1988integer} were the first to formalize this class of assortment optimization  problems and present methods based on mixed-integer programming for solving them.  The design of solution methods based on mixed-integer programming  continued in the ensuing decades, with alternative formulations for this class of problems developed by \cite{belloni2008optimizing}, \cite{bertsimas2019exact}, and \cite{ma2023assortment}, among many others. The theoretical intractability and approximability  of these mixed-integer programs were established by \cite{aouad2019approximation} and \cite{feldman2019assortment}, and special cases of these mixed-integer programs that can be solved in polynomial time have been discovered by  
\cite{honhon2012optimal}, \cite{aouad2021assortment}, and \cite{ma2023assortment}. A two-phase Benders decomposition method  was proposed by \cite{bertsimas2019exact}, which was  rediscovered by \cite{zhang2024approximate} and extended to more general settings by \cite{akchen2021assortment}.  Heuristics  based on local search have been developed and studied by  \cite{jagabathula2014assortment} and \cite{gallego2024efficient}, and  
algorithms for solving robust versions of these assortment optimization problems were developed by \cite{farias2013nonparametric} and \cite{sturt2025value}. 

Up to this point,  the state-of-the-art solution methods for  assortment optimization under ranking-based choice models are those developed by \cite{bertsimas2019exact}. Their solution methods are two-fold.  First, they developed a mixed-integer programming formulation that is more compact and provably stronger (i.e., has a tighter linear programming relaxation) than the formulations of \cite{mcbride1988integer} and \cite{belloni2008optimizing}. Second, \cite{bertsimas2019exact} developed a two-phase Benders decomposition method for solving their mixed-integer program. The first phase of their decomposition method consists of solving the linear programming relaxation  of their mixed-integer programming formulation using constraint generation. The  second phase of their decomposition method consists of solving the mixed-integer programming formulation using constraint generation, which is warm-started with the cuts  generated in the first phase. At the crux of their two-phase Benders decomposition method  are custom algorithms for computing cuts, including a $\mathcal{O}(KN^2)$ algorithm for generating cuts in the first phase and an $\mathcal{O}(KN)$ algorithm for generating cuts in the second phase,  
where $N$ is the number of products and $K$ is the number of rankings.\looseness=-1   

Despite significant efforts in subsequent years, there has been limited progress on developing  solution methods that are more practically efficient than those of \cite{bertsimas2019exact} in the context of problems with large numbers of products and rankings.  For example,  \cite{ma2023assortment} introduced a new class of valid inequalities to tighten the mixed-integer programming formulation from \cite{bertsimas2019exact}; 
however, Ma found that adding those constraints does not lead to faster solution times, because the resulting formulation has ``\emph{$\ldots$ quadratically many constraints in $K$  instead of linearly many, and hence the relaxation itself can be slower to solve if $K$ is large}" \citep[Page 2103]{ma2023assortment}. This issue is particularly acute because the approximation gap of the sample average approximation decreases as the number of samples---and thus the number of rankings---becomes large.   Our main contribution of this paper contends with that fundamental challenge by developing solution methods that scale more gracefully with the number of products and rankings. 
\looseness=-1 

\subsection{Contributions} \label{sec:contributions}
In this paper,  we make the  sample average approximation a more viable approach to solving assortment optimization problems. We do this  by developing the first solution methods for the sample average approximation that are theoretical improvements over the  state-of-the-art  methods by \cite{bertsimas2019exact}. Our developments include a new mixed-integer programming formulation as well as a novel algorithm for computing cuts in Benders decomposition.  Empirically, we show that these developments lead to significant speedups in problems with many products and/or samples in real-world settings from the literature. Our theoretical and empirical results thus expand, as well as  offer a more nuanced understanding of, the viability of using the sample average approximation to solve assortment optimization problems for which no practically efficient distribution-specific algorithms are known.

In greater detail, the  main contributions of this paper are as follows.\looseness=-1
\begin{enumerate}

    \item We propose a new mixed-integer programming formulation for assortment optimization under the ranking-based choice model (Section~\ref{sec:exclusion_set}). Our new formulation, which we dub the ``{exclusion set formulation}'', is valuable because it is both  \emph{stronger} and \emph{more compact} than the formulation of \cite{bertsimas2019exact}. In other words, our formulation has a tighter linear programming relaxation, and we show that our formulation has fewer decision variables and constraints, both for finite numbers of samples as well as  asymptotically  as the number of samples tends to infinity. As such, our new formulation can be solved in faster computation times as the number of samples grows large.   At the crux of our new formulation is an aggregation of decision variables and constraints from the formulation of \cite{bertsimas2019exact} that is made possible by adding to their formulation a particular subset of the valid inequalities proposed by \cite{ma2023assortment}.\looseness=-1 
    \vspace{0.5em}

\item Our second development is an acceleration of the two-phase Benders decomposition method proposed by \cite{bertsimas2019exact} (Section~\ref{sec:main_alg}). In general,  the practical efficiency of any Benders decomposition method depends on the computation time for generating cuts and on the strength of the generated cuts. In Section~\ref{sec:results:toubia}, we  show that those two factors---computation time to generate cuts and strength of those  cuts---were the primary contributors to the computation time of the two-phase Benders decomposition method in the numerical experiments studied by \cite{bertsimas2019exact}. Motivated by these empirical findings, our paper develops an acceleration of  the two-phase Benders decomposition method by  developing  faster algorithms for computing optimal cuts, and, in particular, computing optimal cuts that are the \emph{strongest possible} from a theoretical standpoint. Our developments are as follows.\looseness=-1 
    \vspace{0.5em}
    
\begin{enumerate}
    \item We show that solving the dual of the inner problem of \cite{bertsimas2019exact} can be transformed into solving a separable convex minimization problem over chain constraints, which is a generalization of the isotonic regression problem from statistics (Section~\ref{sec:main_alg:reform}). 
        \vspace{0.5em}
        
    \item We use this reformulation to obtain  faster algorithms for computing optimal cuts in the two-phase Benders decomposition method (Section~\ref{sec:main_alg:algs}). In comparison to the algorithms from \cite{bertsimas2019exact}, our algorithms reduce the computation time for computing optimal Benders cuts from $\mathcal{O}(K N^2)$ to $\mathcal{O}(\sum_{k=1}^K L_k \log L_k)$  in the first phase and from $\mathcal{O}(NK)$ to $\mathcal{O}(\sum_{k=1}^K L_k)$  in the second phase, where $L_k \le N+1$ is the position of the no-purchase option for each ranking $k$. These algorithms can thus lead to  reductions in computation time per iteration of the Benders decomposition method when the numbers of products and rankings are large or when $L_k \ll N$ for most rankings $k$. Our algorithm for the first phase is also  simple, consisting of a reduction  to an algorithm by \cite{ahuja2001fast}. 
        \vspace{0.5em}
        
    \item We present the first characterization of  \emph{Pareto-optimal cuts} for the two-phase Benders decomposition method  (Section~\ref{sec:characterization_pareto}).   Pareto cuts, which are the gold standard property for cuts in Benders decomposition,  are cuts for which there do not exist any other cuts that are strictly stronger; see  \cite{magnanti1981accelerating}.  More specifically, we introduce four properties of Benders cuts, and we prove that a Benders cut is a Pareto cut if and only if it satisfies those four properties simultaneously (Theorems~\ref{thm:transformation} and \ref{thm:sufficient} in Section~\ref{sec:characterization_pareto}). 
    \vspace{0.5em}
    
    \item Equipped with our characterization of Pareto cuts, we present an $\mathcal{O}(\sum_{k=1}^K L_k)$ algorithm for transforming optimal Benders cuts  into  Pareto-optimal cuts (Section~\ref{sec:transformation}). Our algorithm thus accelerates the Benders decomposition method by  reducing the number of cuts needed to obtain an optimal solution for the sample average approximation.\looseness=-1
\end{enumerate}
\end{enumerate}

We demonstrate empirically that our improved  solution methods can lead to {significant} speedups in representative assortment planning problems, thereby enabling the sample average approximation to be solvable in practical computation times in a much wider scope of problem settings. 

\begin{itemize}
    
\item First, we consider assortment optimization under the multinomial logit with rank cutoffs, a parametric joint probability distribution for the random utility vector   proposed by \cite{bai2024assortment}. This is a representative setting for assortment optimization because the multinomial logit with rank cutoffs captures the well-known phenomenon that customers often have small consideration sets, while at the same time this class of assortment optimization problems is computationally demanding. In numerical experiments across various settings, our new algorithms lead to significant speedups over existing solution methods, where the speedups of our new algorithms is most pronounced when the number of samples is large.  
This showcases the value of our improved algorithms in enabling the sample average approximation to solve assortment optimization problems for which no practically-efficient distribution-specific algorithms are known.\looseness=-1
    \vspace{0.5em}
    
\item Second, to investigate the performance of our algorithms in setting with huge numbers of products and samples, we  consider a widely-studied product line design problem by \cite{belloni2008optimizing}. The setting involves more than 3000 products and more than 300 samples which are obtained using conjoint analysis, and thus provides a realistic approximation of settings in which samples are drawn from a known probability distribution. The time required to find an assortment with a cardinality constraint of five products for this setting using the mixed-integer programming formulation of \cite{belloni2008optimizing} was reported to be one week.  \cite{bertsimas2019exact}, using their Benders decomposition method, reduced the computation time to six minutes.\footnote{The running time of six minutes for \cite{bertsimas2019exact},  shown in Section~\ref{sec:results:toubia}, reflects improvements in computer hardware as well as improvements in the Gurobi solver. }  Without parameter tuning or heuristics, our accelerated Benders decomposition method in this setting reduces the computation time down to around {18 seconds}. These results show the viability of the sample average approximation in realistic settings with thousands of products. 
\end{itemize}

The rest of the paper is organized as follows. In Section~\ref{sec:prelim}, we formalize the problem setting and review existing solution methods. In Section~\ref{sec:exclusion_set}, we present our new mixed-integer programming formulation. In Section~\ref{sec:main_alg}, we develop the accelerated Benders decomposition method. In Section~\ref{sec:experiments}, we present numerical experiments. In Section~\ref{sec:conclusion}, we conclude and discuss directions for future work. All omitted proofs are found in the appendices.

\section{Preliminaries}\label{sec:prelim}
In this section, we formalize the problem setting of assortment optimization under the random utility maximization model and the sample average approximation, as well as review the state-of-the-art solution methods for the sample average approximation. 
\subsection{Problem Setting} \label{sec:rum}
We consider assortment optimization under the {random utility maximization} model.  These assortment optimization problems can be represented as  stochastic programming problems of the form
\begin{equation} \label{prob:rum}
\begin{aligned}
&\underset{S  \in \mathscr{S}}{\textnormal{maximize}} && \Exp \left[ \sum_{i \in S} r_i \mathbb{I} \left \{ U_i = \max_{j \in S} U_j  \right \}\right]
\end{aligned}
\end{equation}
%\begin{equation*}
%\begin{aligned}
%&\underset{S  \in \mathscr{S}}{\textnormal{maximize}} && \sum_{i \in S} r_i \mathbb{P} \left( U_i = \max_{j \in S} U_j   \right) 
%\end{aligned}
%\end{equation*}
%\begin{align*}
%    \Prb \left (U_i = \max_{j \in S} U_j \right) = \frac{v_i}{\sum_{j \in S} v_j}
%\end{align*}
The above stochastic program is defined by a universe of products  denoted by $[N+1] \equiv \{1,\ldots,N+1\}$. We assume without loss of generality that product $N+1$ is the no-purchase option, and the revenues of the products are denoted by $r_1,\ldots,r_N > 0$ and $r_{N+1} = 0$.  A subset of products $S \subseteq [N+1]$ that satisfies $N+1 \in S$ is referred to as an assortment, and we denote the set of feasible assortments that the firm can offer to its customers by $\mathscr{S} \subseteq \{ S \subseteq \{1,\ldots,N+1\}: N+1 \in S\}$. 
\looseness=-1 

The random utility maximization model is captured in the stochastic program~\eqref{prob:rum} by the joint probability distribution of the random utility vector $(U_1,\ldots,U_{N+1}) \in \R^{N+1}$. The random utility vector has the interpretation  that each randomly arriving customer will purchase a product if and only if that product has the maximum utility among all products in the assortment. We do not make any structural assumptions on the joint probability distribution, in the sense that we assume that the utilities $U_1,\ldots,U_{N+1}$ are random variables that can be correlated in possibly complex ways. Our only assumption about the random utility vector is the usual requirement that the utilities are unique almost surely; see, e.g., \citet[Section III]{blockmarschak}.
\begin{assumption} \label{ass:rum} 
The distribution of $(U_1,\ldots,U_{N+1})$ satisfies $U_i \neq U_j$ for all $i \neq j$ almost surely.\looseness=-1
\end{assumption} 
\vspace{-0.5em}

A large portion of the assortment optimization problems that have been  studied in revenue management are special cases of \eqref{prob:rum} with different parametric families of joint probability distributions. For example, if the random utilities $U_1,\ldots,U_{N+1}$ are independent Gumbel random variables, then \eqref{prob:rum} is an assortment optimization problem under the multinomial logit (MNL) model. Other popular classes of assortment optimization problems that are special cases of  \eqref{prob:rum} include  the Markov-chain choice model \citep{blanchet2016markov}, the mixed MNL choice model \citep{rusmevichientong2014assortment}, the exponomial choice model~\citep{aouad2023},  the Mallows smoothed choice model~\citep{desir2021mallows}, and  neural network-based models~\citep{aouad2022representing}, among many others.

As discussed in Section~\ref{sec:intro}, the computational difficulty of the stochastic programming problem~\eqref{prob:rum} depends on  the joint probability distribution of the random utility vector as well as the constraints on the set of feasible assortments.  In special cases of joint probability distributions in the random utility maximization model---like those that correspond to the MNL model,  mixed MNL model,  and Markov chain choice model---it is possible to derive  compact, closed-form representations for the expectation  in the objective function of \eqref{prob:rum}. But even when its objective function has a compact representation,  \eqref{prob:rum} is often challenging to solve in the presence of simple  constraints on the set of feasible assortments. The challenge of solving \eqref{prob:rum} is further exacerbated by the fact that the expectation of \eqref{prob:rum} does not in general have a simple closed-form structure. Against this backdrop, our focus in this paper is on methods for obtaining near-optimal solutions to \eqref{prob:rum} that do not make any parametric assumptions on the joint probability distribution of the random utility vector.

\subsection{Sample Average Approximation} \label{sec:saa}

 This paper considers the combination of the sample average approximation and Monte-Carlo simulation for solving the assortment optimization problem~\eqref{prob:rum}.  The primitive of sample average approximation in the context of assortment optimization is a set of samples that are drawn independently from the joint probability distribution of the random utility vector. We denote these samples  by $$(U^{\tilde{k}}_1,\ldots,U^{\tilde{k}}_{N+1}) \text{ for } \tilde{k} \in \left\{1,\ldots,\tilde{K}\right\}$$ 
It is worthwhile to emphasize that this paper is \emph{not} focused on a setting in which the joint probability distribution is unknown and the samples come directly from historical data. Rather, we assume throughout this paper that we are in the typical setting of assortment planning, where the joint probability distribution of the random utility vector has been estimated by a firm (for instance, by fitting a possibly complex parametric choice model to historical data), and where our goal is to solve \eqref{prob:rum} under that estimated joint probability distribution. For this setting, we assume that the samples are generated from the estimated joint probability distribution using Monte-Carlo simulation.\looseness=-1

Given the set of samples,  we approximate \eqref{prob:rum}  by replacing the true joint probability
distribution with the empirical probability distribution constructed from the samples. This approximation, which is known as the sample average approximation, is denoted by the optimization problem  
\begin{equation} \label{prob:saa}
\begin{aligned}
&\underset{S  \in \mathscr{S}}{\textnormal{maximize}} && \frac{1}{\tilde{K}} \sum_{\tilde{k}=1}^{\tilde{K}} \sum_{i \in S} r_i \mathbb{I} \left \{ U_i^{\tilde{k}} = \max_{j \in S} U_j^{\tilde{k}}  \right \} 
\end{aligned}
\end{equation}
The optimal objective value of \eqref{prob:saa} serves as our estimate of the optimal objective value of \eqref{prob:rum}, and the assortment obtained by solving \eqref{prob:saa} constitutes our approximate solution for \eqref{prob:rum}. This algorithmic technique of using Monte-Carlo simulation and the sample average approximation to solve stochastic programming problems with known distributions was popularized by \cite{kleywegt2002sample}, and this algorithmic technique is widely considered to be a foundational tool in the stochastic programming literature.  For a modern introduction to using the sample average approximation and Monte-Carlo simulation to solve stochastic programs, see  \cite{kim2015guide} and \cite{shapiro2021lectures}.\looseness=-1

The statistical properties  of the sample average approximation for   stochastic discrete optimization  have been studied extensively over the past several decades, and these  general results     apply readily to the specific case of \eqref{prob:saa}. In particular, it follows from  existing results that the optimal objective value and set of optimal solutions for the sample average approximation~\eqref{prob:saa} 
will converge to those of the stochastic program~\eqref{prob:rum},  almost surely, as the number of samples $\tilde{K}$ tends to infinity \citep{kleywegt2002sample}. As such,  the sample average approximation~\eqref{prob:saa} can approximate  the stochastic program~\eqref{prob:rum} to any desired level of accuracy by choosing a sufficiently large number of samples $\tilde{K}$ to generate using Monte-Carlo simulation. While the number of samples needed to obtain a desired level of accuracy  is not known \emph{a priori}, we can empirically assess the approximation gap of \eqref{prob:saa} for any finite choice of $\tilde{K}$ by using an out-of-sample validation set,  which can also be obtained using Monte-Carlo simulation; see Section~\ref{sec:experiments} for details.

%As discussed previously, this paper is focused on the design of algorithms for solving the sample average approximation~\eqref{prob:saa}.  That said, it is worthwhile to conclude Section~\ref{sec:saa} with a couple   remarks about statistical properties of the sample average approximation. First, 
We conclude Section~\ref{sec:saa} by noting that there is an extensive literature on  convergence rates and asymptotics of the sample average approximation for stochastic discrete optimization, all of which immediately apply to \eqref{prob:saa}. For these and related  results,  we refer the interested reader to \cite{kleywegt2002sample} and \citet[Chapter 5]{shapiro2021lectures}. Because the scope of this paper is centered on algorithm design, and because the aforementioned guarantees apply to \eqref{prob:saa}, we do not focus in this paper on developing improved convergence rates of the optimal objective value of \eqref{prob:saa}, with the exception of discussions of future work in Section~\ref{sec:conclusion}.  
We also remark that our numerical experiments suggest that the number of samples needed to achieve a small approximation gap can be reasonably small in practice; see Section~\ref{sec:experiments} for details.\looseness=-1

\subsection{Assortment Optimization under the Ranking-Based Choice Model} \label{sec:intro:methods}
The state-of-the-art methods for solving the sample average approximation are based on reformulating \eqref{prob:saa} as an assortment optimization problem under the ranking-based choice model. In the present Section~\ref{sec:intro:methods}, we formally derive this reformulation.

The ranking-based choice model is a choice model defined  by a distribution over {preference rankings}. A ranking in our context refers to a 
 bijection of the form $\sigma: \{1,\ldots,N+1\} \to \{1,\ldots,N+1\}$, with the interpretation that a customer has preferences that correspond to ranking $\sigma$ if and only if  the customer prefers product $i$ to product $j$ for all $i,j \in \{1,\ldots,N+1\}$ that satisfy $\sigma(i) < \sigma(j)$. It follows from Assumption~\ref{ass:rum}  that each sample $(U^{\tilde{k}}_1,\ldots,U^{\tilde{k}}_{N+1})$ of the random utility vector corresponds to exactly one ranking, almost surely.\footnote{Specifically, the corresponding ranking  $\sigma$ is the ranking that sorts $(U^{\tilde{k}}_1,\ldots,U^{\tilde{k}}_{N+1})$  in descending order.} % see, e.g.,  \citet[Section III]{blockmarschak}.  
 Given a set of samples $(U^{\tilde{k}}_1,\ldots,U^{\tilde{k}}_{N+1})$ for $\tilde{k} \in \{1,\ldots,\tilde{K}\}$,  we denote the unique  rankings that correspond to those samples by $\sigma_1,\ldots,\sigma_K$ with $K \le \tilde{K}$.   For each $k \in \{1,\ldots,K\}$, we let the proportion of samples that map to  ranking $\sigma_k$  be denoted by
\begin{align*}
    \lambda_k \triangleq \frac{1}{\tilde{K}} \left| \left \{ \tilde{k} \in \left\{1,\ldots,\tilde{K} \right \}: \sigma_k(i) < \sigma_k(j)  \text{ for all } i,j \text{ that satisfy } U^{\tilde{k}}_i > U^{\tilde{k}}_j \right \}  \right| 
\end{align*}
It is easy to see that the rankings $\sigma_1,\ldots,\sigma_K$ and proportions $\lambda_1,\ldots,\lambda_K$ can be computed efficiently from the set of samples and that $\sum_{k=1}^K \lambda_k = 1$.  It thus follows from the above reasoning and from algebra that the sample average approximation~\eqref{prob:saa} can be reformulated as the following assortment optimization problem under the ranking-based choice model.\looseness=-1
\begin{equation} \label{prob:ranking}
\begin{aligned}
&\underset{S  \in \mathscr{S}}{\textnormal{maximize}} && \sum_{k=1}^{K} \lambda_k \sum_{i \in S} r_i \mathbb{I} \left \{ i = \argmin_{j \in S} \sigma_k(j) \right \}
\end{aligned}
\end{equation}

 For notational convenience,  we will let the number of products that are preferred by ranking $k \in [K] \equiv \{1,\ldots,K\}$ to the no-purchase option be denoted by $L_k \triangleq \sigma_k^{-1}(N+1) - 1$.  We will also make the following standing assumption. 
    \begin{assumption} \label{ass:L_k}
        $L_k \ge 1$ for all $k \in [K]$. 
    \end{assumption}
    This assumption is without loss of generality because if $\sigma_{k}^{-1}(N+1) = 1$, then the no-purchase option is the most preferred option by ranking $k$. This implies that ranking $k$ will never purchase a product from the assortment other than the no-purchase option, which in turn implies that that ranking can be removed from the assortment optimization problem~\eqref{prob:ranking}.

\subsection{Existing Solution Methods}
The state-of-the-art general solution methods for  assortment optimization  under the ranking-based choice model~\eqref{prob:ranking} are those of  \cite{bertsimas2019exact}.  Their first solution method is a mixed-integer programming formulation of \eqref{prob:ranking}, and their second solution method is a two-phase Benders decomposition method for solving \eqref{prob:ranking} to optimality. We review these two solution methods below.

\subsubsection{Mixed-integer programming formulation.} \label{sec:review:mip} The  mixed-integer programming formulation of  \eqref{prob:ranking}  from \citet[Section 3.2]{bertsimas2019exact}  requires the following notation. For each ranking $k \in [K]$, let $i_{k,1},\ldots,i_{k,N+1}$ denote the products  sorted in order of preference, in the sense that $
    \sigma_k(i_{k,1}) < \cdots < \sigma_k(i_{k,N+1})$. In other words, let $i_{k,1}$ be the most preferred product by ranking $k$, let $i_{k,2}$ be the second-most preferred product, and so on and so forth. To simplify our  exposition, we will assume for the remainder of this section that there are no constraints on the set of feasible assortments: that is, $\mathscr{S} = \{S \subseteq \{1,\ldots,N+1\}: N+1 \in S \}$. 
Given that assumption, we will let the set of binary vectors corresponding to feasible assortments be denoted by  $\mathcal{X} \triangleq \{\bx \in \{0,1\}^{N+1}: x_{N+1} = 1\}$.  Given this notation, the mixed-integer programming  formulation  of \eqref{prob:ranking} given by \citet[Section 3.2]{bertsimas2019exact} is equivalent to
\begin{subequations} \label{prob:misic}
\begin{align}
\underset{\bx \in \mathcal{X},\by \in \R^{K \times (N+1)}}{\textnormal{maximize}} \quad & \sum_{k=1}^K 
 \sum_{\ell=1}^{N+1} r_{i_{k,\ell}} y_{k,\ell}  \lambda_k \label{prob:misic:objective}\\
 \textnormal{subject to} \quad & \sum_{\ell=1}^{N+1} y_{k,\ell} = 1 && \forall k \in [K]\label{prob:misic:atmostone}\\
 &x_{i_{k,\ell}} \le \sum_{\ell'=1}^{\ell}  y_{k,\ell'}  && \forall k \in [K], \ell \in [N+1]\label{prob:misic:pushup}\\
&  y_{k,\ell}  \le x_{i_{k,\ell}} &&  \forall k \in [K], \ell \in [N+1]\label{prob:misic:pushdown}\\
&   y_{k,\ell}  \ge 0 &&  \forall k \in [K], \ell \in [N+1]\label{prob:misic:nonneg}
\end{align}
\end{subequations}

The above formulation can be interpreted as follows. The first type of decision variables, $x_i \in \{0,1\}$, specify whether product $i$ is included in the assortment. The second type of decision variables, $y_{k,\ell} \in \R$,  encodes whether ranking $k$ purchases its $\ell$th most preferred product, i.e., whether the  product in the assortment that is  purchased by ranking $k$ is product $i_{k,\ell}$. The constraints \eqref{prob:misic:atmostone} enforce that each ranking purchases exactly one product. The constraints \eqref{prob:misic:pushup} enforce that if product $i_{k,\ell}$ is in the assortment,  then ranking $k$ must purchase either product $i_{k,\ell}$ or purchase a product that is preferred to product $i_{k,\ell}$.  The constraints \eqref{prob:misic:pushdown} and \eqref{prob:misic:nonneg} enforce that ranking $k$ purchases product $i_{k,\ell}$ only if product $i_{k,\ell}$ is in the assortment.  All combined, the constraints of \eqref{prob:misic} enforce that there exists an optimal solution in which each decision variable $y_{k,\ell}$ is binary and satisfies $y_{k,\ell} = 1$ if and only if product $i_{k,\ell}$ is purchased by ranking $k$. 
\begin{remark}
The mixed-integer optimization formulation given by \citet[Section 3.2]{bertsimas2019exact} uses slightly different notation than \eqref{prob:misic}; a formal equivalence of their formulation and \eqref{prob:misic} is given in Appendix~\ref{appendix:equivalence}. 
\end{remark}
\begin{remark}We observe that any constraints on the set of feasible assortments can be incorporated into \eqref{prob:misic} by adding constraints on the binary decision variables $x_1,\ldots,x_N$.
\end{remark} 
\subsubsection{Benders decomposition method.} \label{sec:benders} The Benders decomposition method from \citet[Section 4]{bertsimas2019exact} finds an optimal solution for \eqref{prob:misic} by separating it into an outer problem, in which we optimize over  $\bx$, and  inner problems, in which we calculate the revenue received from each ranking based on the assortment corresponding to $\bx$. In greater detail, their Benders decomposition method is based on the observation that \eqref{prob:misic} can be represented equivalently as
\begin{equation} \label{prob:outer}
\begin{aligned}
\underset{\bx \in \mathcal{X}, \bq \in \R^K}{\textnormal{maximize}} \quad & \sum_{k=1}^K  \lambda_k q_k   \\
 \textnormal{subject to} \quad & q_k \le \rho_k(\bx) && \forall k \in [K]
 \end{aligned}
\end{equation}
where $\rho_k(\bx)$ is defined as 
\begin{equation} \label{prob:inner}
\begin{aligned}
\rho_k(\bx) \triangleq \quad &\underset{\by \in \R^{N+1}}{\textnormal{maximize}} & &  
 \sum_{\ell=1}^{N+1} r_{i_{k,\ell}} y_{\ell} \\
& \textnormal{subject to} & & \sum_{\ell=1}^{N+1} y_{\ell} = 1 \\
 &&&x_{i_{k,\ell}} \le \sum_{\ell'=1}^{\ell} y_{\ell'}  && \forall  \ell \in [N+1]\\
&&&  0 \le y_{\ell}  \le x_{i_{k,\ell}} &&  \forall \ell \in [N+1]
\end{aligned}
\end{equation}
We observe that the dual of the linear program~\eqref{prob:inner} is\looseness=-1
\begin{equation} \label{prob:dual}
\begin{aligned}
\rho_k(\bx) \triangleq \quad &\underset{\balpha, \bbeta \in \R^{N+1}, \gamma \in \R}{\textnormal{minimize}} & &  
\gamma + \sum_{\ell=1}^{N+1} \left( \alpha_\ell - \beta_\ell \right) x_{i_{k,\ell}}  \\
& \textnormal{subject to} & & \gamma + \alpha_\ell - \sum_{\ell' = \ell}^{N+1}  \beta_{\ell'} \ge r_{i_{k,\ell}}  && \forall \ell \in [N+1]\\
&&& \alpha_\ell,\beta_\ell \ge 0 && \forall \ell \in [N+1]
\end{aligned}
\end{equation}
and we let the set of feasible solutions for \eqref{prob:dual} be denoted by $\mathcal{D}_k$.

The solution method proposed in \citet[Section 4]{bertsimas2019exact} finds an optimal solution for \eqref{prob:outer} using a two-phase Benders decomposition method. Phase 1 of their method consists of applying Benders decomposition to the linear programming relaxation of \eqref{prob:outer}, i.e., the modification of \eqref{prob:outer} in which $\mathcal{X}$ is replaced with its convex hull $\mathcal{X}^c \triangleq \{\bx \in [0,1]^{N+1}: x_{N+1} = 1\}$.  The Benders cuts obtained from Phase 1 are then used to warm start Phase 2, in which Benders decomposition is applied to \eqref{prob:outer} with the integrality constraints on $\bx$. In both phases, cuts are obtained by solving the linear program~\eqref{prob:dual} for each ranking. \cite{bertsimas2019exact} develop an an $\mathcal{O}(N)$ algorithm for solving  \eqref{prob:dual} when $\bx \in \mathcal{X}$ and an $\mathcal{O}(N^2)$ algorithm for solving \eqref{prob:dual} when $\bx \in \mathcal{X}^c$.  A formal description of Phase 1 and Phase 2 from \citet[Section 4]{bertsimas2019exact}, which provably converges to an optimal solution for \eqref{prob:outer} after finitely many iterations, can be found in Appendix~\ref{appx:benders_review}.\looseness=-1

\section{The Exclusion Set Formulation}\label{sec:exclusion_set}
In this section, we present our new mixed-integer programming formulation of \eqref{prob:ranking}, which we dub the \emph{exclusion set formulation}. We prove that this new formulation   is  stronger  and more compact  than the mixed-integer programming formulation~\eqref{prob:misic}. By \emph{stronger}, we mean that the new formulation has a tighter linear programming relaxation than \eqref{prob:misic}. By \emph{more compact}, we mean that the new formulation has fewer decision variables and constraints than~\eqref{prob:misic}, both for finitely many products and samples as well as asymptotically as the number of samples tends to infinity.  For these reasons, the exclusion set formulation is viewed as a theoretical improvement over \eqref{prob:misic} and can lead to faster solve times using modern mixed-integer programming solvers (see Section~\ref{sec:experiments:MNL_Cutoff}).

\subsection{The New Formulation} \label{sec:derivation:xset}
In this section, we derive our new mixed-integer programming formulation of \eqref{prob:ranking}. The derivation of our new formulation is relatively simple, and consists of reducing the size of the formulation \eqref{prob:misic} from  \cite{bertsimas2019exact} by adding a particular subset of the valid inequalities proposed by \cite{ma2023assortment}.
Our formulation of \eqref{prob:ranking} is derived in three steps. 

\subsubsection{Step 1.}  \label{sec:exclusionset:step1} We begin our derivation  by  removing from \eqref{prob:misic} the decision variables $y_{k,\ell}$ for all products $i_{k,\ell}$ that are less preferred than the no-purchase option.  Indeed, consider any $\bx \in \mathcal{X}$ that is feasible for \eqref{prob:misic}. It follows from the definition of $\mathcal{X}$  that $x_{N+1} = 1$, and so it follows from the definition of $L_k$ that $x_{i_{k,L_k+1}} = 1$.\footnote{Recall from Section~\ref{sec:intro:methods} that $L_k$ is the number of products that are preferred by ranking $k$ to the no-purchase option, which implies that  $i_{k,L_k+1} =N+1$ for each ranking $k$.} Therefore, it follows from constraints~\eqref{prob:misic:atmostone}, \eqref{prob:misic:pushup}, and \eqref{prob:misic:nonneg} that every feasible solution for \eqref{prob:misic} satisfies $y_{k,L_k+2} = \cdots = y_{k,N+1} = 0$.  By removing those decision variables from \eqref{prob:misic} for each ranking $k$, and by removing the decision variable $x_{N+1}$, we conclude that \eqref{prob:misic} can be written equivalently as
\begin{equation}\label{prob:misic_no_extra}
\begin{aligned}
\underset{\bx \in \{0,1\}^N,\by}{\textnormal{maximize}} \quad & \sum_{k=1}^K 
 \sum_{\ell=1}^{L_k} r_{i_{k,\ell}} y_{k,\ell}  \lambda_k \\
 \textnormal{subject to} \quad & \sum_{\ell=1}^{L_k} y_{k,\ell} \le 1 && \forall k \in [K] \\
 &x_{i_{k,\ell}} \le \sum_{\ell'=1}^{\ell}  y_{k,\ell'}  && \forall k \in [K], \ell \in [L_k] \\
& 0 \le   y_{k,\ell}  \le x_{i_{k,\ell}} &&  \forall k \in [K], \ell \in [L_k]
\end{aligned}
\end{equation}

\subsubsection{Step 2.} \label{sec:exclusionset:step2} We next add a particular subset of the valid inequalities proposed by \cite{ma2023assortment} into the formulation \eqref{prob:misic_no_extra}. Specifically, consider the following valid inequalities:
\begin{align} 
&\sum_{\ell=1}^{L}  y_{k, \ell} \ge \sum_{\ell=1}^L  y_{k', \ell} \qquad \forall k,k' \in [K], L \in [L_k]: \{i_{k,1},\ldots,i_{k,L} \} = \{i_{k',1},\ldots,i_{k',L} \} \label{prob:equality_nonsimple}
\end{align}
The fact that the inequalities~\eqref{prob:equality_nonsimple} are valid for \eqref{prob:misic_no_extra} follows from \citet[Proposition 12]{ma2023assortment}.  The above constraints can be interpreted as requiring that the sum of the decision variables $y_{k,\ell}$ for the $L$ most preferred products by ranking $k$ must be equal to the sum of the decision variables $y_{k',\ell}$ for the $L$ most preferred products by ranking $k'$, whenever the sets of $L$ most preferred products for both rankings are equal. It follows from symmetry that the  inequalities in \eqref{prob:equality_nonsimple} can be replaced by equalities, as formalized by the following lemma, the proof of which is found in Appendix~\ref{appx:proof:lem:inequality_to_equality}. 
\begin{lemma} \label{lem:inequality_to_equality}
    The constraints \eqref{prob:equality_nonsimple} are equivalent to
    \begin{align}
        \sum_{\ell=1}^{L}  y_{k,\ell} = \sum_{\ell=1}^L  y_{k',\ell} \quad \forall k,k' \in [K], L \in [ L_k]: \{i_{k,1},\ldots,i_{k,L} \} = \{i_{k',1},\ldots,i_{k',L} \} \label{line:equality_simple_cuts}
    \end{align}
\end{lemma}
In view of Lemma~\ref{lem:inequality_to_equality}, we observe  that \eqref{prob:misic_no_extra} with the  valid inequalities from \eqref{prob:equality_nonsimple}   can be written as\looseness=-1
\begin{equation} \label{prob:misic_extra}
\begin{aligned}
&\underset{\substack{\textbf{x} \in \{0,1\}^N,\textbf{y}}}{\textnormal{maximize}} && \sum_{k=1}^K  \sum_{\ell=1}^{L_k}  r_{i_{k,\ell}} y_{k,\ell}  \lambda_k \\
&\textnormal{subject to}&&\begin{aligned}[t]
    &\sum_{\ell=1}^{L}  y_{k,\ell} = \sum_{\ell=1}^L  y_{k',\ell} && \forall k,k' \in [K], L \in [ L_k]: \{i_{k,1},\ldots,i_{k,L} \} = \{i_{k',1},\ldots,i_{k',L} \}\\
& \sum_{\ell=1}^{L_k} y_{k,\ell} \le 1 && \forall k \in [K]\\
 &x_{i_{k,\ell}} \le \sum_{\ell'=1}^{\ell}  y_{k,\ell'}  && \forall k \in [K], \ell \in [L_k]\\
& 0 \le   y_{k,\ell}  \le x_{i_{k,\ell}} &&  \forall k \in [K], \ell \in [L_k]
\end{aligned}
\end{aligned}
\end{equation}
We observe that \eqref{prob:misic_extra} is  at least as strong as  \eqref{prob:misic_no_extra}, since the constraints in \eqref{prob:misic_extra} are a superset of the constraints in \eqref{prob:misic_no_extra}. Moreover, the linear programming relaxation of \eqref{prob:misic_extra} can be strictly tighter than the linear programming relaxation of \eqref{prob:misic_no_extra}. In  Appendix~\ref{appx:example:tighter_relaxation}, we give a simple example in which the optimal objective value of the linear programming relaxation of \eqref{prob:misic_extra} is strictly less than that of \eqref{prob:misic_no_extra}.\looseness=-1

While \eqref{prob:misic_extra} is a stronger formulation than \eqref{prob:misic_no_extra}, the former has more constraints. Specifically, the number of constraints in \eqref{prob:misic_no_extra} grows linearly in the number of rankings, while the number of constraints in \eqref{prob:misic_extra} grows quadratically in the number of rankings. Our main result of Section~\ref{sec:exclusion_set}, which will be formalized in Step 3 as Theorem~\ref{Theorem:Formulation}, consists of  showing that \eqref{prob:misic_extra} can be reformulated as a mixed-integer linear program with fewer decision variables and constraints than \eqref{prob:misic_no_extra}. 

\subsubsection{Step 3.} Our final step consists of showing that \eqref{prob:misic_extra} can be reformulated as a mixed-integer linear program in which the numbers of decision variables and constraints grow \emph{sublinearly} in the number of rankings $K$.   Our  reformulation of \eqref{prob:misic_extra} is motivated by an observation that if two different rankings $k,k'$ have the same sets of $L$ most preferred products, $E = \{i_{k,1},\ldots,i_{k,L} \} = \{i_{k',1},\ldots,i_{k',L} \}$, then we can introduce a single aggregated decision variable $z_E = \sum_{\ell=1}^L y_{k,\ell} = \sum_{\ell=1}^L y_{k',\ell}$ to encode whether each of the two rankings purchase one of their $L$ most preferred products.   These aggregated decision variables are useful because they will enable  bookkeeping in a mixed-integer linear program formulation to be performed on \emph{subsets of products}  instead of \emph{permutations of products}. In what follows, we will refer to these subsets of products as \emph{exclusion sets}.

The derivation of our reformulation   requires the following additional  notation and terminology.  We will say that a subset of products $E \subseteq [N]$ is an {exclusion set} if and only if there exists $k \in [K]$ and $L \in \{0,\ldots,L_k\}$ for which $E$ is the set of the $L$ most preferred products for ranking $k$. We let the collection of all exclusion sets be denoted by\looseness=-1
\begin{align*}
    \mathscr{E} \triangleq \left \{ E \in 2^{[N]}: \exists k \in [K], L \in \{0,\ldots,L_k\} \textnormal{ such that } E = \{i_{k,1},\ldots,i_{k,L}  \} \right \}
\end{align*}
Given an exclusion set $E \in \mathscr{E}$ and a product $i \notin E$, we will say that product $i$ is a \emph{continuation} of  exclusion set $E$ if and only if there exists a ranking  for which $E$ is the set of the ranking's $L$ most preferred products and for which $i$ is the ranking's $L+1$th most preferred product. We let the collection of all exclusion set and continuation pairs be denoted by 
\begin{align*}
\mathscr{P} \triangleq \left \{ (E,i) \in \mathscr{E} \times [N]: \; \exists k \in [K], L \in \{0,\ldots,L_k-1\}\textnormal{ such that } E = \{i_{k,1},\ldots,i_{k,L} \} \text{ and } i = i_{k,L+1} \right \}
\end{align*}
Finally, let the probability of an exclusion set and continuation pair $(E,i) \in \mathscr{P}$ be defined as
\begin{align*}
\lambda_{E,i}  \triangleq \sum_{k=1}^K \mathbb{I} \left \{ 
      E = \{i_{k,1},\ldots,i_{k,|E|}\} \text{ and } i = i_{k,|E|+1}   \right \}  \lambda_k
\end{align*}

Equipped with the above notation, we now state our main mixed-integer  programming formulation of Section~\ref{sec:exclusion_set}, which we refer to as the `exclusion set formulation': 
\begin{equation} 
\label{prob:exclusionset}
\begin{aligned}
\underset{\textbf{x} \in \{0,1\}^N,\textbf{z}}{\textnormal{maximize}} \quad & \sum_{(E,i) \in \mathscr{P}} r_i \lambda_{E,i} \left(z_{E \cup \{i \}}  - z_E \right) \\
\textnormal{subject to}\quad 
&0 \le z_{E \cup \{i \}}  - z_E \le x_i  & \forall (E,i) \in \mathscr{P}\\
&x_i \le z_{E \cup \{i \}} & \forall (E,i) \in \mathscr{P} \\
&z_E \le 1 & \forall E \in \mathscr{E}\\
&z_\emptyset = 0 
\end{aligned}
\end{equation}
Our main theorem, the proof of which
 can be found in Appendix~\ref{appx:exclusion_set:proof},  is the following.  
\begin{theorem}\label{Theorem:Formulation}
The mixed-integer linear programs \eqref{prob:misic_extra}  and  \eqref{prob:exclusionset} are equivalent.
\end{theorem}
We conclude Section~\ref{sec:derivation:xset} by giving an interpretation of the exclusion set formulation~\eqref{prob:exclusionset}.  Indeed, it follows from the proof of Theorem~\ref{Theorem:Formulation} that every feasible solution for \eqref{prob:exclusionset} satisfies the equality $z_E = \max_{j \in E} x_j$ for all $E \in \mathscr{E}$. Therefore,  for each $(E,i) \in \mathscr{P}$, we have that $z_{E \cup \{i\}} - z_E \in \{0,1\}$ is equal to one if and only if $x_j = 0$ for all $j \in E$ and $x_i = 1$. Our terminology of `exclusion set' comes from the interpretation  that if $E$ is the unordered set of a  ranking's $|E|$ most preferred products, and if $i$ is that ranking's $|E|+1$th most preferred product, then that ranking can purchase product $i$ only if all products $j \in E$ are excluded from the assortment.

\subsection{Comparison of Two Formulations} \label{sec:exclusion_set:comparison}
 We conclude Section~\ref{sec:exclusion_set} by discussing the settings in which the exclusion set formulation~\eqref{prob:exclusionset} offers practical value over the original mixed-integer programming formulation \eqref{prob:misic_no_extra}. 

 The primary advantages of \eqref{prob:exclusionset} compared to \eqref{prob:misic_no_extra} are two-fold. First, we observe that \eqref{prob:exclusionset} is a stronger formulation than \eqref{prob:misic_no_extra}. Indeed,  Theorem~\ref{Theorem:Formulation} shows that \eqref{prob:exclusionset} is equivalent to \eqref{prob:misic_extra}, and we showed in Section~\ref{sec:exclusionset:step2} that \eqref{prob:misic_extra} is stronger than \eqref{prob:misic_no_extra}.  Second, we observe that \eqref{prob:exclusionset} has fewer decision variables and constraints than \eqref{prob:misic_no_extra}, both for finite as well as asymptotic number of rankings. Indeed, for any finite number of rankings $K$, we observe that \eqref{prob:misic_no_extra} has approximately %$\sum_{k=1}^K L_k$ decision variables and constraints
 $\sum_{k=1}^K L_k$ decision variables and $3 \sum_{k=1}^K (L_k+1)$ constraints. In contrast, the exclusion set formulation~\eqref{prob:exclusionset} has approximately $| \mathscr{E}|$ decision variables and $3| \mathscr{P}| + | \mathscr{E}|$ constraints, where $| \mathscr{E}|-1 \le | \mathscr{P}| \le \sum_{k=1}^K L_k$.  Moreover, as the number of samples in the sample average approximation~\eqref{prob:saa} grows to infinity, we observe that the exclusion set formulation~\eqref{prob:exclusionset} will converge to at most $\mathcal{O}(2^N)$ decision variables and  $\mathcal{O}(N 2^{N})$ constraints (since the number of subsets of products  is at most $2^N$),   whereas \eqref{prob:misic_no_extra} will have at most $\mathcal{O}(N!)$ decision variables and constraints (since there are $N!$ possible rankings). 

The advantages of the exclusion set formulation can be expected to be most significant in settings where the rankings that correspond to the samples in \eqref{prob:saa} have significant numbers of what we will henceforth refer to as \emph{collisions}. Specifically, we say that a set of rankings $\mathcal{K} \subseteq [K]$ collide at an exclusion set $E \in \mathscr{E}$ if $\{i_{k,1},\ldots,i_{k,|E|} \} = E$ for all rankings $k \in \mathcal{K}$. The notion of collisions is useful because a large number of collisions translates to a smaller number of decision variables and constraints in \eqref{prob:exclusionset} as well as more of the valid equalities \eqref{line:equality_simple_cuts} that are included in formulation~\eqref{prob:misic_extra}. 

When considering the sample average approximation~\eqref{prob:saa}, there are various situations in which a large number of collisions can be expected. One situation is that in which the number of samples $\tilde{K}$ is very large relative to the number of products $N$. 
Another situation is where the most preferred products are often similar across many samples. A third situation is where the joint probability distribution in \eqref{prob:rum} typically produces small consideration sets, meaning that the rankings corresponding to the samples frequently have $L_k \ll N$. These situations in which the numbers of collisions are expected to be large are studied empirically in Section~\ref{sec:experiments:mnl:xset}.\looseness=-1

We conclude by noting two  structural difference between \eqref{prob:exclusionset} and \eqref{prob:misic_no_extra}. First, in contrast to \eqref{prob:misic_no_extra}, the revenue calculations for each ranking in \eqref{prob:exclusionset} are not decomposable, because \eqref{prob:exclusionset} inherits the linking constraints in \eqref{prob:misic_extra} across rankings. Because of this, it is not straightforward to develop a Benders decomposition method for \eqref{prob:exclusionset} in which a cut can be computed separately  for each ranking. For this reason, our accelerated Benders decomposition method in Section~\ref{sec:main_alg} focuses on the original mixed-integer programming formulation. Second, we note that the advantages of \eqref{prob:exclusionset} with respect to formulation strength appear to be most significant in assortment optimization problems where there either is no cardinality constraint, or where the cardinality constraint is not particularly small. This observation is formalized by the following proposition, the proof of which is found in Appendix~\ref{appx:proof:prop:integrality_budget_one}, which shows that the relative strength of the exclusion set formulation~\eqref{prob:exclusionset} compared to \eqref{prob:misic_no_extra} disappears for assortment optimization problems that have a cardinality constraint of a single product.\looseness=-1  
\begin{proposition} \label{prop:integrality_budget_one}
    If the set of feasible assortments is $\mathscr{S} = \{S  \subseteq \{1,\ldots,N+1\}: N + 1 \in S \text{ and }|S| \le 2\}$, then the linear programming relaxations of \eqref{prob:exclusionset} and \eqref{prob:misic_no_extra} are equivalent and are both integral. 
\end{proposition}

\section{The Accelerated Benders Decomposition} \label{sec:main_alg}
% \subsection{Overview and Main Results} \label{sec:main_alg:overview}

At a high level, the practical efficiency of the Benders decomposition method in Section~\ref{sec:benders} (see also Appendix~\ref{appx:benders_review}) depends on two  factors. The first factor is the computation time for generating optimal cuts, i.e., the time required to find an optimal solution for the linear program~\eqref{prob:dual} for each ranking in each iteration of Phase 1 and Phase 2. The second factor is  the {strength} of the optimal cuts that are generated in each iteration.   Considering the strength of cuts will be important in our context because, as we will show throughout Section~\ref{sec:main_alg},  the linear program~\eqref{prob:dual} will often have \emph{many} optimal solutions, and the strength of the cuts corresponding to those optimal solutions  can vary widely. Stronger cuts improve the practical efficiency of Benders decomposition because they can lead to a fewer total number of cuts for the Benders decomposition method to converge to an optimal solution.\looseness=-1

We focus in particular in Section~\ref{sec:main_alg} on the gold-standard class of cuts in the Benders decomposition literature, which are known as  \emph{Pareto cuts}  \citep[Section 2]{magnanti1981accelerating}. Pareto cuts are attractive because they are the {strongest possible} cuts, in the sense that a cut is a Pareto cut if and only if there do not exist any other  cuts that are strictly stronger. Recall from Section~\ref{sec:benders} that $\mathcal{D}_k$ denotes the set of feasible solutions for the linear program~\eqref{prob:dual}. The definition of Pareto cuts in  our context of \eqref{prob:dual} from Section~\ref{sec:benders} is as follows. 
\begin{definition} \label{defn:pareto}We say that   $(\balpha,\bbeta,\gamma) \in \mathcal{D}_k$ dominates  $(\balpha',\bbeta',\gamma') \in \mathcal{D}_k$ if 
\begin{align*}
 \gamma + \sum_{\ell=1}^{N+1} \left(\alpha_\ell - \beta_\ell \right) x_{i_{k,\ell}} &\le   \gamma' + \sum_{\ell=1}^{N+1} \left(\alpha_\ell' - \beta_\ell' \right) x_{i_{k,\ell}}  \text{ for all } \bx \in \mathcal{X}^c
\end{align*}
with the above inequality being strict at some $\bar{\bx} \in \mathcal{X}^c$.\footnote{It follows from linearity that each instance of $\mathcal{X}^c$ in Definition~\ref{defn:pareto} can without loss of generality be replaced by $\mathcal{X}$.}  We say that $(\balpha,\bbeta,\gamma) \in \mathcal{D}_k$ is a Pareto cut if there does not exist $(\balpha',\bbeta',\gamma') \in \mathcal{D}_k$ that dominates it. 
\end{definition}

In view of the above, our main contribution  of Section~\ref{sec:main_alg} are new algorithms for computing optimal solutions for the linear program~\eqref{prob:dual}.   Our algorithms are attractive for two reasons. The first reason is  their {computational efficiency}.  Specifically, our algorithms compute an optimal solution for \eqref{prob:dual} in $\mathcal{O}(L_k \log L_k)$ time for any $\bx \in \mathcal{X}^c$ and in $\mathcal{O}(L_k)$ time for any $\bx \in \mathcal{X}$, where we recall from Section~\ref{sec:intro:methods} that $L_k \le N$ is the number of products that are preferred to the no-purchase option in ranking $k$. Our algorithms are thus significantly faster than those from \cite{bertsimas2019exact} if the number of products is large or if the no-purchase option satisfies $L_k \ll N$.  The second attractive aspect of our algorithms is that they output optimal solutions to \eqref{prob:dual} that   are guaranteed to be {Pareto cuts} in the sense of Definition~\ref{defn:pareto}.  This  property is not satisfied by any previous algorithms  from the literature, and we show in Section~\ref{sec:experiments} that the Pareto optimality of our cuts leads to significant reductions in the total number of cuts  for Benders decomposition to find optimal solutions in Phase 1 and Phase 2.\looseness=-1

The rest of Section~\ref{sec:main_alg} is organized as follows. In Section~\ref{sec:main_alg:reform}, we   derive a new reformulation of the linear program~\eqref{prob:dual}. In Section~\ref{sec:main_alg:algs}, we use that new reformulation to develop faster algorithms for computing optimal solutions for \eqref{prob:dual}. In Section~\ref{sec:characterization_pareto}, we  use the reformulation from Section~\ref{sec:main_alg:reform} to characterize the structure of Pareto cuts. In Section~\ref{sec:transformation}, we exploit the structure of Pareto cuts to develop an efficient  algorithm for obtaining optimal solutions for \eqref{prob:dual} that are Pareto cuts.\looseness=-1
 \subsection{A New Reformulation}  \label{sec:main_alg:reform}
Our algorithms in Section~\ref{sec:main_alg} are based on a novel reformulation of \eqref{prob:dual}. Specifically, we show through a transformation of decision variables that \eqref{prob:dual} can be represented as a {separable convex optimization problem under chain constraints}. By utilizing this reformulation of \eqref{prob:dual}, we will derive a number of structural and algorithmic results for the linear program~\eqref{prob:dual}.

Our novel reformulation of \eqref{prob:dual} requires the following notation. Recall that $L_k =\sigma_k^{-1}(N+1) - 1$ denotes the number of products that are preferred by ranking $k$ to the no-purchase option, and recall from Section~\ref{sec:review:mip} that $i_{k,1},\ldots,i_{k,N+1}$ denote the products sorted by preference for ranking $k$.    Let the unique revenues among the products that are not less preferred than the no-purchase option for ranking $k$ be denoted by $$\mathcal{R}_k \triangleq \{r: \exists \ell \in \{1,\ldots,L_k+1\} \text{ such that } r = r_{i_{k,\ell}} \}$$ 
Let the highest revenue among the products that are not less preferred than the no-purchase option for ranking $k$ be denoted by $$\bar{r}_k \triangleq \max \{r \in \mathcal{R}_k\}$$
It follows from $r_1,\ldots,r_N > 0$ and  Assumption~\ref{ass:L_k} that the inequalities $|\mathcal{R}_k| \ge 2$ and $\bar{r}_k > 0$ always hold.\looseness=-1

Equipped with the above notation, we now introduce our main optimization problem of Section~\ref{sec:main_alg}:
    \begin{equation} \label{prob:dual_reform}
\begin{aligned}
&\underset{\bdelta \in \R^{L_k+1}}{\textnormal{minimize}} \quad &&J_k(\bx,\bdelta) \triangleq \delta_{L_k+1}+ \sum_{\ell=1}^{L_k}  \left( \max \left \{0, r_{i_{k,\ell}} - \delta_\ell \right \} - \left( \delta_{\ell+1} - \delta_\ell \right) \right) x_{i_{k,\ell}} \\
&\textnormal{subject to} \quad && 0 \le \delta_1 \le \cdots \le \delta_{L_k+1} \le \bar{r}_k
\end{aligned}
    \end{equation}
Let the set of feasible solutions for \eqref{prob:dual_reform} be denoted by $\Delta_k$, and note  that $J_k(\bx,\bdelta)$ refers to the objective function of the above optimization problem.  We readily observe that $J_k(\bx,\cdot)$ is a convex function for any fixed $\bx$, and it is easy to show (see  Section~\ref{sec:alg:phase1}) that \eqref{prob:dual_reform} is an instance of a particular class of optimization problems known as \emph{separable convex optimization under chain constraints}.\looseness=-1

Our main result of Section~\ref{sec:main_alg:reform}, which is stated below as  Theorem~\ref{thm:reform}, establishes the equivalence of  \eqref{prob:dual_reform} and \eqref{prob:dual}.   This  result will make use of the following analogous definition of Pareto cuts for \eqref{prob:dual_reform}.\looseness=-1
\begin{definition} \label{defn:pareto_reform} We say that $\bdelta \in \Delta_k$ dominates $\bdelta' \in \Delta_k$ if 
\begin{align*}
    J_k(\bx,\bdelta) \le J_k(\bx,\bdelta') \text{ for all }\bx \in \mathcal{X}^c
\end{align*}
where the above inequality is strict at some $\bar{\bx} \in \mathcal{X}^c$.\footnote{It follows from linearity that each instance of $\mathcal{X}^c$ in Definition~\ref{defn:pareto_reform} can without loss of generality be replaced by $\mathcal{X}$.} We say that $\bdelta \in \Delta_k$ is a Pareto cut if there does not exist $\bdelta' \in \Delta_k$ that dominates it. 
\end{definition}
\vspace{-1em}
In view of the above definition, our main result of Section~\ref{sec:main_alg:reform} is the following. 
     \begin{theorem} \label{thm:reform}
        If $\bdelta \in \Delta_k$ is an optimal solution for \eqref{prob:dual_reform}, then there exists $(\balpha,\bbeta,\gamma) \in \mathcal{D}_k$ that is an optimal solution for \eqref{prob:dual} and satisfies
\begin{align*}
              J_k(\bx,\bdelta) = \gamma + \sum_{\ell=1}^{N+1} \left( \alpha_\ell - \beta_\ell \right) x_{i_{k,\ell}}  \quad \forall \bx \in \mathcal{X}^c
          \end{align*}
         Moreover, if $\bdelta$ is a Pareto cut in the sense of Definition~\ref{defn:pareto_reform}, then $(\balpha,\bbeta,\gamma)$ is a Pareto cut in the sense of Definition~\ref{defn:pareto}.\looseness=-1
    \end{theorem}
    
The above theorem can be interpreted as follows.  First, Theorem~\ref{thm:reform} establishes the equivalence of \eqref{prob:dual} and \eqref{prob:dual_reform}; that is, it implies that the optimal objective values for the two problems are equal and that every optimal solution for \eqref{prob:dual_reform} yields an optimal solution for \eqref{prob:dual}. Second, Theorem~\ref{thm:reform} says that if we can find an optimal solution for \eqref{prob:dual_reform} that is a Pareto cut in the sense of Definition~\ref{defn:pareto_reform}, then we have found an optimal solution for \eqref{prob:dual} that is a Pareto cut in the sense of Definition~\ref{defn:pareto}. 
For these reasons, Theorem~\ref{thm:reform}  implies that we can reduce the problem of finding an optimal solution for \eqref{prob:dual} that is a Pareto cut to the problem of finding an optimal solution for \eqref{prob:dual_reform} that is a Pareto cut.

Our proof of Theorem~\ref{thm:reform} is found in Appendix~\ref{appx:proof_reformulation}. The difficulty of proving Theorem~\ref{thm:reform}  stems from the lower bound $0 \le \delta_1$ and the upper bound $\delta_{L_k+1} \le \bar{r}_k$ on the decision variables in \eqref{prob:dual_reform}. These bounds play a critical role in Sections~\ref{sec:characterization_pareto} and \ref{sec:transformation} by enabling a succinct characterization of feasible solutions  that are Pareto cuts. However, these bounds require the use of a piecewise linear mapping from the decision variables of \eqref{prob:dual} to the decision variables of \eqref{prob:dual_reform}, which in turn leads to numerous cases to analyze when comparing the objective functions of \eqref{prob:dual} and \eqref{prob:dual_reform}. Our approach that circumvents a tedious case-by-case analysis  can be found in Proposition~\ref{prop:reform_hard:obj} in Appendix~\ref{appx:reform_hard}.

 \subsection{Fast Algorithms for Computing Optimal Cuts} \label{sec:main_alg:algs}
Equipped with the reformulation~\eqref{prob:dual_reform} from Section~\ref{sec:main_alg:reform}, we now present two efficient algorithms for computing an optimal solution for \eqref{prob:dual_reform}. The first  is an $\mathcal{O}(L_k \log L_k)$ time algorithm for computing optimal solution for \eqref{prob:dual_reform} that can be applied to any $\bx \in \mathcal{X}^c$. The second is a $\mathcal{O}(L_k)$ time algorithm for computing an optimal solution for \eqref{prob:dual_reform} that can be applied to any $\bx \in \mathcal{X}$.  The algorithms can thus be applied in Phase 1 and Phase 2  of the Benders decomposition method from Section~\ref{sec:benders}.  

\subsubsection{Algorithm for Phase 1.} \label{sec:alg:phase1}
We begin with our algorithm  for solving \eqref{prob:dual_reform} for any $\bx \in \mathcal{X}^c$. It follows from algebra that the objective function of \eqref{prob:dual_reform}  can be  rewritten as 
\begin{align*}
   J_k(\bx,\bdelta) &=  \underbrace{\left( \max \left \{ 0, r_{i_{k,1}} - \delta_1 \right \} + \delta_1 \right) x_{i_{k,1}} }_{C_{k,1}(\bx,\delta_1)} \\
   &\quad +  \sum_{\ell=2}^{L_k} \underbrace{\left( \max \left \{ 0, r_{i_{k,\ell}} - \delta_\ell \right \} + \delta_\ell \right)x_{i_{k,\ell}} -  \delta_\ell x_{i_{k,\ell-1}}}_{C_{k,\ell}(\bx,\delta_\ell)}  + \underbrace{\left( 1 - x_{i_{k,L_k}} \right) \delta_{L_k+1}}_{C_{k,L_k+1} (\bx,\delta_{L_k+1})} %\label{line:J_k_breakdown}
\end{align*}
It follows from the above equality  that \eqref{prob:dual_reform} is equivalent to
\begin{equation*}
\begin{aligned}
&\underset{\bdelta}{\textnormal{minimize}} \quad && \sum_{\ell=1}^{L_k+1} C_{k,\ell}(\bx,\delta_\ell)\\
&\text{subject to} \quad && 0 \le  \delta_1 \le \cdots \le \delta_{L_k+1} \le \bar{r}_k
\end{aligned}
\end{equation*}
We observe for any fixed $\bx \in \mathcal{X}^c$ that the functions $C_{k,\ell}(\bx,\cdot)$ for $\ell \in [L_k+1]$ are piecewise linear convex functions, which makes \eqref{prob:dual_reform}  a separable convex optimization problem with chain constraints. 

This particular class of separable convex optimization problems under chain constraints has been extensively studied in statistics  as a generalization of the  \emph{isotonic regression} problem, i.e., the problem of finding a regression line for a dependent variable that is monotonically increasing in the independent variable. The most common algorithm for solving separable convex optimization problems under chain constraints is known as the pool adjacent violators (PAV) algorithm; see \cite{de2010isotone}. The idea of this algorithm is to first solve \eqref{prob:dual_reform} without any constraints, and then iteratively add in violated constraints by merging the convex functions. In our context, 
we observe for any fixed $\bx \in \mathcal{X}^c$ that each of the  convex functions  $C_{k,\ell}(\bx,\cdot)$ for $\ell \in [L_k+1]$ is piecewise linear and can be evaluated in $\mathcal{O}(1)$ time. As such, the implementation of the PAV algorithm given by \citet[Theorem 3]{ahuja2001fast} can be applied to solve \eqref{prob:dual_reform} in $\mathcal{O}(L_k \log L_k)$ time. Our algorithm for solving \eqref{prob:dual_reform} is thus the algorithm  of  \citet[Theorem 3]{ahuja2001fast}. 

\subsubsection{Algorithm for Phase 2.} For the case where $\bx \in \mathcal{X}$, we have a simple $\mathcal{O}(L_k)$ algorithm for solving  \eqref{prob:dual_reform}. Specifically, given $\bx \in \mathcal{X}$ and ranking $k \in [K]$, we first compute the most-preferred product by ranking $k$ that is in the assortment $\bx$; that is, we compute the index
\begin{align}
    \ell^* \leftarrow \min \{\ell \in [L_k+1]: x_{i_{k,\ell}} = 1 \} \label{line:phase2:ell_star}
\end{align}
It is clear that $\ell^*$ can be computed in $\mathcal{O}(L_k)$ time, and it follows from the fact that $\bx \in \mathcal{X}$ that $x_{i_{k,L_k+1}} = x_{N+1} = 1$, which implies that the index $\ell^*$ is  well defined. Given that index, we then construct a vector $\bdelta \in \R^{L_k+1}$ that is defined for each $\ell \in [L_k+1]$ as\looseness=-1
\begin{align}
    \delta_\ell \leftarrow \begin{cases}
        r_{i_{k,\ell^*}},&\text{if } \ell \in \{1,\ldots,\ell^*\},\\
        \bar{r}_k,&\text{if } \ell \in \{\ell^*+1,\ldots,L_k+1\}
    \end{cases} \label{line:phase2:delta_ell}
\end{align}
We observe from the above definition that $0 \le \delta_1 = \cdots = \delta_{\ell^*} \le \delta_{\ell^*+1} = \cdots =  \delta_{L_k+1} = \bar{r}_k$, which implies that $\bdelta$ is a feasible solution for \eqref{prob:dual_reform}. One can also clearly compute $\bdelta$ in $\mathcal{O}(L_k)$ time. The optimality of this solution is formalized by the following proposition, the proof of which  is found in Appendix~\ref{appx:proof:phase2:optimal}. 

\begin{proposition} \label{prop:phase2:optimal}
    Let $k \in [K]$ and $\bx \in \mathcal{X}$. If $\bdelta$ is defined by  \eqref{line:phase2:ell_star} and \eqref{line:phase2:delta_ell}, then $\bdelta$ is an optimal solution for \eqref{prob:dual_reform}. 
\end{proposition}

\subsection{Characterization of Pareto cuts} \label{sec:characterization_pareto}
Our algorithms from Section~\ref{sec:main_alg:algs} are guaranteed to yield  optimal solutions for \eqref{prob:dual_reform}, but those  optimal solutions are not guaranteed to be Pareto cuts in the sense of Definition~\ref{defn:pareto_reform}. To address this,  Sections~\ref{sec:characterization_pareto} and \ref{sec:transformation} develop an efficient algorithm for transforming optimal solutions for \eqref{prob:dual_reform} into optimal solutions for \eqref{prob:dual_reform} that are Pareto cuts.  The present  Section~\ref{sec:characterization_pareto} establishes an exact characterization of feasible solutions for \eqref{prob:dual_reform} that are Pareto cuts, and Section~\ref{sec:transformation} uses that characterization to develop an algorithm for transforming optimal solutions for \eqref{prob:dual_reform} into Pareto cuts. 

Our characterization of Pareto cuts for \eqref{prob:dual_reform} makes use of four properties, which are denoted below by Properties~\ref{property:notspecial}, \ref{property:forward}, \ref{property:gap}, and \ref{property:reverse}.  For each  $\bdelta \in \Delta_k$, let $T_k(\bdelta) \in \{1,\ldots,L_k+1\}$ denote the maximum integer $L$ that satisfies $\delta_L \le r_{i_{k,L}}$, that is,
\begin{align}
    T_k(\bdelta) \triangleq \max \left \{ L \in \{1,\ldots,L_k+1\}: \delta_L \le r_{i_{k,L}} \right \} \label{defn:T_k_delta}
\end{align}We observe that $T_k(\bdelta)$ is  well defined for all $\bdelta \in \Delta_k$.\footnote{We recall from the definition of \eqref{prob:dual_reform} that if $\bdelta \in \Delta_k$, then $\delta_1,\ldots,\delta_{L_k+1} \le \bar{r}_k$. We also recall from   Section~\ref{sec:main_alg:reform} that $\bar{r}_k = \max \{r \in \mathcal{R}_k \} > 0$, which together with the definition of $\mathcal{R}_k$ and the fact that $r_{i_{k,L_k+1}} = 0$ implies that there exists $\hat{L} \in \{1,\ldots,L_k \}$ that satisfies $r_{i_{k,\hat{L}}} = \bar{r}_k$. Therefore, it follows from \eqref{defn:T_k_delta} that  $\hat{L} \le T_k(\bdelta) \le L_k+1$ for all $\bdelta \in \Delta_k$. }  Equipped with that notation, we  now  state  four properties that may be satisfied by a feasible solution $\bdelta \in \Delta_k$. 
\begin{property} \label{property:notspecial}
$T_k(\bdelta) > 1$.
\end{property}
\begin{property}
\label{property:forward}
    $\delta_1 \le r_{i_{k,1}}$.
\end{property}
\begin{property} \label{property:gap}
For all $\ell \in \{2,\ldots,L_k\}$, we have that if $\delta_\ell < r_{i_{k,\ell}}$, then $\delta_\ell = \delta_{\ell+1}$.
\end{property}
\begin{property} \label{property:reverse}
$\delta_{T_k(\bdelta)} = \cdots = \delta_{L_k+1}$.
\end{property}
In Appendix~\ref{sec:interpretation:four}, we provide an interpretation of Properties~\ref{property:notspecial}, \ref{property:forward}, \ref{property:gap}, and \ref{property:reverse} by discussing  four examples. Each example in Appendix~\ref{sec:interpretation:four} shows how removing any one of those four properties allows for feasible solutions for \eqref{prob:dual_reform} that are not Pareto cuts in the sense of Definition~\ref{defn:pareto_reform}.

Our main result of Section~\ref{sec:characterization_pareto} shows that Properties~\ref{property:notspecial}, \ref{property:forward}, \ref{property:gap}, and \ref{property:reverse} provide an exact characterization of the structure of Pareto cuts. That is, we show that the combination of those four properties constitutes a \emph{necessary} as well as \emph{sufficient} condition for a feasible solution of \eqref{prob:dual_reform} to be a Pareto cut in the sense of Definition~\ref{defn:pareto_reform}. This main result is split into the following two theorems.\looseness=-1  
\begin{theorem}\label{thm:transformation}
If $\bdelta \in \Delta_k$ does not  satisfy  Properties~\ref{property:notspecial}, \ref{property:forward}, \ref{property:gap}, and \ref{property:reverse}, then $\bdelta$ is not a Pareto cut.  
\end{theorem}
\begin{theorem} \label{thm:sufficient}
    If $\bdelta \in \Delta_k$ satisfies Properties~\ref{property:notspecial}, \ref{property:forward}, \ref{property:gap}, and \ref{property:reverse}, then $\bdelta$ is a Pareto cut. 
\end{theorem}
Stated in words, Theorem~\ref{thm:transformation} shows that the combination of Properties~\ref{property:notspecial}, \ref{property:forward}, \ref{property:gap}, and \ref{property:reverse} is a necessary condition for a feasible solution for \eqref{prob:dual_reform} to be a Pareto cut, in the sense of Definition~\ref{defn:pareto_reform}, and Theorem~\ref{thm:sufficient} shows that condition is also sufficient. 
The proof of Theorem~\ref{thm:transformation} is constructive and found in Section~\ref{sec:transformation}. The  proof of  Theorem~\ref{thm:sufficient}, which is found in Appendix~\ref{appx:thm:sufficient}, consists of an intricate construction of  five mutually exclusive and collectively exhaustive cases that relate any two vectors $\bdelta,\bdelta' \in \Delta_k$ that satisfy Properties~\ref{property:notspecial}, \ref{property:forward}, \ref{property:gap}, and \ref{property:reverse}.

\subsection{The Transformation Algorithm} \label{sec:transformation}
In this section, we prove Theorem~\ref{thm:transformation}. We do this by developing an $\mathcal{O}(L_k)$ algorithm for  transforming any feasible solution for \eqref{prob:dual_reform} that does not simultaneously satisfy Properties~\ref{property:notspecial}, \ref{property:forward}, \ref{property:gap}, and \ref{property:reverse} into a feasible solution for \eqref{prob:dual_reform}  that dominates the original feasible solution and  satisfies Properties~\ref{property:notspecial}, \ref{property:forward}, \ref{property:gap}, and \ref{property:reverse}.  We will henceforth refer to this algorithm as the `transformation' algorithm. Because the feasible solution obtained by the transformation algorithm satisfies Properties~\ref{property:notspecial}, \ref{property:forward}, \ref{property:gap}, and \ref{property:reverse}, it follows from Theorem~\ref{thm:sufficient} that the output of the transformation algorithm is a Pareto cut.   Moreover, the fact that the transformation algorithm  runs in $\mathcal{O}(L_k)$ time implies that this transformation algorithm can be applied to optimal solutions obtained in Phase 1 and Phase 2 without increasing the computation time of the algorithms from Section~\ref{sec:main_alg:algs}. Our algorithms from Section~\ref{sec:main_alg} thus consists of using the algorithms from Section~\ref{sec:main_alg:algs} to find an optimal solution for \eqref{prob:dual_reform}, followed by using the transformation algorithm from Section~\ref{sec:transformation} to obtain an optimal solution for \eqref{prob:dual_reform} that is a Pareto cut.\looseness=-1%In other words, the algorithms from Section~\ref{sec:main_alg:algs} combined with the transformation algorithm from this section yield our $\mathcal{O}(L_k \log L_k)$ algorithm for computing an optimal solution that is a Pareto cut in Phase 1 for any $\bx \in \mathcal{X}^c$, and our $\mathcal{O}(L_k)$ algorithm for computing an optimal solution that is a Pareto cut for Phase 2 for any $\bx \in \mathcal{X}$. 

Our transformation algorithm consists of the following steps. The input to the transformation algorithm is a feasible solution $\bdelta \in \Delta_k$  for   \eqref{prob:dual_reform}. We then modify that feasible solution by applying four subroutines, which are presented as Subroutines 1, 2, 3, and 4 in Algorithm~\ref{alg:transformation_algorithm}. We apply the subroutines sequentially, in the sense the feasible solution  $\bdelta$ is the input to Subroutine 1, the output of  Subroutine 1 is the input into Subroutine 2, the output of  Subroutine 2 is the input into Subroutine 3, and the output of  Subroutine 3 is the input into Subroutine 4. The  output of Subroutine 4 is the output of the transformation algorithm. It is straightforward to show that each of the four subroutines from Algorithm~\ref{alg:transformation_algorithm} can be implemented in $\mathcal{O}(L_k)$ time.

The four examples in Appendix~\ref{sec:interpretation:four} provide an illustration of the four subroutines from Algorithm~\ref{alg:transformation_algorithm}. Specifically, Example~\ref{example:property1:violated} from  Appendix~\ref{sec:interpretation:four} shows an example of an input to Subroutine 1 and the corresponding output of that subroutine. Example~\ref{example:property2:violated} from  Appendix~\ref{sec:interpretation:four}   shows an example of an input to Subroutine 2 and the corresponding output. Examples~\ref{example:property3:violated} and \ref{example:property4:violated} from  Appendix~\ref{sec:interpretation:four}  show similar examples for Subroutines 3 and 4. These four examples thus illustrate  not only  the necessity of Properties~\ref{property:notspecial},  \ref{property:forward},  \ref{property:gap}, and \ref{property:reverse}, but also  the correctness of  Subroutines 1, 2, 3, and 4.\looseness=-1  

\afterpage{%
\null
\vfill
\begin{algorithm}[H]
\begin{center}
\fbox{%
    \begin{minipage}[t][0.55\textwidth][t]{0.47\textwidth}
      \begin{center}
          \underline{\textbf{Subroutine 1}}
      \end{center}
\vspace{1em}

\textbf{Input}:
\begin{itemize}
\item A vector $\bdelta \in \Delta_k$. 
\end{itemize}
\vspace{1em}

\textbf{Output}:
\begin{itemize}
\item A vector $\bdelta' \in \Delta_k$ that satisfies Property~\ref{property:notspecial}. Moreover, if $\bdelta' \neq \bdelta$, then $\bdelta'$  dominates $\bdelta$. 
\end{itemize}
\vspace{1em}

\textbf{Procedure}:
  \begin{enumerate}
\item Let  $\bar{r}_k' \leftarrow \max \{r \in \mathcal{R}_k: r < \bar{r}_k \}$ denote the second-highest revenue  in $\mathcal{R}_k$.
      \item If $T_k(\bdelta) > 1$, then let $\bdelta' \leftarrow \bdelta$ and return $\bdelta'$.
      \item \label{line:subroutine1:else} Else, define $\bdelta'$ as
      \begin{align*}
          \delta'_{\ell} \leftarrow \min \left \{ \delta_\ell, \bar{r}_k' \right \} \quad \forall \ell \in \{1,\ldots,L_k+1\}
      \end{align*}
      and return $\bdelta'$. 
\end{enumerate}
\end{minipage}}\;
\fbox{%
    \begin{minipage}[t][0.55\textwidth][t]{0.47\textwidth}
      \begin{center}
          \underline{\textbf{Subroutine 2}}
      \end{center}
\vspace{1em}

\textbf{Input}:
\begin{itemize}
\item A vector $\bdelta \in \Delta_k$ that satisfies Property~\ref{property:notspecial}. 
\end{itemize}
\vspace{1em}

\textbf{Output}:
\begin{itemize}
\item A vector $\bdelta' \in \Delta_k$ that satisfies  Properties~\ref{property:notspecial} and \ref{property:forward}. Moreover, if $\bdelta' \neq \bdelta$, then $\bdelta'$ dominates $\bdelta$. 
\end{itemize}
\vspace{1em}

\textbf{Procedure}:
  \begin{enumerate}
        \item Let $\bdelta' \leftarrow \bdelta$. 
        \item If $\delta'_1 > r_{i_{k,1}}$, then let  $\delta'_1 \leftarrow r_{i_{k,1}}$. 
        \item Return $\bdelta'$.
\end{enumerate}
\end{minipage}} 
\\
\vspace{0.75em}
\fbox{%
    \begin{minipage}[t][0.55\textwidth][t]{0.47\textwidth}
      \begin{center}
          \underline{\textbf{Subroutine 3}}
      \end{center}
\vspace{1em}

\textbf{Input}:
\begin{itemize}
\item A vector $\bdelta \in \Delta_k$ that satisfies  Properties~\ref{property:notspecial} and \ref{property:forward}. 
\end{itemize}
\vspace{1em}

\textbf{Output}:
\begin{itemize}
\item A vector $\bdelta' \in \Delta_k$ that satisfies  Properties~\ref{property:notspecial},  \ref{property:forward}, and \ref{property:gap}. Moreover, if $\bdelta' \neq \bdelta$, then $\bdelta'$  dominates $\bdelta$. 
\end{itemize}
\vspace{1em}

\textbf{Procedure}:
  \begin{enumerate}
  \item Let $\delta'_{L_k+1} \leftarrow \delta_{L_k+1}$
        \item For $\ell \leftarrow L_k,\ldots,2$, let
        \begin{align*}
            \delta'_\ell \leftarrow  \begin{cases}
   \delta_\ell,&\text{if } \delta_\ell \ge r_{i_{k,\ell}},\\
   \min \left \{ r_{i_{k,\ell}}, \delta'_{\ell+1} \right \},&\text{if } \delta_\ell < r_{i_{k,\ell}} 
   \end{cases}
        \end{align*}
\item Let $\delta'_1 \leftarrow \delta_1$. 
        \item Return $\bdelta'$.
\end{enumerate}
\end{minipage}}\;
\fbox{%
    \begin{minipage}[t][0.55\textwidth][t]{0.47\textwidth}
      \begin{center}
          \underline{\textbf{Subroutine 4}}
      \end{center}
\vspace{1em}

\textbf{Input}:
\begin{itemize}
\item A vector $\bdelta \in \Delta_k$ that satisfies  Properties~\ref{property:notspecial},  \ref{property:forward}, and \ref{property:gap}. 
\end{itemize}
\vspace{1em}

\textbf{Output}:
\begin{itemize}
\item A vector $\bdelta' \in \Delta_k$ that satisfies  Properties~\ref{property:notspecial},  \ref{property:forward}, \ref{property:gap}, and \ref{property:reverse}. Moreover, if $\bdelta' \neq \bdelta$, then $\bdelta'$  dominates $\bdelta$.  
\end{itemize}
\vspace{1em}

\textbf{Procedure}:
  \begin{enumerate}
  \item Let $L \leftarrow T_k(\bdelta)$.
  \item Let $\hat{r} \leftarrow \max \{ \delta_L, \max_{\ell \in \{L+1,\ldots,L_k+1\}} r_{i_{k,\ell}} \}$. 
        \item For $\ell \leftarrow 1,\ldots,L_k+1$, let  $\delta'_\ell \leftarrow \min \left \{ \delta_\ell, \hat{r} \right \}$.
        \item Return $\bdelta'$.
\end{enumerate}
\end{minipage}}
\end{center}
\vspace{1.5em}
\caption{The four subroutines constitute the transformation algorithm from Section~\ref{sec:transformation}. Note for Subroutine 1 that the fact that $|\mathcal{R}_k| \ge 2$ is established in Section~\ref{sec:main_alg:reform}. }\label{alg:transformation_algorithm}
\end{algorithm}
\null
\vfill
\clearpage
}%

We conclude Section~\ref{sec:transformation} by  proving the correctness of the four subroutines from Algorithm~\ref{alg:transformation_algorithm}. Specifically, the following Propositions~\ref{prop:subroutine:1}, \ref{prop:subroutine:2}, \ref{prop:subroutine:3}, and \ref{prop:subroutine:4} establish that Subroutines 1, 2, 3, and 4 are guaranteed to give their specified outputs. Moreover, the combination of the following propositions constitutes the proof of Theorem~\ref{thm:transformation}, since they show that any $\bdelta \in \Delta_k$ that does not simultaneously satisfy Properties~\ref{property:notspecial},  \ref{property:forward},  \ref{property:gap}, and \ref{property:reverse} can, by using Subroutines 1, 2, 3, and 4, be transformed into a feasible solution $\bdelta' \in \Delta_k$ that satisfies Properties~\ref{property:notspecial},  \ref{property:forward},  \ref{property:gap}, and \ref{property:reverse} and dominates $\bdelta$. The proofs of the following propositions are found in Appendices~\ref{appx:proof:subroutine:1}, \ref{appx:proof:subroutine:2}, \ref{appx:proof:subroutine:3}, and \ref{appx:proof:subroutine:4}.\looseness=-1
\begin{proposition} \label{prop:subroutine:1}
        Let $\bdelta \in \Delta_k$. If $\bdelta$ is the input to Subroutine 1, then the output $\bdelta'$ of Subroutine 1 satisfies $\bdelta' \in \Delta_k$ and  satisfies Property~\ref{property:notspecial}. Moreover,  if $\bdelta' \neq \bdelta$, then $\bdelta'$ dominates  $\bdelta$. 
\end{proposition}
\begin{proposition}\label{prop:subroutine:2}
        Let $\bdelta \in \Delta_k$ satisfy Property~\ref{property:notspecial}. If $\bdelta$ is the input to Subroutine 2, then the output $\bdelta'$ of Subroutine 2 satisfies   $\bdelta' \in \Delta_k$ and satisfies  Properties~\ref{property:notspecial} and \ref{property:forward}. Moreover, if $\bdelta' \neq \bdelta$, then $\bdelta'$ dominates $\bdelta$.
\end{proposition}
\begin{proposition}\label{prop:subroutine:3}
        Let $\bdelta \in \Delta_k$ satisfy Properties~\ref{property:notspecial} and \ref{property:forward}. If $\bdelta$ is the input to Subroutine 3, then the output $\bdelta'$ of Subroutine 3 satisfies   $\bdelta' \in \Delta_k$ and satisfies  Properties~\ref{property:notspecial},  \ref{property:forward}, and \ref{property:gap}. Moreover, if $\bdelta' \neq \bdelta$, then $\bdelta'$ dominates $\bdelta$. 
\end{proposition}
\begin{proposition} \label{prop:subroutine:4}
        Let $\bdelta \in \Delta_k$  satisfy Properties~\ref{property:notspecial},  \ref{property:forward}, and \ref{property:gap}. If $\bdelta$ is the input to Subroutine 4, then the output $\bdelta'$ of Subroutine 4  satisfies   $\bdelta' \in \Delta_k$ and satisfies  Properties~\ref{property:notspecial},  \ref{property:forward},  \ref{property:gap}, and \ref{property:reverse}. Moreover if $\bdelta' \neq \bdelta$, then $\bdelta'$ dominates $\bdelta$. 
\end{proposition}
Despite the simplicity of Subroutines 1, 2, 3, and 4, the proofs of the correctness of these subroutines are fairly involved. The difficulty of the proofs of Propositions~\ref{prop:subroutine:1} and \ref{prop:subroutine:3} (Appendices~\ref{appx:proof:subroutine:1} and \ref{appx:proof:subroutine:3}), for example, lies in showing that the outputs of Subroutines 1 and 3 dominate their inputs. The proof of Proposition~\ref{prop:subroutine:4} (Appendix~\ref{appx:proof:subroutine:4}), in particular, requires handling the two cases in the definition of $\hat{r}$ from Subroutine 4, and showing for each case that the properties of the input (Properties~\ref{property:notspecial}, \ref{property:forward}, and \ref{property:gap}) are retained in the process of obtaining an output that satisfies Property~\ref{property:reverse} and dominates the input.

\section{Numerical Experiments} \label{sec:experiments}
In this section, we conduct numerical experiments using synthetic and real-world data. Our goal in these experiments is to investigate the practical efficiency of our proposed algorithms, and the viability of the sample average approximation to solve assortment optimization problems more generally. 
A detailed discussion about the implementation of the various solution methods can be found in Appendix~\ref{appx:extra_experiments}. All experiments are performed on a MacBook Pro with $16$GB RAM and $10$ Cores. The Gurobi version is $12.0.1$.

\subsection{Experiments with Synthetic Data}\label{sec:experiments:MNL_Cutoff}

In this section, we consider assortment optimization under the multinomial logit with rank cutoffs, a parametric joint probability distribution for the random utility vector   proposed by \cite{bai2024assortment}. 
This is a representative setting for assortment optimization because the multinomial logit with rank cutoffs captures the well-known phenomenon that customers often have small consideration sets, while at the same time this class of assortment optimization problems is computationally demanding. 
This is also a natural setting to investigate because it offers several parameters that can be tuned, from the size of the consideration sets to the attraction parameters of the underlying multinomial logit model, making it attractive for investigating the practical efficiency of our proposed algorithms.

\subsubsection{Data Generation.}\label{sec:data_generation}
The multinomial logit model with rank cutoffs is a random utility maximization model. Similar to a traditional multinomial logit model, the utilities $U_1,\ldots,U_{N+1}$ are independent Gumbel random variables.\footnote{In a multinomial logit model, the random utilities $U_1,\ldots,U_{N+1}$  are of the form $U_i = \nu_i + \epsilon_i$, where the deterministic scalars $\nu_1,\ldots,\nu_{N+1}$ are referred to as the  {attraction parameters} and $\epsilon_1,\ldots,\epsilon_{N+1}$ are independent standard Gumbel random variables.} However, if an assortment does not contain any of the $L$ products with the highest utilities, then the customer does not make any purchase. Said another way, this joint probability distribution consists of drawing independent Gumbel random variables $U_1,\ldots,U_{N+1}$, and if there are more than $L$ products that satisfy $U_i > U_{N+1}$, then $U_{N+1}$ is increased so that it has the $(L+1)$th highest utility among all products. The fact that this is a random utility maximization model is established by \citet[Section 3]{bai2024assortment}. 

Our process for generating instances of the multinomial logit model with rank cutoffs begins by selecting integer parameters $N, M, L$. The integer $N$  denotes the number of products. The integer $M$ is parameter that is used to control the spread of the attraction parameters  in the multinomial logit model, with larger choices of $M$ generally leading to a smaller spread.\footnote{In Appendix~\ref{appx:M_spread}, we  demonstrate through numerical experiments that choosing larger values for the parameter $M$ generally leads to a smaller spread of the attraction parameters $\nu_1,\ldots,\nu_{N+1}$ in the underlying multinomial logit model.\looseness=-1} The integer $L$ denotes the rank cutoff, i.e., the maximum size of the consideration sets.  Given these parameters, our process for generating a random instance of a multinomial logit model with rank cutoffs is as follows.  \begin{itemize}
    \item \emph{Step 1}: Our process begins by  generating a ranking-based choice model that is supported on $M$ rankings. Specifically, we first select random rankings over the set of products $[N+1]$, with these rankings  $\sigma_1,\ldots,\sigma_M$ selected uniformly at random over the space of all $(N+1)!$ permutations. We then draw a random arrival probability for each of these rankings, $\lambda_1,\ldots,\lambda_M$, by sampling uniformly at random over  the probability simplex.\footnote{This is accomplished by drawing independent exponential random variables $\beta_1,\ldots,\beta_M \sim \text{Exponential}(1)$, and then defining $\lambda_m = \beta_m / (\sum_{j \in [M]} \beta_j)$. } 
    \vspace{0.5em} 
    
    \item \emph{Step 2:} We next fit a multinomial logit model to the underlying ranking-based choice model. We do this by generating transactional data, and then fitting a multinomial logit model to the transactional data. Specifically, we first generate $25000$ random assortments, where each product $i \in [N]$ is  included in the assortment with probability $0.05$; the no-purchase option $N+1$ is always included in the assortment. For each random assortment, we generate a transaction by selecting a random ranking $\sigma_1,\ldots,\sigma_M$ using the probabilities $\lambda_1,\ldots,\lambda_M$, and then selecting the most preferred product according to that ranking that is available in the assortment. We then fit the attraction parameters in the multinomial logit model to this transaction data using maximum likelihood estimation,  with the attraction parameter of the no-purchase option fixed to be equal to zero.
    \vspace{0.5em}
    
    \item \emph{Step 3}: To assign revenues to each of the products, we first sort the products by the attraction parameters that were obtained from the maximum likelihood estimation. We then generate random revenues $r_1,\ldots,r_N$ uniformly over $\{1,\ldots,10000\}$.  These revenues are then assigned to the attraction parameters in ascending order, with the lowest revenue being assigned to the highest attraction parameter, the second lowest revenue being assigned to the second-highest attraction parameter, and so on.\footnote{After performing Step 3, the attraction parameters in the underlying multinomial logit model satisfy $\nu_1 \ge \cdots \ge \nu_N$ and the revenues of the products satisfy $r_1 \le \cdots \le r_N$. We perform this step in order to make the downstream assortment optimization problem less trivial, since if high-revenue products also have high values for attraction parameters, then optimal assortments are likely to simply contain the highest revenue products. }
    \vspace{0.5em}

    \item \emph{Step 4}: To draw a sample from the multinomial logit model with rank cutoffs, we first draw the utilities $U_1,\ldots,U_{N+1}$ as independent Gumbel random variables with the attraction parameters obtained from the maximum likelihood estimation in Step 2. If $|\{i \in [N]: U_i > U_{N+1}\}| > L$, then we increase $U_{N+1}$ so that it is the $(L+1)$th highest utility.  
\end{itemize}

To solve the resulting assortment optimization problem with the sample average approximation, we follow the process outlined Section~\ref{sec:saa}. That is, we first select an integer $\tilde{K}$, and then we generate the set of $\tilde{K}$ samples using the process in Step 4. To evaluate the out-of-sample performance of  assortments obtained by the sample average approximation, we independently generate a validation set of $10^4$ samples using the same process as in Step 4 (see Section~\ref{sec:experiments:mnl:accelerated}). 
\cite{bai2024assortment} developed an algorithm for these assortment optimization problems that finds an $(1-\epsilon)$-factor approximate solution that runs in $\mathcal{O}((N/\epsilon)^{\mathcal{O}(1/\epsilon^2)})$ time. Because we are interested in computing approximate solutions that are within several percentage points of optimal for problems with moderate-to-large numbers of products $N$, we focus in the following numerical experiments on methods based on the sample average approximation. Specifically, we compare the algorithms proposed in this paper for solving the sample average approximation to those from \cite{bertsimas2019exact}.\looseness=-1

\subsubsection{Experiment  1.} \label{sec:experiments:mnl:xset} In our first set of experiments, we investigate a setting of assortment optimization problems in which there is a moderate number of products, $N = 50$, and where we choose a very large number of samples $\tilde{K}$ with the goal of the sample average approximation producing an assortment that is  close to optimal.  This setting is essentially that described in Section~\ref{sec:exclusion_set:comparison} in which the exclusion set formulation is intended to be efficacious. 

Table~\ref{tab:runtime-comparison}  compares the computation times  for three different solution methods for the sample average approximation. The first solution method is our exclusion set formulation~\eqref{prob:exclusionset} from Section~\ref{sec:exclusion_set}. The second solution method is our accelerated Benders decomposition method from Section~\ref{sec:main_alg}. The third solution method is the mixed-integer programming formulation~\eqref{prob:misic_no_extra} from \cite{bertsimas2019exact}. We solve \eqref{prob:misic_no_extra} from Section~\ref{sec:exclusionset:step1} instead of their original formulation \eqref{prob:misic} because the multinomial logit model with rank-cutoffs leads to small consideration sets, which makes  the original formulation~\eqref{prob:misic} much slower to solve. Table~\ref{tab:runtime-comparison} does not compare to the original two-phase Benders decomposition method by \cite{bertsimas2019exact} from Section~\ref{sec:benders} because it is slower than our accelerated Benders decomposition method (see Sections~\ref{sec:experiments:mnl:accelerated} and \ref{sec:results:toubia}), which itself is slower than the exclusion set formulation in Table~\ref{tab:runtime-comparison}. All results in Table~\ref{tab:runtime-comparison} are averaged over ten replications.\looseness=-1

\begin{table}[t]
\TABLE{Experiments from \S\ref{sec:experiments:mnl:xset} - Computation Time. \label{tab:runtime-comparison}}
{\centering
\begin{tabular}{cccccrrr}
\toprule
& & & & & \multicolumn{3}{c}{Computation Time} \\
\cmidrule(lr){6-8}
$N$ & $M$ & $\tilde{K}$ & $L$ & Budget  &\shortstack{XSET}& \shortstack{MIP}& \shortstack{ABD}\\
\midrule
50 & 5 & 25000 & 5 & Inf & 3.45 & 23.01 & 70.01  \\ 
50 & 5 & 50000 & 5 & Inf & 8.51 & 297.58 & 324.47  \\ 
50 & 5 & 100000 & 5 & Inf & 19.50 & 356.43 & 551.52  \\ 
50 & 15 & 25000 & 5 & Inf & 5.80 & 36.82 & 144.08  \\ 
50 & 15 & 50000 & 5 & Inf & 15.17 & 86.91 & 388.73  \\ 
50 & 15 & 100000 & 5 & Inf & 30.51 & 329.10 & 724.62  \\ 
50 & 25 & 25000 & 5 & Inf & 6.71 & 28.52 & 102.25  \\ 
50 & 25 & 50000 & 5 & Inf & 13.46 & 66.79 & 293.68  \\ 
50 & 25 & 100000 & 5 & Inf & 34.78 & 397.20 & 720.32  \\ 
\bottomrule
\end{tabular}}
{Computation times in seconds of our exclusion set formulation~\eqref{prob:exclusionset} from \S\ref{sec:exclusion_set} (XSET), the original mixed-integer programming formulation \eqref{prob:misic_no_extra} by \cite{bertsimas2019exact} (MIP), and our accelerated Benders decomposition method from \S\ref{sec:main_alg} (ABD). Results are averaged over ten replications and rounded to two decimal places.}
\end{table}

There are several main observations from Table~\ref{tab:runtime-comparison}. First, we observe that the relative speedups of the exclusion set formulation~\eqref{prob:exclusionset} compared to the other solution methods increases with the number of samples $\tilde{K}$. For example, when $M = 5$, the ratio of computation times of  the existing mixed-integer programming formulation~\eqref{prob:misic_no_extra} to the exclusion set formulation~\eqref{prob:exclusionset} increases from a ratio of $23.01/3.45 \approx 6.67$ (in the case of $\tilde{K} = 2.5 \times 10^4$) to a ratio of  $356.43/19.50 \approx 18.27$ (in the case of $\tilde{K} = 10^5$). This can be attributed to the fact that the number of `collisions' of rankings in the exclusion set formulation increases as the number of samples $\tilde{K}$ increases.\footnote{The notion of collisions is found in Section~\ref{sec:exclusion_set:comparison}.} This explanation is corroborated by  Table~\ref{tab:problem_size}, which shows that  increasing the number of samples $\tilde{K}$ leads to larger ratio in problem sizes between \eqref{prob:exclusionset} and \eqref{prob:misic_no_extra}. For example, when $M = 5$, Table~\ref{tab:problem_size} shows that the ratio of the number of decision variables of the existing mixed-integer programming formulation~\eqref{prob:misic_no_extra} to the exclusion set formulation~\eqref{prob:exclusionset} increases from a ratio of $121537 / 41933 \approx 2.9$ (in the case of $\tilde{K} = 2.5 \times 10^4$) to a ratio of $472752 / 117683 \approx 4.0$ (in the case of $\tilde{K} = 10^5$). A large number of collisions not only implies that \eqref{prob:exclusionset} has a smaller problem size than \eqref{prob:misic_no_extra}; it also implies  that \eqref{prob:exclusionset} generally has a tighter linear programming relaxation than \eqref{prob:misic_no_extra}.\footnote{Additional  results regarding the  linear programming relaxation gaps of \eqref{prob:misic_no_extra} and \eqref{prob:exclusionset} can be found in Appendix~\ref{appx:exclusion_set_budget}.\looseness=-1}\looseness=-1 

Table~\ref{tab:runtime-comparison} also highlights the impact of the spread of the attraction parameters in the multinomial logit model. As discussed in Section~\ref{sec:data_generation}, larger choices of $M$ generally lead to a lower spread of the attraction parameters. Indeed,  we observe from Table~\ref{tab:runtime-comparison} that the ratio of computation times of the existing mixed-integer programming formulation~\eqref{prob:misic_no_extra} to the exclusion set formulation~\eqref{prob:exclusionset}  increases as $M$ decreases, everything else held constant. This observation is further corroborated by comparing the ratios of problem sizes of \eqref{prob:exclusionset} and \eqref{prob:misic_no_extra}  in Table~\ref{tab:problem_size}.  We emphasize  that  even in the case of $M=5$, the resulting attraction parameters are still close enough for the problem to be nontrivial, in the sense that all $N$ products still appear  among the $L=5$ most preferred products in at least some of the generated samples; see Appendix~\ref{appx:M_spread} for details.

\begin{table}[t]
\TABLE{Experiments from \S\ref{sec:experiments:mnl:xset} - Problem Size. \label{tab:problem_size}}
{\centering 
\begin{tabular}{ccccc   
                rr    
                rr}   
\toprule
& & & & &
\multicolumn{2}{c}{\shortstack{Number of \\Decision Variables}} & 
\multicolumn{2}{c}{\shortstack{Number of \\ Constraints}} \\
\cmidrule(lr){6-7} \cmidrule(lr){8-9}
$N$ & $M$& $T$& $\tilde{K}$ & $L$ & \shortstack{XSET} & MIP & XSET & MIP \\
\midrule
50 & 5 & 25000 & 25000 & 5 & 41933 & 121537 & 225085 & 389083  \\ 
50 & 5 & 25000 & 50000 & 5 & 56817 & 234661 & 330831 & 751297  \\ 
50 & 5 & 25000 & 100000 & 5 & 117683 & 472752 & 690470 & 1513579  \\ 
50 & 15 & 25000 & 25000 & 5 & 52779 & 118578 & 263446 & 379954  \\ 
50 & 15 & 25000 & 50000 & 5 & 94916 & 238601 & 489692 & 764422  \\ 
50 & 15 & 25000 & 100000 & 5 & 160192 & 481439 & 871195 & 1541985  \\ 
50 & 25 & 25000 & 25000 & 5 & 56244 & 119008 & 275483 & 381340  \\ 
50 & 25 & 25000 & 50000 & 5 & 100621 & 239326 & 510527 & 766761  \\ 
50 & 25 & 25000 & 100000 & 5 & 174470 & 482741 & 925476 & 1546146  \\  
\bottomrule
\end{tabular}
}{Number of decision variables and constraints for our exclusion set formulation~\eqref{prob:exclusionset} from \S\ref{sec:exclusion_set} (XSET) and the original mixed-integer programming formulation \eqref{prob:misic_no_extra} by \cite{bertsimas2019exact} (MIP). Results are averaged over ten replications and rounded to nearest integer.}
\end{table}

In summary, this first set of experiments shows that the exclusion set formulation can offer significant speedups for solving the sample average approximation, where the speedups are most pronounced in settings with large numbers of samples. This showcases the value of the exclusion set formulation in applications with small consideration sets, particularly in realistic settings where some products  appear in consideration sets more often than others. We note that the aforementioned tables compare the solution methods in settings without constraints on the set of feasible assortments; additional experiments  for the exclusion set formulation for settings with cardinality constraints on the set of feasible assortments can be found in 
Appendix~\ref{appx:exclusion_set_budget}.

\subsubsection{Experiment 2.} \label{sec:experiments:mnl:accelerated}
In our second set of experiments, we investigate assortment optimization problems in which there are larger numbers of products, $N \in \{100,500\}$, and where the rank cutoff satisfies $L \in \{10,15\}$. This setting contrasts to that of Section~\ref{sec:experiments:mnl:xset} because  the inflated size of the assortment optimization problems in the present Section~\ref{sec:experiments:mnl:accelerated}  prevents the formulations  \eqref{prob:exclusionset} and \eqref{prob:misic_no_extra}  from being solvable in practical computation times when the number of samples $\tilde{K}$ is large.  This second set of experiments also provides a natural setting for   empirically assessing   the out-of-sample performance of assortments obtained by the sample average approximation as a function of the number of samples $\tilde{K}$ in problems with hundreds of products.\looseness=-1 
\begin{table}[t]
\TABLE{Experiments from \S\ref{sec:experiments:mnl:accelerated} - Computation Time.\label{tab:runtime-comparison-ab}}
{\centering  
\begin{tabular}{cccccrr}
\toprule
& & & & & \multicolumn{2}{c}{\shortstack{Computation Time}} \\
\cmidrule(lr){6-7}
$N$ & $M$ & $\tilde{K}$ & $L$ & Budget  &  ABD & BD  \\
\midrule
100 & 50 & 2500 & 10 & 5 & 2.40 & 39.40  \\ 
100 & 50 & 5000 & 10 & 5 & 10.48 & 277.76  \\ 
100 & 50 & 10000 & 10 & 5 & 61.13 & 945.10  \\ 
100 & 50 & 20000 & 10 & 5 & 125.97 & 2398.48  \\ 
500 & 50 & 2500 & 15 & 5 & 2.06 & 91.63  \\ 
500 & 50 & 5000 & 15 & 5 & 3.85 & 302.44  \\ 
500 & 50 & 10000 & 15 & 5 & 12.10 & 987.97  \\ 
500 & 50 & 20000 & 15 & 5 & 11.85 & 4805.63  \\ 
\bottomrule
\end{tabular}
}{Computation times in seconds of our accelerated Benders decomposition method from \S\ref{sec:main_alg} (ABD) and the original two-phase Benders decomposition method from \S\ref{sec:benders} by \cite{bertsimas2019exact} (BD). Results are averaged over ten replications and rounded to two decimal places.}
\end{table}

Table~\ref{tab:runtime-comparison-ab}   compares the computation times  for the two   decomposition methods for the sample average approximation; namely, our accelerated Benders decomposition method, which is implemented as described in Section~\ref{sec:main_alg}, and the original two-phase Benders decomposition method from \cite{bertsimas2019exact} from Section~\ref{sec:benders}.\footnote{Our implementation of the  original two-phase Benders decomposition method from Section~\ref{sec:benders} is based on the publicly available code provided by \cite{bertsimas2019exact}; see Appendix~\ref{appx:misic_benders_remark} for details.}. The results in Table~\ref{tab:runtime-comparison-ab} show  significant speedups of our accelerated Benders decomposition method  over the original two-phase Benders decomposition method.\looseness=-1

As discussed in Section~\ref{sec:main_alg}, the improved computation times of our accelerated Benders decomposition method can be attributed to two factors: our faster algorithms for generating cuts, and the strength of the generated cuts. To better understand the role of those two factors in  driving the speedups of the accelerated Benders decomposition method in Table~\ref{tab:runtime-comparison-ab}, we present  granular results for the two solution methods in Tables~\ref{tab:time_by_phase} and \ref{tab:small_benders_cuts_added}. Table~\ref{tab:time_by_phase} shows the average computation time for each solution method for Phase 1 and Phase 2, and Table~\ref{tab:small_benders_cuts_added} shows the average number of cuts generated by each solution method in Phase 1 and Phase 2. These tables show significant  reductions in computation time and in number of generated cuts by our accelerated Benders decomposition method in both phases. In particular, we draw several insights from Tables~\ref{tab:time_by_phase} and \ref{tab:small_benders_cuts_added}. 

    \begin{table}[t]
\TABLE{Experiments from \S\ref{sec:experiments:mnl:accelerated} - Computation Time by Phase.\label{tab:time_by_phase}}
{\centering 
\begin{tabular}{cccccrrrr}                 
\toprule
& & & & & \multicolumn{4}{c}{\shortstack{Computation Time}} \\
\cmidrule(lr){6-9}
&&&&& \multicolumn{2}{c}{ABD} & 
\multicolumn{2}{c}{BD}\\
\cmidrule(lr){6-7} \cmidrule(lr){8-9}
 $N$ & $M$ & $\tilde{K}$ & $L$ & Budget &Phase 1 & Phase 2 
&Phase 1 & Phase 2 \\
\midrule
100 & 50 & 2500 & 10 & 5 & 0.67 & 1.73 & 18.73 & 20.67  \\ 
100 & 50 & 5000 & 10 & 5 & 1.85 & 8.63 & 68.01 & 209.75  \\ 
100 & 50 & 10000 & 10 & 5 & 12.18 & 48.95 & 134.04 & 811.06  \\ 
100 & 50 & 20000 & 10 & 5 & 11.45 & 114.52 & 618.38 & 1780.1  \\ 
500 & 50 & 2500 & 15 & 5 & 1.18 & 0.88 & 76.38 & 15.25  \\ 
500 & 50 & 5000 & 15 & 5 & 3.01 & 0.85 & 247.34 & 55.09  \\ 
500 & 50 & 10000 & 15 & 5 & 9.73 & 2.37 & 894.57 & 93.39  \\ 
500 & 50 & 20000 & 15 & 5 & 2.68 & 9.17 & 4464.64 & 340.99  \\ 
\bottomrule
\end{tabular}
}{Computation times in seconds of our accelerated Benders decomposition method from \S\ref{sec:main_alg} (ABD) and the original two-phase Benders decomposition method from \S\ref{sec:benders} by \cite{bertsimas2019exact} (BD), split by Phase 1 and Phase 2. Results are averaged over ten replications and rounded to two decimal places.}
\end{table}
\begin{table}[t]
\TABLE{Experiments from \S\ref{sec:experiments:mnl:accelerated} - Number of Cuts Added by Phase.\label{tab:small_benders_cuts_added}}
{\centering
\begin{tabular}{ccccc
                r 
                r 
                r
                r}                 
\toprule
&&&&&\multicolumn{4}{c}{\shortstack{Number of Generated Cuts}} \\
\cmidrule(lr){6-9}
&&&&& \multicolumn{2}{c}{ABD} & \multicolumn{2}{c}{BD}\\
\cmidrule(lr){6-7} \cmidrule(lr){8-9}
 $N$ & $M$ & $\tilde{K}$ & $L$ & Budget &Phase 1 & Phase 2 
& Phase 1 & Phase 2 \\
\midrule
100 & 50 & 2500 & 10 & 5 & 2790 & 341 & 8235 & 14156  \\ 
100 & 50 & 5000 & 10 & 5 & 5492 & 1406 & 16357 & 33617  \\ 
100 & 50 & 10000 & 10 & 5 & 10872 & 4636 & 32732 & 74950  \\ 
100 & 50 & 20000 & 10 & 5 & 21402 & 5933 & 66355 & 100202  \\ 
500 & 50 & 2500 & 15 & 5 & 2609 & 14 & 7686 & 3267  \\ 
500 & 50 & 5000 & 15 & 5 & 5164 & 62 & 15246 & 6810  \\ 
500 & 50 & 10000 & 15 & 5 & 10275 & 268 & 30380 & 18660  \\ 
500 & 50 & 20000 & 15 & 5 & 20479 & 720 & 60722 & 23181  \\ 
\bottomrule
\end{tabular}}
{Number of cuts generated in the accelerated Benders decomposition method from \S\ref{sec:main_alg} (ABD) and the original two-phase Benders decomposition method from \S\ref{sec:benders} by \cite{bertsimas2019exact} (BD), split by Phase 1 and Phase 2. Results are averaged over ten replications and rounded to nearest integer.}
\end{table}

The first insight is that the speedups of the accelerated Benders decomposition method in Phase 1 can be attributed in large part to the reduction in computation times for generating cuts.  Specifically, Table~\ref{tab:time_by_phase} shows that the ratios of computation times of   Phase 1 of the original two-phase Benders decomposition method to Phase 1 of our accelerated Benders decomposition method range from $134.04/12.18\approx 11 $ (in the case of $N = 100$ and $\tilde{K} = 10000$) to $4464.64 / 2.68 \approx 1666 $ (in the case of $N = 500$ and $\tilde{K} = 20000$). These ratios in computation times are much larger than the ratios in the number of cuts generated in Phase 1 between the two solution methods shown in Table~\ref{tab:small_benders_cuts_added}, which range from $7686 / 2609 \approx 2.95$ (in the case of $N = 500$ and $\tilde{K} = 2500$) to $66355 / 21402 \approx 3.10$ (in the case of $N = 100$ and $\tilde{K} = 20000$). Furthermore, it follows from  Table~\ref{tab:time_by_phase} that in the case of $N = 500$, the total computation times in Table~\ref{tab:runtime-comparison-ab} is primarily driven by the computation time of  Phase 1. These results underscore the value of our faster algorithms from Section~\ref{sec:main_alg:algs} for generating cuts, which reduced the computation time in Phase 1 from $\mathcal{O}(KN^2)$ to $\mathcal{O}(\sum_{k=1}^K L_k \log L_k)$. 

The second insight is that the improved strength of our generated cuts in the accelerated Benders decomposition method also plays a significant role in reducing the overall computation time compared to the original two-phase Benders decomposition method. Indeed, Table~\ref{tab:small_benders_cuts_added} shows a significant reduction in the accelerated Benders decomposition method in the total number of cuts generated across Phase 1 and Phase 2 compared to the original two-phase Benders decomposition method. This difference is seen in Table~\ref{tab:small_benders_cuts_added} most prominently in Phase 2, which reflects the strength of the cuts from both Phase 1 and Phase 2. For example, when $N = 500$ and $\tilde{K} = 2500$, the number of Phase 2 cuts reduces from 3267 in the original two-phase Benders decomposition method to 14 in the accelerated Benders decomposition  method.  

It is worthwhile to make a couple of additional comments about Table~\ref{tab:small_benders_cuts_added} and our implementation of the solution methods. First, we note that the numbers of cuts shown in this table refers only to the number of generated cuts, i.e., it does not include the number of initial cuts. To avoid unboundedness of the outer problem~\eqref{prob:outer}, both of the two solution methods initialize the outer problem~\eqref{prob:outer} in Phase 1 with a single constraint for each ranking (see Appendix~\ref{appx:extra_experiments:initial_cuts} for details). In both solution methods, these initial as well as the generated cuts in Phase 1 are passed as warm starts into Phase 2.   Second, we observe that the number of cuts in Phase 1 in Table~\ref{tab:small_benders_cuts_added} for the accelerated Benders decomposition method is typically close to the number of samples $\tilde{K}$. This suggests that only a  small number of cuts were actually required in these experiments to determine the optimal solution for the linear programming relaxation of \eqref{prob:outer} that is solved in Phase 1. However, this fact is not identified by the original two-phase Benders decomposition method because its cuts are weaker, as demonstrated by the high number of cuts for Phase 1 of that solution method in Table~\ref{tab:small_benders_cuts_added}. This underscores the value of the Pareto-optimal cuts that are generated by  our accelerated Benders decomposition method.

We conclude by calculating the approximation gap of the assortments obtained by the sample average approximation.  Given any assortment $S \in \mathscr{S}$, let $J^*(S) \triangleq  \Exp \left[ \sum_{i \in S} r_i \mathbb{I} \left \{ U_i = \max_{j \in S} U_j  \right \}\right]$ denote the expected revenue, and let $v^* \triangleq \sup_{S \in \mathscr{S}} J^*(S)$ denote the optimal objective value of the stochastic program~\eqref{prob:rum}.  Let $\hat{v}_{\tilde{K}}$ and $\hat{S}_{\tilde{K}}$ denote the optimal objective value and optimal solution of the sample average approximation~\eqref{prob:saa}, which is shown in Sections~\ref{sec:intro:methods} and \ref{sec:benders} to be equivalent to the outer problem~\eqref{prob:outer}.  We recall that the optimal objective value of the sample average approximation $\hat{v}_{\tilde{K}}$  is an upward biased estimate of the optimal objective value of the stochastic program $v^*$ \cite[Proposition 5.6]{shapiro2021lectures}. Moreover, an unbiased estimate of $J^*(\hat{S}_{\tilde{K}})$ can be obtained by using Monte-Carlo simulation to generate an out-of-sample  validation set of samples $(\breve{U}^{\tilde{k}}_1,\ldots,\breve{U}^{\tilde{k}}_{N+1})$ for $\tilde{k} \in \{1,\ldots,\breve{K}\}$ and calculating $\breve{J}_{\breve{K}}(\hat{S}_{\tilde{K}}) \triangleq  \frac{1}{\breve{K}} \sum_{\tilde{k}=1}^{\breve{K}}\sum_{i \in \hat{S}_{\tilde{K}}} r_i \mathbb{I} \left \{ \breve{U}^{\tilde{k}}_i = \max_{j \in \hat{S}_{\tilde{K}}} \breve{U}^{\tilde{k}}_j  \right \}$.   A lower bound estimate of the approximation gap $J^*(\hat{S}_{\tilde{K}}) / v^* \times 100\%$ for the  sample average approximation~\eqref{prob:saa} is thus given by $\breve{J}_{\tilde{K}'}(\hat{S}_{\tilde{K}})  / \hat{v}_{\tilde{K}} \times 100\%$. 

\begin{table}[t]
\TABLE{Experiments from \S\ref{sec:experiments:mnl:accelerated} - Estimated Approximation Gap.\label{tab:saa-convergence}}
{\centering
\begin{tabular}{ccccccc}
\toprule
$N$ & $M$ & $\tilde{K}$ & $L$ & Budget & $\breve{K}$ & \shortstack{Approximation Gap}\\
\midrule
100 & 50 & 2500 & 10 & 5 & 10000 & 91.28\%  \\ 
100 & 50 & 5000 & 10 & 5 & 10000 & 94.21\%  \\ 
100 & 50 & 10000 & 10 & 5 & 10000 & 95.71\%  \\ 
100 & 50 & 20000 & 10 & 5 & 10000 & 97.99\%  \\ 
500 & 50 & 2500 & 15 & 5 & 10000 & 77.80\%  \\ 
500 & 50 & 5000 & 15 & 5 & 10000 & 83.85\%  \\ 
500 & 50 & 10000 & 15 & 5 & 10000 & 87.68\%  \\ 
500 & 50 & 20000 & 15 & 5 & 10000 & 91.88\%  \\ 
\bottomrule
\end{tabular}}
{Estimated approximation gap of optimal assortments obtained by solving the sample average approximation~\eqref{prob:saa} using the accelerated Benders decomposition method from \S\ref{sec:main_alg}. Calculations performed as described in \S\ref{sec:experiments:mnl:accelerated}. Results averaged over ten replications and rounded to two decimal places. }
\end{table}

Table~\ref{tab:saa-convergence} reports the lower bound estimate of the approximation gap of the assortments obtained from solving the sample average approximation, using a validation set of $\breve{K} =10000$ samples. As expected, the results show that as the number of products $N$ and the rank cutoff $L$ increase, a larger number of samples $\tilde{K}$ are required for the sample average approximation~\eqref{prob:saa} to yield assortments that are close to optimal with respect to the stochastic program~\eqref{prob:rum}. That said, Tables~\ref{tab:runtime-comparison-ab} and \ref{tab:saa-convergence} together show for these experiments that assortments that are within $90\%$ of optimal can be found within one minute. This underscores the value of our accelerated Benders decomposition method, and the potential of the sample average approximation more generally, for solving assortment optimization problems with hundreds of products.

\subsection{Experiments with Real Data} \label{sec:results:toubia}
 In this section, we evaluate the performance of our accelerated Benders decomposition method (Section~\ref{sec:main_alg})  and the original two-phase Benders decomposition method of \cite{bertsimas2019exact} (Section~\ref{sec:benders}) using a widely studied dataset from  \cite{toubia2003fast}. This setting is  valuable for three reasons. First, it represents a real-world application, where the samples are generated using individual-level questionnaire responses designed from conjoint analysis. As such, this setting is representative of the complexity and scale typically encountered in practical assortment optimization  problems, even though the samples in this dataset are obtained by surveys instead of Monte-Carlo simulation. Second, in contrast to the experiments from Section~\ref{sec:experiments:MNL_Cutoff}, the samples in this real-world dataset do not have small consideration sets, which provides a different setting for exploring the efficacy of the accelerated Benders decomposition method. Third, this dataset has been used by \cite{bertsimas2019exact} to benchmark their own Benders decomposition method due to its large size ($N = 3584, \tilde{K} = 330)$, thereby offering an established point of comparison for assessing the practical efficiency of our accelerated Benders decomposition method. Additional numerical results on the out-of-sample performance of the sample average approximation on the dataset from \cite{toubia2003fast} can be found in Appendix~\ref{appx:cross_fold}.

\begin{table}[t]
\TABLE{Experiments from \S\ref{sec:results:toubia} - Computation Time and Objective. \label{tab:total_time_objective}}
{
\begin{tabular}{crrrr}
\toprule
{Constraints} & 
\multicolumn{2}{c}{\shortstack{Computation Time}} & 
\multicolumn{2}{c}{Objective} \\
\cmidrule(lr){2-3} \cmidrule(lr){4-5}
& \shortstack{ABD} 
& \shortstack{BD} 
& Phase 1 & Phase 2 \\
\midrule
$\sum x_i = 2$ & 3.40 & 106.47 & 59.22 & 59.22 \\ 
$\sum x_i = 3$ & 5.87 & 276.21 & 66.57 & 66.29 \\ 
$\sum x_i = 4$ & 14.05 & 392.50 & 71.21 & 70.24 \\ 
$\sum x_i = 5$ & 18.09 & 387.86 & 73.85 & 72.82 \\ 
$\sum x_i = 6$ & 167.14 & 705.85 & 75.59 & 74.33 \\ 
$\sum x_i = 7$ & 157.88 & 861.02 & 76.89 & 75.32 \\ 
$\sum x_i = 8$ & 371.61 & 1396.28 & 77.88 & 76.12 \\ 
$\sum x_i = 9$ & 508.31 & 2177.62 & 78.68 & 76.88 \\ 
$\sum x_i = 10$ & 2215.70 & 4240.28 & 79.36 & 77.41 \\ 
\bottomrule
\end{tabular}
}{Results under different cardinality constraints for the computation time in seconds for the accelerated Benders decomposition method from \S\ref{sec:main_alg} (ABD), the computation time in seconds of the original two-phase Benders decomposition method from \S\ref{sec:benders} by \cite{bertsimas2019exact} (BD), the optimal objective value of the linear programming relaxation (Phase 1), and the optimal objective value of the mixed-integer program (Phase 2). Results are averaged over three replications to remove  variability in solve times, and results are rounded to two decimal places. }
\end{table}

\begin{table}[t]
\TABLE{Experiments from \S\ref{sec:results:toubia} - Computation Time by Phase.\label{tab:large_benders_time_phases}}
{\centering 
\vspace{0.5em}
\begin{tabular}{l r r r r
                } 
\toprule
&\multicolumn{4}{c}{\shortstack{Computation Time}} \\
\cmidrule(lr){2-5}
Constraints & 
\multicolumn{2}{c}{\shortstack{ABD}} & 
\multicolumn{2}{c}{\shortstack{BD}} 
 \\
\cmidrule(lr){2-3} \cmidrule(lr){4-5}
& \multicolumn{1}{c}{Phase 1} & \multicolumn{1}{c}{Phase 2} 
& \multicolumn{1}{c}{Phase 1} & \multicolumn{1}{c}{Phase 2} \\
\midrule
$\sum x_i = 2$ & 2.78 & 0.62 & 102.66 & 3.80 \\ 
$\sum x_i = 3$ & 3.31 & 2.55 & 263.99 & 12.22 \\ 
$\sum x_i = 4$ & 3.40 & 10.65 & 332.14 & 60.36 \\ 
$\sum x_i = 5$ & 4.22 & 13.87 & 333.92 & 53.94 \\ 
$\sum x_i = 6$ & 4.05 & 163.09 & 404.12 & 301.73 \\ 
$\sum x_i = 7$ & 4.75 & 153.12 & 464.12 & 396.90 \\ 
$\sum x_i = 8$ & 4.94 & 366.68 & 497.50 & 898.79 \\ 
$\sum x_i = 9$ & 5.23 & 503.08 & 590.07 & 1587.55 \\ 
$\sum x_i = 10$ & 5.39 & 2210.31 & 651.65 & 3588.63 \\ 
\bottomrule
\end{tabular}}
{Results under different cardinality constraints for the computation time in seconds for the accelerated Benders decomposition method from \S\ref{sec:main_alg} (ABD) and the computation time in seconds of the original two-phase Benders decomposition method from \S\ref{sec:benders} by \cite{bertsimas2019exact} (BD), split by Phase 1 and Phase 2. Results are averaged over three replications to remove  variability in solve times, and results are rounded to two decimal places. }
\end{table}

    The results of these experiments using the dataset from \cite{toubia2003fast} are shown in Tables~\ref{tab:total_time_objective}, \ref{tab:large_benders_time_phases}, and \ref{tab:large_benders_cuts_added}.  Table~\ref{tab:total_time_objective} presents the computation times of the two solution methods, as well as gives the optimal objective values after Phase 1 and after Phase 2.\footnote{In Table~\ref{tab:total_time_objective},  Objective for Phase 1 is the optimal objective value of the linear programming relaxation of \eqref{prob:outer}, and  Objective for Phase 2 is the optimal objective value of \eqref{prob:outer}. } In the particular case of the cardinality constraint of five products ($\sum_i x_i = 5$),     our accelerated Benders decomposition method reduces the computation time from around six minutes to around 18 seconds. Tables~\ref{tab:large_benders_time_phases} and \ref{tab:large_benders_cuts_added} compares the computation times and number of generated cuts of the two solution methods separated by Phase 1 and Phase 2. These tables show that the speedups from the accelerated Benders decomposition method can be attributed in large part to our faster algorithm for generating cuts in Phase 1. In particular, the time to generate cuts in Phase 1 for our accelerated Benders decomposition method remained less than five seconds across all budget settings, demonstrating that our accelerated Benders decomposition method can rapidly generate high quality upper bounds.  
    
    In summary, these experiments show the capability of our accelerated Benders decomposition method in solving real-world large-scale problems without any heuristics or parameter tuning. We emphasize that although this setting has considerably fewer rankings than products, we show using cross validation in  Appendix~\ref{appx:cross_fold} that the relatively small number of samples is sufficient in this real-world dataset to obtain a near-optimal assortment for the true assortment optimization problem~\eqref{prob:rum}. As such, these experiments show the potential of the sample average approximation, in combination with our algorithms, in obtaining near-optimal solutions to problems with thousands of products. 

\begin{table}[t]
\TABLE{Experiments from \S\ref{sec:results:toubia} - Number of Cuts Added by Phase.\label{tab:large_benders_cuts_added}}
{\centering 
\begin{tabular}{l r r  r r
                }                 
\toprule
&\multicolumn{4}{c}{\shortstack{Number of Generated Cuts}} \\
\cmidrule(lr){2-5}
Constraints & 
\multicolumn{2}{c}{\shortstack{ABD}} & 
\multicolumn{2}{c}{BD}\\
\cmidrule(lr){2-3} \cmidrule(lr){4-5} 
& \multicolumn{1}{c}{Phase 1} & \multicolumn{1}{c}{Phase 2} 
& \multicolumn{1}{c}{Phase 1} & \multicolumn{1}{c}{Phase 2} \\
\midrule
$\sum x_i = 2$ & 1140 & 0 & 1312 & 922 \\ 
$\sum x_i = 3$ & 1887 & 47 & 3014 & 2567 \\ 
$\sum x_i = 4$ & 1967 & 286 & 3636 & 3408 \\ 
$\sum x_i = 5$ & 2444 & 439 & 3776 & 5597 \\ 
$\sum x_i = 6$ & 2606 & 1857 & 4267 & 13220 \\ 
$\sum x_i = 7$ & 2887 & 3708 & 5072 & 17565 \\ 
$\sum x_i = 8$ & 3032 & 5682 & 5562 & 25744 \\ 
$\sum x_i = 9$ & 3361 & 3989 & 5851 & 24357 \\ 
$\sum x_i = 10$ & 3478 & 10970 & 6321 & 41677 \\ 
\bottomrule
\end{tabular}}{Number of cuts generated under different budget constraints by the accelerated Benders decomposition method from \S\ref{sec:main_alg} (ABD) and the original two-phase Benders decomposition method from \S\ref{sec:benders} by \cite{bertsimas2019exact} (BD), split by Phase 1 and Phase 2.}
\end{table}

\section{Conclusion and Future Work} \label{sec:conclusion}
In this paper, we make the sample average approximation a more viable approach to solving assortment optimization problems by developing faster algorithms for solving assortment optimization problems under the ranking-based choice model. We developed our algorithms  by drawing connections between  several streams of literature, ranging from Pareto optimality in Benders decomposition \citep{magnanti1981accelerating} to mechanism design \citep{ma2023assortment} and isotonic regression \citep{ahuja2001fast}. We believe that our algorithms and their accompanying theoretical guarantees provide new compelling evidence that the sample average approximation can offer value to industry, allowing firms to focus on the task of estimating accurate  discrete choice models. These results also raise a number of interesting directions for future work on the theory and practice of assortment optimization.\looseness=-1 
\begin{enumerate}
    \item \textbf{Theoretical convergence rates:}  It is well known for stochastic discrete optimization problems that the probability that the optimal solution set of the sample average approximation is a subset of the optimal solution set of the stochastic program converges to one exponentially fast as the number of samples tends to infinity \citep[Section 2.2]{kleywegt2002sample}. To the best of our knowledge, it is unknown for the particular application of assortment optimization~\eqref{prob:rum} whether the sample average approximation~\eqref{prob:saa} converges even faster to \eqref{prob:rum}, even in simple settings where the utilities  $U_1,\ldots,U_{N+1}$ are independent Gumbel random variables. Future work on  refined convergence rates would lead to an improved theoretical understanding of the settings where \eqref{prob:saa} can provide a close approximation of \eqref{prob:rum} with small or moderate numbers of samples.\looseness=-1  
    \vspace{0.5em}
    
    \item \textbf{Faster algorithms under independent random utilities.} Much of the assortment optimization literature has focused on random utility maximization models where the random utilities $U_1,\ldots,U_{N+1}$ are independent random variables. At the same time, these problems are known to be computationally demanding, even in the simple case where the utilities are independent exponential random variables \citep{aouad2023}. An interesting direction for future work would be to develop   faster Monte Carlo-based algorithms for solving \eqref{prob:rum} that are tailored to the setting where the utilities $U_1,\ldots,U_{N+1}$ are independent random variables. 
        \vspace{0.5em}

    \item \textbf{Faster algorithms for non-rational choice models}. The accelerated Benders decomposition from Section~\ref{sec:main_alg} introduces new connections between isotonic regression and develops structural results for Pareto-optimal cuts. These insights may prove to be useful in the development of Benders decomposition methods for classes of assortment optimization problems under irrational choice models; see \cite{chen2022decision} and \cite{akchen2021assortment}.   
\end{enumerate}

\bibliographystyle{plainnat} 
\bibliography{bibliography.bib}

\clearpage

\ECSwitch

\begin{APPENDICES}
% Redefining the section format to include "Appendix"
\SingleSpacedXI % current default line spacing

\section{Additional details from Section~\ref{sec:prelim}}
\subsection{Equivalence of MIP formulations} \label{appendix:equivalence}
\citet[Section 3.2]{bertsimas2019exact} present the same mixed-integer programming formulation of \eqref{prob:ranking}, with the exception that the constraints
\begin{align*}
     &x_{i_{k,\ell}} \le \sum_{\ell'=1}^{\ell}  y_{k,\ell'}  && \forall k \in [K], \ell \in [N+1]\tag{\ref{prob:misic:pushup}}
\end{align*}
from the mixed-integer program~\eqref{prob:misic} are replaced in \citet[Section 3.2]{bertsimas2019exact} by
\begin{align*}
      & \sum_{\ell'=\ell+1}^{N+1}  y_{k,\ell'} \le 1 - x_{i_{k,\ell}}  && \forall k \in [K], \ell \in [N+1]\end{align*}
      We observe that these two sets of  constraints are equivalent because constraints~\eqref{prob:misic:atmostone} and \eqref{prob:misic:pushup} imply  that  $\sum_{\ell'=\ell+1}^{N+1} y_{k,\ell'} = 1 - \sum_{\ell'=1}^{\ell} y_{k,\ell'}$ for all $k \in [K]$ and $\ell \in [N+1]$.

\subsection{Benders decomposition method from \cite{bertsimas2019exact}} \label{appx:benders_review}
 As discussed in Section~\ref{sec:benders}, the Benders decomposition method from \cite{bertsimas2019exact} consists of two phases. In the first phase,  Benders decomposition is used to solve the linear programming relaxation of \eqref{prob:outer}. In the second phase, Benders decomposition is used to solve \eqref{prob:outer}, which is warm started with the cuts obtained from the first phase. Each iteration of Phase 1 and Phase 2 of the Benders decomposition method of \cite{bertsimas2019exact} requires solving the linear program~\eqref{prob:dual} to optimality  for each ranking $k$. 

In greater detail, the Benders decomposition method  consists of the following two phases.

\paragraph{Phase 1:} In the first phase of the Benders decomposition method, we solve the linear programming relaxation of \eqref{prob:outer} using  constraint generation. Indeed,  it follows from strong duality that the linear programming relaxation of \eqref{prob:outer}  is equivalent to 
\begin{equation} \label{prob:outer_seminfinite:lp}
\begin{aligned}
\underset{\bx \in \mathcal{X}^c, \bq}{\textnormal{maximize}} \quad & \sum_{k=1}^K  \lambda_k q_k   \\
 \textnormal{subject to} \quad & q_k \le \gamma + \sum_{\ell=1}^{N+1} \left(\alpha_\ell - \beta_\ell \right) x_{i_{k,\ell}} && \forall k \in [K], (\balpha,\bbeta,\gamma) \in \mathcal{D}_k
 \end{aligned}
\end{equation}
where we recall from Section~\ref{sec:benders} that  $\mathcal{D}_k$ is the set of feasible solutions for \eqref{prob:dual}. In each iteration of the constraint generation method in Phase 1, we  solve a relaxation of  \eqref{prob:outer_seminfinite:lp} obtained by, for each ranking $k \in [K]$, replacing  $\mathcal{D}_k$ with a finite subset $\hat{\mathcal{D}}_k \subset \mathcal{D}_k$. Given the optimal solution  $(\hat{\bx},\hat{\bq})$ for that relaxation, we then loop over the rankings, and, for each ranking $k$, we solve   \eqref{prob:dual} to obtain a  vector $(\hat{\balpha}_k, \hat{\bbeta}_k,\hat{\gamma}_k) \in \mathcal{D}_k$ that satisfies $\rho_k(\hat{\bx}) = \hat{\gamma}_k + \sum_{\ell=1}^{N+1} (\hat{\alpha}_{k,\ell} - \hat{\beta}_{k,\ell}) \hat{x}_{i_{k,\ell}}$. If  $\hat{q}_k \le \rho_k(\hat{\bx})$ for every ranking $k \in [K]$, then we conclude that $(\hat{\bx}, \hat{\bq})$ is an optimal solution for  \eqref{prob:outer_seminfinite:lp}.  Otherwise, for each ranking that satisfies  $\hat{q}_k > \rho_k(\hat{\bx})$, we add  $(\hat{\balpha}_k, \hat{\bbeta}_k,\hat{\gamma}_k)$ into $\hat{\mathcal{D}}_k$, and we then continue to the next iteration.\looseness=-1

\paragraph{Phase 2:} In the second phase of the Benders decomposition method, we solve \eqref{prob:outer} using constraint generation, which is warm-started using the constraints generated in Phase 1.   Indeed,  it follows from strong duality that  \eqref{prob:outer}  is equivalent to 
\begin{equation} \label{prob:outer_seminfinite:mip}
\begin{aligned}
\underset{\bx \in \mathcal{X}, \bq}{\textnormal{maximize}} \quad & \sum_{k=1}^K  \lambda_k q_k   \\
 \textnormal{subject to} \quad & q_k \le \gamma + \sum_{\ell=1}^{N+1} \left(\alpha_\ell - \beta_\ell \right) x_{i_{k,\ell}} && \forall k \in [K], (\balpha,\bbeta,\gamma) \in \mathcal{D}_k \end{aligned}
\end{equation}
Similarly as in Phase 1, each iteration of the constraint generation method for Phase 2 consists of solving \eqref{prob:outer_seminfinite:mip}, followed by  solving \eqref{prob:dual}  for each ranking $k$ to see if there is a $(\balpha_k, \bbeta_k,\gamma_k) \in \mathcal{D}_k$ for which the constraints are violated. If so, we add the violated constraints and solve again; otherwise, we conclude that the current integer solution  is an optimal solution for \eqref{prob:outer}. The  purpose of the first phase is to generate constraints to warm start the constraint generation method for solving \eqref{prob:outer_seminfinite:mip}.

\begin{remark} \citet[Appendix~EC.2]{bertsimas2019exact} shows that the constraint generation methods for Phase 1 and Phase 2 are guaranteed to converge after finite numbers of iterations and output optimal solutions for \eqref{prob:outer_seminfinite:lp} and \eqref{prob:outer_seminfinite:mip}, respectively. We note that constraint generation in Phase 2 can implemented using lazy constraints within branch-and-bound; see \citet[Section 4.1]{bertsimas2019exact}. 
\end{remark}

\section{Omitted Details from Section~\ref{sec:exclusion_set} }

\subsection{Proof  of Lemma~\ref{lem:inequality_to_equality}} \label{appx:proof:lem:inequality_to_equality}
\begin{proof}{Proof of Lemma~\ref{lem:inequality_to_equality}.}
Consider any $k,k' \in [K]$, $L \in [L_k]$, and suppose that $\{i_{k,1},\ldots,i_{k,L}\} = \{i_{k',1},\ldots,i_{k',L}\}$. It follows from the fact that $L \in [L_k]$ that $i_{k,\ell} < N+1$ for all $\ell \in \{1,\ldots,L\}$, which implies that $i_{k',\ell} < N+1$ for all $\ell \in \{1,\ldots,L\}$. We thus conclude that $L \le L_{k'}$,  which implies that  \eqref{prob:equality_nonsimple} contains the constraints $\sum_{\ell=1}^{L}  y_{k,\ell} \ge \sum_{\ell=1}^L  y_{k',\ell}$ and   $\sum_{\ell=1}^{L}  y_{k',\ell} \ge \sum_{\ell=1}^L  y_{k,\ell}$. 
\halmos \end{proof}
\subsection{Strength of  formulation \eqref{prob:misic_extra}} \label{appx:example:tighter_relaxation}
In the following example, we show that the linear programming relaxation of \eqref{prob:misic_extra} can be strictly tighter than the linear programming relaxation of \eqref{prob:misic_no_extra}. Indeed, consider an example in which there are $N= 3$ products and $K = 2$ rankings. The two rankings are given by 
\begin{align*}
i_{1,1} &= 1, & i_{2,1} &= 2,\\
 i_{1,2}&= 2, & i_{2,2} &= 1,\\
 i_{1,3} &= 3, & i_{2,3} &= 4,\\
 i_{1,4} &= 4, & i_{2,4} &= 3,
\end{align*}
the probabilities of the two rankings are $\lambda_1 = \lambda_2 = 0.5$, and the revenues of the products are $ r_1 = r_2 = \$100$, and $r_3 = \$150$.  For this example, we observe that  \eqref{prob:misic_no_extra} is 
\begin{equation}\label{prob:misic:example}
\begin{aligned}
\underset{\bx \in \{0,1\}^3,\by}{\textnormal{maximize}} \quad &\left( 100 y_{1,1}  + 100 y_{1,2} + 150 y_{1,3}\right) \times 0.5 + \left( 100 y_{2,1}  + 100 y_{2,2} \right) \times 0.5 \\
 \textnormal{subject to} \quad &\begin{aligned}[t] 
 &y_{1,1} + y_{1,2} + y_{1,3} \le 1\\
  &y_{2,1} + y_{2,2} \le 1\\
  &x_{1} \le y_{1,1}\\
  & x_{2} \le y_{1,1} + y_{1,2}\\
  &x_{3} \le y_{1,1} + y_{1,2} + y_{1,3}\\
  &x_{2} \le y_{2,1}\\
  &x_{1} \le y_{2,1} + y_{2,2}\\
& 0 \le   y_{k,\ell}  \le x_{i_{k,\ell}} &&  \forall k \in \{1,2\}, \ell \in [L_k]
\end{aligned}
\end{aligned}
\end{equation}
We now make several claims.
\begin{claim} \label{claim:tightness_example:1}
The optimal objective value of \eqref{prob:misic:example} is equal to $\$100$ and  $\bx \in \{0,1\}^3$ is an optimal solution for \eqref{prob:misic:example} if and only if $x_{1} = 1$ or $x_{2} = 1$.
\end{claim}
\begin{proof}{Proof of Claim~\ref{claim:tightness_example:1}.}
Suppose that $x_1 = x_2 = 0$. Then we observe that the revenue from the second ranking is equal to \$0 (because the second ranking prefers the no-purchase option 4 to product 3) and that the revenue from the first ranking is at most $\$150$. However, if $x_1 = 1$ or $x_2 = 1$, then the revenue we get from the first ranking is $\$100$ and the revenue we get from the second ranking is $\$100$. Since $\lambda_1 = \lambda_2 = 0.5$, then we conclude that every optimal solution of \eqref{prob:misic:example} must satisfy $x_1 = 1$ or $x_2 = 1$ and that the optimal objective value of \eqref{prob:misic:example} is $\$100$. 
\halmos \end{proof}
\begin{claim} \label{claim:tightness_example:2}
The optimal objective value of the linear programming relaxation of \eqref{prob:misic:example} is equal to $\$112.5$ and  $\bx \in [0,1]^3$ is an optimal solution for the linear programming relaxation of \eqref{prob:misic:example} if and only if $x_{1} = x_{2} = 0.5$ and $x_3 \ge 0.5$. 
\end{claim}
\begin{proof}{Proof of Claim~\ref{claim:tightness_example:2}.}
While this claim can be verified numerically, it is useful to sketch out the intuition. Indeed, we observe for all $\bx \in [0,1]^3$ that every feasible solution for the linear programming relaxation of 
\eqref{prob:misic:example} must satisfy  $y_{2,1} = x_2$. Therefore, we observe that every optimal solution will satisfy $y_{2,2} =  \min \{x_1, 1 - y_{2,1} \}= \min \{x_1, 1 - x_2 \}$, and so the revenue from the second ranking will be
\begin{align*}
100 y_{2,1} + 100 y_{2,2} = 100 x_2 + 100  \min \{x_1, 1 - x_2 \} = 100 \min \{x_1 + x_2, 1 \}
\end{align*}
We further observe that if $x_1 = x_2 = 0.5$ and $x_3 \ge 0.5$, then there exists a feasible solution in which $y_{1,1} = 0.5$, $y_{1,2} = 0$, and $y_{1,3} = 0.5$, and so the revenue from the first ranking will be
\begin{align*}
100 y_{1,1} + 100 y_{1,2}  + 150 y_{1,3} = 125
\end{align*}
We observe that increasing $x_{1} = 0.5$ or $x_{2} = 0.5$ by $\epsilon > 0$ will result in the revenue from the first ranking decreasing by $150\epsilon - 100\epsilon = 50 \epsilon$ and the revenue from the second product being unchanged. Conversely, decreasing $x_{1} = 0.5$ or $x_{2} = 0.5$ by $\epsilon > 0$ will result in the revenue from the first ranking increasing by at most $150 \epsilon - 100 \epsilon = 50 \epsilon$ and the revenue from the second ranking decreasing by $100 \epsilon$. We thus conclude that $x_{1} = 0.5$ and $x_{2} = 0.5$ is optimal, and we readily observe that every optimal solution will satisfy $x_3 \ge 0.5$. 
\halmos \end{proof}

If follows from Claims~\ref{claim:tightness_example:1} and \ref{claim:tightness_example:2} that the linear programming relaxation of \eqref{prob:misic:example} has an optimal objective value ($0.5 \times \$125 + 0.5 \times \$100 = \$112.5$) that is strictly greater than the optimal objective value of \eqref{prob:misic:example} ($\$100$). Moreover, it follows from  Claim~\ref{claim:tightness_example:2} that every optimal solution for the linear programming relaxation of \eqref{prob:misic:example}  is violated by the constraint  $y_{1,1} + y_{1,2} = y_{2,1} + y_{2,2}$. We have thus shown for this example that the linear programming relaxation of \eqref{prob:misic_extra} is tighter than the linear programming relaxation of \eqref{prob:misic_no_extra}.

\subsection{Proof of Theorem~\ref{Theorem:Formulation}} \label{appx:exclusion_set:proof}
\begin{proof}{Proof of Theorem~\ref{Theorem:Formulation}.}

Recall the optimization problem~\eqref{prob:misic_extra},  repeated below for convenience.
\begin{equation} \tag{\ref{prob:misic_extra}}
\begin{aligned}
&\underset{\substack{\textbf{x} \in \{0,1\}^N,\textbf{y}}}{\textnormal{maximize}} && \sum_{k=1}^K  \sum_{\ell=1}^{L_k}  r_{i_{k,\ell}} y_{k,\ell}  \lambda_k \\
&\textnormal{subject to}&&\begin{aligned}[t]
    &\sum_{\ell=1}^{L}  y_{k,\ell} = \sum_{\ell=1}^L  y_{k',\ell} && \forall k,k' \in [K], L \in [L_k]: \{i_{k,1},\ldots,i_{k,L} \} = \{i_{k',1},\ldots,i_{k',L} \}\\
& \sum_{\ell=1}^{L_k} y_{k,\ell} \le 1 && \forall k \in [K]\\
 &x_{i_{k,\ell}} \le \sum_{\ell'=1}^{\ell}  y_{k,\ell'}  && \forall k \in [K], \ell \in [L_k]\\
& 0 \le   y_{k,\ell}  \le x_{i_{k,\ell}} &&  \forall k \in [K], \ell \in [L_k]
\end{aligned}
\end{aligned}
\end{equation}
We observe that we can, without loss of generality, introduce a decision variable $z_{E}$ into \eqref{prob:misic_extra} with the constraint that $z_E = \sum_{\ell=1}^{L} y_{k,\ell}$ for all  $k \in [K]$ and $L \in [L_k]$ that satisfy $E = \{i_{k,1},\ldots,i_{k,L} \}$. Note that adding these decision variables and constraints into \eqref{prob:misic_extra} is without loss of generality because the constraint $\sum_{\ell=1}^{|E|} y_{k,\ell}= \sum_{\ell=1}^{|E|} y_{k',\ell}$ is imposed for all $k,k'$ that satisfy $E =\{ i_{k,1},\ldots,i_{k,|E|} \}=\{ i_{k',1},\ldots,i_{k',|E|} \}$. Therefore, the above optimization problem can be written as
\begin{equation*}
\begin{aligned}
&\underset{\substack{\textbf{x} \in \{0,1\}^N,\textbf{y}, \bz}}{\textnormal{maximize}} && \sum_{k=1}^K  \sum_{\ell=1}^{L_k}  r_{i_{k,\ell}} y_{k,\ell}  \lambda_k \\
&\textnormal{subject to}&&\begin{aligned}[t]
    &\sum_{\ell=1}^{L}  y_{k, \ell} = z_{\{i_{k,1},\ldots,i_{k,L} \}} && \forall k \in [K], L \in [L_k]\\
& \sum_{\ell=1}^{L_k} y_{k,\ell} \le 1 && \forall k \in [K]\\
 &x_{i_{k,\ell}} \le \sum_{\ell'=1}^{\ell}  y_{k,\ell'}  && \forall k \in [K], \ell \in [L_k]\\
& 0 \le   y_{k,\ell}  \le x_{i_{k,\ell}} &&  \forall k \in [K], \ell \in [L_k]
\end{aligned}
\end{aligned}
\end{equation*}
We next observe that the nonnegativity of the decision variables $y_{k,\ell}$ and the constraints $\sum_{\ell=1}^{L_k} y_{k,\ell} \le 1$ together imply that $\sum_{\ell=1}^{L} y_{k,\ell} \le 1$ for all $k \in [K]$ and $L \in [L_k]$, and so the above problem is equivalent to 
\begin{equation*}
\begin{aligned}
&\underset{\substack{\textbf{x} \in \{0,1\}^N,\textbf{y}, \bz}}{\textnormal{maximize}} && \sum_{k=1}^K  \sum_{\ell=1}^{L_k}  r_{i_{k,\ell}} y_{k,\ell}  \lambda_k \\
&\textnormal{subject to}&&\begin{aligned}[t]
    &\sum_{\ell=1}^{L}  y_{k, \ell} = z_{\{i_{k,1},\ldots,i_{k,L} \}} && \forall k \in [K], L \in [L_k]\\
& \sum_{\ell=1}^{L} y_{k,\ell} \le 1 && \forall k \in [K], L \in [L_k]\\
 &x_{i_{k,\ell}} \le \sum_{\ell'=1}^{\ell}  y_{k,\ell'}  && \forall k \in [K], \ell \in [L_k]\\
& 0 \le   y_{k,\ell}  \le x_{i_{k,\ell}} &&  \forall k \in [K], \ell \in [L_k]
\end{aligned}
\end{aligned}
\end{equation*}
It follows from substitution that the above problem is equivalent to 
\begin{equation} \label{prob:misic_extra:2a}
\begin{aligned}
&\underset{\substack{\textbf{x} \in \{0,1\}^N,\textbf{y}, \bz}}{\textnormal{maximize}} && \sum_{k=1}^K  \sum_{\ell=1}^{L_k}  r_{i_{k,\ell}} y_{k,\ell}  \lambda_k \\
&\textnormal{subject to}&&\begin{aligned}[t]
    &\sum_{\ell=1}^{L}  y_{k, \ell} = z_{\{i_{k,1},\ldots,i_{k,L} \}} && \forall k \in [K], L \in [L_k]\\
& z_{\left \{i_{k,1},\ldots,i_{k,L} \right\}}\le 1 && \forall k \in [K], L \in [L_k]\\
 &x_{i_{k,\ell}} \le z_{\{i_{k,1},\ldots,i_{k,\ell} \}}  && \forall k \in [K], \ell \in [L_k]\\
& 0 \le   y_{k,\ell}  \le x_{i_{k,\ell}} &&  \forall k \in [K], \ell \in [L_k]
\end{aligned}
\end{aligned}
\end{equation}
We also observe that the constraints 
\begin{align*}
    \sum_{\ell=1}^{L}  y_{k, \ell} = z_{\{i_{k,1},\ldots,i_{k,L} \}} \quad \forall k \in [K], L \in [L_k]
\end{align*}
can be rewritten equivalently as 
\begin{align*}
    z_{\{i_{k,1},\ldots,i_{k,\ell} \}} - z_{\{i_{k,1},\ldots,i_{k,\ell-1} \}} &= y_{k,\ell} \quad \forall k \in [K], \ell \in [L_k]\\
      z_{\emptyset} &= 0 
    \end{align*}
Therefore, \eqref{prob:misic_extra:2a} is equivalent to
\begin{equation} \label{prob:misic_extra:3a}
\begin{aligned}
&\underset{\substack{\textbf{x} \in \{0,1\}^N,\textbf{z}}}{\textnormal{maximize}} && \sum_{k=1}^K  \sum_{\ell=1}^{L_k}  r_{i_{k,\ell}} \left(     z_{\{i_{k,1},\ldots,i_{k,\ell} \}} - z_{\{i_{k,1},\ldots,i_{k,\ell-1} \}} \right)  \lambda_k \\
&\textnormal{subject to}&&\begin{aligned}[t]
& z_{\left \{i_{k,1},\ldots,i_{k,L} \right\}}\le 1 && \forall k \in [K], L \in [L_k]\\
 &x_{i_{k,\ell}} \le z_{\{i_{k,1},\ldots,i_{k,\ell} \}}  && \forall k \in [K], \ell \in [L_k]\\
& 0 \le       z_{\{i_{k,1},\ldots,i_{k,\ell} \}} - z_{\{i_{k,1},\ldots,i_{k,\ell-1} \}}  \le x_{i_{k,\ell}} &&  \forall k \in [K], \ell \in [L_k]\\
&z_\emptyset = 0
\end{aligned}
\end{aligned}
\end{equation}
It follows from the definitions of $\mathscr{E}$ and  $\mathscr{P}$ that the constraints of \eqref{prob:misic_extra:3a} can be written equivalently  as
\begin{equation} \label{prob:misic_extra:4a}
\begin{aligned}
&\underset{\substack{\textbf{x} \in \{0,1\}^N,\textbf{z}}}{\textnormal{maximize}} && \sum_{k=1}^K  \sum_{\ell=1}^{L_k}  r_{i_{k,\ell}} \left(     z_{\{i_{k,1},\ldots,i_{k,\ell} \}} - z_{\{i_{k,1},\ldots,i_{k,\ell-1} \}} \right)  \lambda_k \\
&\textnormal{subject to}&&\begin{aligned}[t]
& z_{E}\le 1 && \forall E \in \mathscr{E}\\
 &x_{i} \le z_{E \cup \{i\}} && \forall (E,i) \in \mathscr{P}\\
& 0 \le       z_{E \cup \{i\}} - z_{E}  \le x_{i} &&  \forall (E,i) \in \mathscr{P}\\
&z_\emptyset = 0
\end{aligned}
\end{aligned}
\end{equation}
    To conclude the proof of Theorem~\ref{Theorem:Formulation}, we reformulate the objective function of \eqref{prob:misic_extra:4a}. Indeed, we observe that the objective function of \eqref{prob:misic_extra:4a} satisfies 
    \begin{align*}
   & \sum_{k=1}^K  \sum_{\ell=1}^{L_k}  r_{i_{k,\ell}} \left(     z_{\{i_{k,1},\ldots,i_{k,\ell} \}} - z_{\{i_{k,1},\ldots,i_{k,\ell-1} \}} \right)  \lambda_k \\
       & =   \sum_{(E,i) \in \mathscr{P}} \left(\sum_{k=1}^K  \sum_{\ell=1}^{L_k}  r_{i_{k,\ell}} \left(     z_{\{i_{k,1},\ldots,i_{k,\ell} \}} - z_{\{i_{k,1},\ldots,i_{k,\ell-1} \}} \right)  \lambda_k \mathbb{I} \left \{E = \{i_{k,1},\ldots,i_{k,\ell-1} \} \text{ and } i = i_{k,\ell} \right \}   \right)\\
             & =   \sum_{(E,i) \in \mathscr{P}} \left(\sum_{k=1}^K  \sum_{\ell=1}^{L_k}  r_{i_{k,\ell}} \left(     z_{E \cup \{i\}} - z_{E} \right)  \lambda_k \mathbb{I} \left \{E = \{i_{k,1},\ldots,i_{k,\ell-1} \} \text{ and } i = i_{k,\ell} \right \}   \right)\\             
             & =   \sum_{(E,i) \in \mathscr{P}}  r_{i_{k,\ell}} \left(     z_{E \cup \{i\}} - z_{E} \right)  \left(\sum_{k=1}^K  \sum_{\ell=1}^{L_k}  \mathbb{I} \left \{E = \{i_{k,1},\ldots,i_{k,\ell-1} \} \text{ and } i = i_{k,\ell} \right \}  \lambda_k  \right)\\
             & =   \sum_{(E,i) \in \mathscr{P}}  r_{i_{k,\ell}} \left(     z_{E \cup \{i\}} - z_{E} \right)  \left(\sum_{k=1}^K   \mathbb{I} \left \{E = \{i_{k,1},\ldots,i_{k,|E|} \} \text{ and } i = i_{k,|E|+1} \right \}  \lambda_k  \right)\\
             & =   \sum_{(E,i) \in \mathscr{P}}  r_{i_{k,\ell}} \left(     z_{E \cup \{i\}} - z_{E} \right)  \lambda_{E,i}
    \end{align*}
    where first equality follows from the definition of $\mathscr{P}$, the second, third, and fourth equalities follow from algebra, and the fifth equality is the definition of $\lambda_{E,i}$. That concludes our proof of Theorem~\ref{Theorem:Formulation}. 
\Halmos \end{proof}

\subsection{Proof of Proposition~\ref{prop:integrality_budget_one}} \label{appx:proof:prop:integrality_budget_one}
\begin{proof}{Proof of Proposition~\ref{prop:integrality_budget_one}.}
We recall that the exclusion set formulation~\eqref{prob:exclusionset} is equivalent to \eqref{prob:misic_extra} (Theorem~\ref{Theorem:Formulation}), and we recall from Section~\ref{sec:exclusionset:step2} that \eqref{prob:misic_extra} is stronger than \eqref{prob:misic_no_extra}. Therefore, to  prove Proposition~\ref{prop:integrality_budget_one}, it suffices to show that \eqref{prob:misic_no_extra} is integral in the case where we have a cardinality constraint that is equal to one. Indeed, consider the formulation~\eqref{prob:misic_no_extra} with this cardinality constraint:
\begin{equation}\label{prob:misic_no_extra:cardinality_one}
\begin{aligned}
\underset{\bx \in \{0,1\}^N,\by}{\textnormal{maximize}} \quad & \sum_{k=1}^K 
 \sum_{\ell=1}^{L_k} r_{i_{k,\ell}} y_{k,\ell}  \lambda_k \\
 \textnormal{subject to} \quad & \sum_{\ell=1}^{L_k} y_{k,\ell} \le 1 && \forall k \in [K] \\
 &x_{i_{k,\ell}} \le \sum_{\ell'=1}^{\ell}  y_{k,\ell'}  && \forall k \in [K], \ell \in [L_k] \\
& 0 \le   y_{k,\ell}  \le x_{i_{k,\ell}} &&  \forall k \in [K], \ell \in [L_k]\\
&\sum_{i=1}^N x_i \le 1
\end{aligned}
\end{equation}
We now consider the linear programming relaxation of \eqref{prob:misic_no_extra:cardinality_one}. Indeed, it follows from the fact that $r_1,\ldots,r_N > 0$ that there exists an optimal solution for the linear programming relaxation that satisfies $\sum_{i=1}^N x_i = 1$. Now consider any $\bx \in [0,1]^N$ that satisfies $\sum_{i=1}^N x_i = 1$. It follows from the constraints $y_{k,\ell} \le x_{i_{k,\ell}}$ and from the fact that $r_1,\ldots,r_N > 0$ that an upper bound on \eqref{prob:misic_no_extra:cardinality_one} is achieved by letting $y_{k,\ell} = x_{i_{k,\ell}}$ for all $k \in [K]$ and $\ell \in [L_k]$, i.e., the optimal objective value of the linear programming relaxation of \eqref{prob:misic_no_extra:cardinality_one} is upper bounded by
\begin{align*}
    \sum_{k=1}^K \sum_{\ell=1}^{L_k} r_{i_{k,\ell}} x_{i_{k,\ell}} \lambda_k
\end{align*}
We also observe that letting $y_{k,\ell} = x_{i_{k,\ell}}$ for each ranking $k \in [K]$ and $\ell \in [L_k]$ is a {feasible solution} for the linear programming relaxation of \eqref{prob:misic_no_extra:cardinality_one}, since 
\begin{align*}
    &\sum_{\ell=1}^{L_k} y_{k,\ell} =  \sum_{\ell=1}^{L_k} x_{i_{k,\ell}} \le 1 && \forall k \in [K]\\
    &x_{i_{k,\ell}} \le \sum_{\ell'=1}^\ell x_{i_{k,\ell'}} = \sum_{\ell'=1}^\ell y_{k,\ell'}  && \forall k \in [K], \ell \in [L_k]\\
    &0 \le y_{k,\ell} = x_{i_{k,\ell}} \le 1 &&\forall k \in [K], \ell \in [L_k]
\end{align*}
Therefore, the optimal objective value of \eqref{prob:misic_no_extra:cardinality_one} is lower bounded by 
\begin{align*}
    \sum_{k=1}^K \sum_{\ell=1}^{L_k} r_{i_{k,\ell}} x_{i_{k,\ell}} \lambda_k
\end{align*}
Since the above reasoning holds for all $\bx \in [0,1]^N$ that satisfy $\sum_{i=1}^N x_i = 1$, we conclude that the linear programming relaxation of \eqref{prob:misic_no_extra:cardinality_one} is equivalent to
\begin{equation}\label{prob:misic_no_extra:cardinality_one_reform}
\begin{aligned}
\underset{\bx \in [0,1]^N: \sum_{i=1}^N x_i \le 1}{\textnormal{maximize}} \quad & \sum_{k=1}^K 
 \sum_{\ell=1}^{L_k} r_{i_{k,\ell}} x_{i_{k,\ell}} \lambda_k 
 \end{aligned}
\end{equation}
Because there exists an optimal solution for \eqref{prob:misic_no_extra:cardinality_one_reform} that is integral, we conclude that the linear programming relaxations of \eqref{prob:exclusionset} and \eqref{prob:misic_no_extra} are equivalent and integral when we have a cardinality constraint of a single product.  \halmos \end{proof}
\section{Proof of Theorem~\ref{thm:reform}}\label{appx:proof_reformulation}
    Our proof of Theorem~\ref{thm:reform} is split into the following two theorems.  \begin{theorem} \label{thm:reform_easy}
      If $\bdelta \in \Delta_k$, then there exists $(\balpha,\bbeta,\gamma) \in \mathcal{D}_k$ that satisfies 
          \begin{align*}
              J_k(\bx,\bdelta) = \gamma + \sum_{\ell=1}^{N+1} \left( \alpha_\ell - \beta_\ell \right) x_{i_{k,\ell}}  \quad \forall \bx \in \mathcal{X}^c
          \end{align*}
\end{theorem}
\begin{theorem}\label{thm:reform_hard}
            If $(\balpha, \bbeta, \gamma) \in \mathcal{D}_k$, then there exists  $\bdelta \in \Delta_k$  that satisfies 
          \begin{align*}
              J_k(\bx,\bdelta) \le \gamma + \sum_{\ell=1}^{N+1} \left( \alpha_\ell - \beta_\ell \right) x_{i_{k,\ell}}  \quad \forall \bx \in \mathcal{X}^c
          \end{align*} 
          \end{theorem}
%It is straightforward to see  that Theorem~\ref{thm:reform} follows from Theorems~\ref{thm:reform_easy} and \ref{thm:reform_hard}. 
The proofs of Theorems~\ref{thm:reform_easy} and \ref{thm:reform_hard} are  found in Appendices~\ref{appx:reform_easy} and \ref{appx:reform_hard}. For the sake of completeness, we show in Appendix~\ref{appx:reform_bringing_together} that Theorem~\ref{thm:reform}  follows from Theorems~\ref{thm:reform_easy} and \ref{thm:reform_hard}. \looseness=-1

\subsection{Proof of Theorem~\ref{thm:reform_easy}} \label{appx:reform_easy} 
Let  $\bar{r}^* \triangleq \max_{i \in [N+1]} r_i$ denote the maximum revenue among all products.
Our proof of Theorem~\ref{thm:reform_easy} is constructive. Indeed, consider any  $\bdelta \in \Delta_k$, and define $(\balpha,\bbeta,\gamma)$ as
\begin{subequations} \label{line:reform_easy:defn}
\begin{align}
   \alpha_\ell&= \begin{cases}
        \max \left \{0, r_{i_{k,\ell}}- \delta_\ell \right \}, &\text{if } \ell \in\{1,\ldots,L_k+1\},\\
      0,&\text{if } \ell \in\{L_k+2,\ldots,N+1\}
    \end{cases}\\
    \beta_\ell &= \begin{cases}
        \delta_{\ell+1} - \delta_\ell,&\text{if } \ell \in \{1,\ldots,L_k\},\\
        \bar{r}^* - \delta_{L_k+1},&\text{if } \ell = L_k+1,\\
        0,&\text{if } \ell \in \{L_k+2,\ldots,N+1\}
    \end{cases}\\
    \gamma &= \bar{r}^*
\end{align}
\end{subequations}
In what follows, we complete the proof of Theorem~\ref{thm:reform_easy} by  showing that the solution $(\balpha,\bbeta,\gamma)$ constructed in \eqref{line:reform_easy:defn} satisfies $(\balpha,\bbeta,\gamma) \in\mathcal{D}_k$  as well as  satisfies $            J_k(\bx,\bdelta) = \gamma + \sum_{\ell=1}^{N+1} \left( \alpha_\ell - \beta_\ell \right) x_{i_{k,\ell}}$ for all $\bx \in \mathcal{X}^c$. We prove these two results in the following Propositions~\ref{prop:reform_easy:feas} and \ref{prop:reform_easy:obj}.   
\begin{proposition}\label{prop:reform_easy:feas}
$(\balpha,\bbeta,\gamma) \in \mathcal{D}_k$. 
\end{proposition}
\begin{proof}{Proof of Proposition~\ref{prop:reform_easy:feas}.}
Our proof consists of showing that  $(\balpha,\bbeta,\gamma)$ satisfies the constraints of the linear program~\eqref{prob:dual}. Indeed, for all $\ell \in \{1,\ldots,L_k +1\}$, we have
\begin{align}
    \gamma + \alpha_\ell - \sum_{\ell'=\ell}^{N+1} \beta_{\ell'} &=   \gamma + \alpha_\ell - \sum_{\ell'=\ell}^{L_k} \beta_{\ell'} - \beta_{L_k+1} -  \sum_{\ell'=L_k+2}^{N+1} \beta_{\ell'} \notag \\
    &=   \bar{r}^* + \max \left \{ 0, r_{i_{k,\ell}} - \delta_\ell \right \} - \sum_{\ell'=\ell}^{L_k} \left( \delta_{\ell'+1} - \delta_{\ell'} \right) - \left(\bar{r}^* - \delta_{L_k+1} \right)  -  \sum_{\ell'=L_k+2}^{N+1} 0 \notag \\
      &=   \delta_\ell + \max \left \{ 0, r_{i_{k,\ell}} - \delta_\ell \right \} \notag \\
      &\ge r_{i_{k,\ell}}\label{line:wolves}
\end{align}
The first equality follows from algebra. The second equality follows from our  construction of $(\balpha,\bbeta,\gamma)$. The third equality from canceling terms. The  inequality follows from algebra.   Moreover, for $\ell \in \{L_k+2,\ldots,N+1\}$, we have
\begin{align}
    \gamma + \alpha_\ell - \sum_{\ell'=\ell}^{N+1} \beta_{\ell'} 
    &=   \bar{r}^* + 0  -   \sum_{\ell'=\ell}^{N+1} 0 \ge r_{i_{k,\ell}} \label{line:wolves2}
\end{align}
The first equality follows from our construction of $(\balpha,\bbeta,\gamma)$, and the inequality follows from the definition of $\bar{r}^*$. 
Combining \eqref{line:wolves} and \eqref{line:wolves2}, we have shown that 
\begin{align*}
     \gamma + \alpha_\ell - \sum_{\ell'=\ell}^{N+1} \beta_{\ell'} \ge r_{i_{k,\ell}} \quad \forall \ell \in \{1,\ldots,N+1\}
\end{align*}
It is straightforward to show that $\alpha_\ell\ge 0$ for all $\ell \in \{1,\ldots,N+1\}$. Moreover, for all $\ell \in \{1,\ldots,N+1\}$,
\begin{align*}
    \beta_\ell &\ge \begin{cases}
        \delta_{\ell} - \delta_\ell,&\text{if } \ell \in \{1,\ldots,L_k\},\\
        \bar{r}^* - \delta_{L_k+1},&\text{if } \ell = L_k+1,\\
        0,&\text{if } \ell \in \{L_k+2,\ldots,N+1\}
    \end{cases}\\
    &\ge 0
\end{align*}
where the first inequality follows from the fact that $\bdelta \in \Delta_k$, and the second inequality follows from algebra and from the definition of $\bar{r}^*$. We have thus shown that $(\balpha,\bbeta,\gamma)$ satisfies the constraints of the linear program~\eqref{prob:dual}, which implies that $(\balpha,\bbeta,\gamma) \in \mathcal{D}_k$. Our proof of Proposition~\ref{prop:reform_easy:feas} is thus complete. 
\halmos \end{proof}
\begin{proposition}\label{prop:reform_easy:obj}
$J_k(\bx,\bdelta) = \gamma + \sum_{\ell=1}^{N+1} \left( \alpha_\ell - \beta_\ell \right) x_{i_{k,\ell}}$ for all $\bx \in \mathcal{X}^c$
\end{proposition}
\begin{proof}{Proof of Proposition~\ref{prop:reform_easy:obj}.}
Consider any $\bx \in \mathcal{X}^c$. For each $\ell \in \{1,\ldots,L_k\}$, it follows from our construction of $(\balpha,\bbeta,\gamma)$ that
\begin{align*}
    \left( \alpha_\ell - \beta_\ell \right) x_{i_{k,\ell}}&= \left( \max \left \{0, r_{i_{k,\ell}} - \delta_\ell \right \} - \left( \delta_{\ell+1} - \delta_\ell \right) \right) x_{i_{k,\ell}} \end{align*}
For the case of $\ell = L_k+1$, we have 
\begin{align*}
    \left( \alpha_\ell - \beta_\ell \right)x_{i_{k,\ell}} &=   \left( \max \left \{0, r_{i_{k,L_k+1}} - \delta_{L_k+1} \right \} - \left(\bar{r}^* - \delta_{L_k+1}  \right)  \right)x_{i_{k,L_k+1}} \\
    &=   \left( \max \left \{0, r_{N+1} - \delta_{L_k+1} \right \} - \left(\bar{r}^* - \delta_{L_k+1}  \right)  \right)x_{N+1} \\
 &=   \left( \max \left \{0, 0 - \delta_{L_k+1} \right \} - \left(\bar{r}^* - \delta_{L_k+1}  \right)  \right)1 \\
 &=   \delta_{L_k+1} - \bar{r}^* 
    \end{align*}
The first equality follows from our construction of $(\balpha,\bbeta,\gamma)$. The second equality follows from algebra and from the fact that $i_{k,L_k+1} = N+1$. The third equality follows from the fact that $r_{N+1} = 0$ and from the fact that $\bx \in \mathcal{X}^c$ (which implies that $x_{N+1} = 1$). The fourth equality follows from the fact that $\bdelta \in \Delta_k$ (which implies that $\delta_{L_k+1} \ge 0$). For each  $\ell \in \left \{ L_k+2,\ldots,N+2\right \}$, it follows from our construction of $(\balpha,\bbeta,\gamma)$ that 
\begin{align*}
    \left( \alpha_\ell - \beta_\ell \right)x_{i_{k,\ell}} &=   \left( 0 - 0 \right)x_{i_{k,\ell}} = 0
    \end{align*}
We have thus shown that 
\begin{align*}
   \gamma + \sum_{\ell=1}^{N+1} \left( \alpha_\ell - \beta_\ell \right) x_{i_{k,\ell}}    &=    \gamma + \sum_{\ell=1}^{L_k} \left( \alpha_\ell - \beta_\ell \right) x_{i_{k,\ell}}  + \left( \alpha_{L_k+1} - \beta_{L_k+1} \right) x_{i_{k,L_k+1}} + \sum_{\ell=L_k+2}^{N+1} \left( \alpha_\ell - \beta_\ell \right) x_{i_{k,\ell}}    \\
   &= \bar{r}^* + \sum_{\ell=1}^{L_k} \left( \max \left \{0, r_{i_{k,\ell}} - \delta_\ell \right \} - \left( \delta_{\ell+1} - \delta_\ell \right)\right)  x_{i_{k,\ell}} + \left( \delta_{L_k+1} - \bar{r}^* \right) + \sum_{\ell=L_k+2}^{N+1} 0 \\
  &= \delta_{L_k+1} + \sum_{\ell=1}^{L_k} \left( \max \left \{0, r_{i_{k,\ell}} - \delta_\ell \right \} - \left( \delta_{\ell+1} - \delta_\ell \right)\right)  x_{i_{k,\ell}} 
\end{align*}
which completes our proof of Proposition~\ref{prop:reform_easy:obj}. 
\halmos \end{proof}

\subsection{Proof of Theorem~\ref{thm:reform_hard}} \label{appx:reform_hard} 

Our proof of Theorem~\ref{thm:reform_hard} is constructive. Indeed, consider any  $(\balpha,\bbeta,\gamma) \in \mathcal{D}_k$, and define $\bdelta$ as
\begin{align}
    \delta_\ell &= \max \left \{ \min \left \{ \gamma - \sum_{\ell'=\ell}^{N+1} \beta_{\ell'}, \bar{r}_k \right \}, 0 \right \}   \quad \forall \ell \in \{1,\ldots,L_k+1\}\label{line:reform_hard:defn}
\end{align}
In what follows, we complete the proof of Theorem~\ref{thm:reform_hard} by  showing that the solution $\bdelta$ constructed in \eqref{line:reform_hard:defn} satisfies $\bdelta \in\Delta_k$  as well as  satisfies $            J_k(\bx,\bdelta) \le \gamma + \sum_{\ell=1}^{N+1} \left( \alpha_\ell - \beta_\ell \right) x_{i_{k,\ell}}$ for all $\bx \in \mathcal{X}^c$. We prove these two results in the following Propositions~\ref{prop:reform_hard:feas} and \ref{prop:reform_hard:obj}.   
\begin{proposition} \label{prop:reform_hard:feas}
    $\bdelta \in \Delta_k$. 
\end{proposition}
\begin{proof}{Proof of Proposition~\ref{prop:reform_hard:feas}.}
It follows immediately from \eqref{line:reform_hard:defn}  that $\delta_1 \ge 0$ and that $\delta_{L_k+1} \le \bar{r}_k$. Moreover, for each $\ell \in \{1,\ldots,L_k\}$, it follows from \eqref{line:reform_hard:defn} that
\begin{align*}
    \delta_{\ell+1} &= \max \left \{ \min \left \{ \gamma  - \sum_{\ell'=\ell}^{N+1} \beta_{\ell'} + \beta_{\ell+1}, \bar{r}_k \right \}, 0 \right \}  \ge  \max \left \{ \min \left \{ \gamma  - \sum_{\ell'=\ell}^{N+1} \beta_{\ell'}, \bar{r}_k \right \}, 0 \right \}   = \delta_{\ell}
\end{align*}
where the inequality follows from the fact that $(\balpha,\bbeta,\gamma) \in \mathcal{D}_k$, which implies that $\beta_{\ell+1} \ge 0$. We have thus shown that $\bdelta \in \Delta_k$.
\halmos \end{proof}
\begin{proposition} \label{prop:reform_hard:obj}
      $J_k(\bx,\bdelta) \le \gamma + \sum_{\ell=1}^{N+1} \left( \alpha_\ell - \beta_\ell \right) x_{i_{k,\ell}}$ for all  $\bx \in \mathcal{X}^c$. 
\end{proposition}
\begin{proof}{Proof of Proposition~\ref{prop:reform_hard:obj}.}
Consider any $\bx \in \mathcal{X}^c$. It follows from fact that $(\balpha,\bbeta,\gamma) \in \mathcal{D}_k$ that
\begin{align}
\gamma + \sum_{\ell=1}^{N+1} \left( \alpha_\ell - \beta_\ell \right) x_{i_{k,\ell}} \ge \gamma + \sum_{\ell=1}^{N+1} \left( \max \left \{ r_{i_{k,\ell}} - \left( \gamma - \sum_{\ell'=\ell}^{N+1} \beta_{\ell'} \right) , 0 \right \} - \beta_\ell \right) x_{i_{k,\ell}} \label{line:last_call} 
\end{align}
 For notational convenience, define
          \begin{align*}
              \zeta_\ell \triangleq \gamma - \sum_{\ell'=\ell}^{N+1} \beta_{\ell'} \quad \forall \ell \in \{1,\ldots,N+2\}
          \end{align*}
It follows from the above definition that $\beta_{\ell} = \zeta_{\ell+1} - \zeta_\ell$ for all $\ell \in \{1,\ldots,N+1\}$ and that
\begin{align}
    \delta_\ell &= \max \left \{ \min \left \{ \zeta_\ell, \bar{r}_k \right \},0 \right \} \quad \forall \ell \in \{1,\ldots,L_k+1\}\label{line:alt_defn_delta}
\end{align}
In view of the above notation and \eqref{line:last_call}, to complete the proof of Proposition~\ref{prop:reform_hard:obj},  it suffices for us to show that 
                    \begin{align} \label{line:reform_hard:main_inequality}
\gamma + \sum_{\ell=1}^{N+1} \left( \max \left \{ r_{i_{k,\ell}} - \zeta_\ell, 0 \right \} - \left( \zeta_{\ell+1} - \zeta_\ell \right) \right) x_{i_{k,\ell}} \ge J_k(\bx,\bdelta)
          \end{align} 
Our proof of \eqref{line:reform_hard:main_inequality}  will make use of the following Claims~\ref{claim:back_from_travel_1} and \ref{claim:back_from_travel_2}. 
          \begin{claim} \label{claim:back_from_travel_1}
\begin{align*}
&\sum_{\ell=1}^{L_k} \left( \max \left \{r_{i_{k,\ell}} - \zeta_\ell, 0 \right \} - \left( \zeta_{\ell+1} - \zeta_{\ell} \right) \right)x_{i_{k,\ell}} -  \sum_{\ell=1}^{L_k} \left( \max \left \{r_{i_{k,\ell}} - \delta_\ell, 0 \right \} - \left(\delta_{\ell+1} - \delta_\ell \right) \right) x_{i_{k,\ell}} \\
&\ge - \max \left \{\zeta_{L_k+1} - \bar{r}_k, 0 \right \} \end{align*}  
          \end{claim}
          \begin{proof}{Proof of Claim~\ref{claim:back_from_travel_1}.}
      Consider any $\ell \in \{1,\ldots,L_k\}$. We observe that
          \begin{align}
             & \left( \max \left \{r_{i_{k,\ell}} - \zeta_\ell, 0 \right \} - \left(\zeta_{\ell+1} - \zeta_\ell \right)  \right) -  \left( \max \left \{r_{i_{k,\ell}} - \delta_\ell, 0 \right \} - \left(\delta_{\ell+1} - \delta_\ell \right) \right)\notag  \\
              &= \begin{cases}
\left( \max \left \{r_{i_{k,\ell}} - \zeta_\ell, 0 \right \} - \left(\zeta_{\ell+1} - \zeta_\ell \right) \right) -  \left( \max \left \{r_{i_{k,\ell}} - \bar{r}_k, 0 \right \} - \left(\bar{r}_k - \bar{r}_k\right) \right) ,&\text{if } \zeta_\ell > \bar{r}_k,\\
\left( \max \left \{r_{i_{k,\ell}} - \zeta_\ell, 0 \right \} - \left(\zeta_{\ell+1} - \zeta_\ell \right) \right) -  \left( \max \left \{r_{i_{k,\ell}} - \zeta_\ell, 0 \right \} - \left(\delta_{\ell+1} - \zeta_\ell \right) \right) ,&\text{if } 0 \le \zeta_\ell \le \bar{r}_k,\\
\left( \max \left \{r_{i_{k,\ell}} - \zeta_\ell, 0 \right \} - \left(\zeta_{\ell+1} - \zeta_\ell \right) \right) -  \left( \max \left \{r_{i_{k,\ell}} - 0, 0 \right \} - \left(\delta_{\ell+1} - 0 \right) \right) ,&\text{if } \zeta_\ell < 0
              \end{cases}\notag \\
 &= \begin{cases}
 -\left(\zeta_{\ell+1} - \zeta_\ell\right) ,&\text{if } \zeta_\ell > \bar{r}_k,\\
-\left( \zeta_{\ell+1} - \delta_{\ell+1} \right),&\text{if } 0 \le \zeta_\ell \le \bar{r}_k,\\
-\left( \zeta_{\ell+1} - \delta_{\ell+1} \right) ,&\text{if } \zeta_\ell < 0
\end{cases} \notag \\
 &= \begin{cases}
-\left( \zeta_{\ell+1} - \zeta_\ell \right),&\text{if } \zeta_\ell > \bar{r}_k,\\
-\left( \zeta_{\ell+1} - \delta_{\ell+1} \right) ,&\text{if } \zeta_\ell \le \bar{r}_k
              \end{cases}\notag \\
 &= \begin{cases}
-\left( \zeta_{\ell+1} - \zeta_\ell \right),&\text{if } \zeta_\ell > \bar{r}_k,\\
-\left( \zeta_{\ell+1} - \delta_{\ell+1} \right) ,&\text{if } \zeta_\ell \le \bar{r}_k < \zeta_{\ell+1},\\
-\left( \zeta_{\ell+1} - \delta_{\ell+1} \right) ,&\text{if } \zeta_\ell \le \zeta_{\ell+1} \le \bar{r}_k
\end{cases}\notag \\
 &= \begin{cases}
-\left( \zeta_{\ell+1} - \zeta_\ell \right),&\text{if } \zeta_\ell > \bar{r}_k,\\
-\left( \zeta_{\ell+1} - \bar{r}_k \right) ,&\text{if } \zeta_\ell \le \bar{r}_k < \zeta_{\ell+1},\\
- \left( \zeta_{\ell+1} - \max \left \{ \zeta_{\ell+1}, 0 \right \} \right) ,&\text{if } \zeta_\ell \le \zeta_{\ell+1} \le \bar{r}_k
\end{cases}\notag \\
 &= \begin{cases}
-\left( \zeta_{\ell+1} - \zeta_\ell \right),&\text{if } \zeta_\ell > \bar{r}_k,\\
-\left( \zeta_{\ell+1} - \bar{r}_k \right) ,&\text{if } \zeta_\ell \le \bar{r}_k < \zeta_{\ell+1},\\
 \max \left \{0,- \zeta_{\ell+1} \right \}  ,&\text{if } \zeta_\ell \le \zeta_{\ell+1} \le \bar{r}_k
\end{cases}\notag \\
 &\ge \begin{cases}
-\left( \zeta_{\ell+1} - \zeta_\ell \right),&\text{if } \zeta_\ell > \bar{r}_k,\\
-\left( \zeta_{\ell+1} - \bar{r}_k \right) ,&\text{if } \zeta_\ell \le \bar{r}_k < \zeta_{\ell+1},\\
0,&\text{if } \zeta_\ell \le \zeta_{\ell+1} \le \bar{r}_k
\end{cases}\notag \\
&= - \left(\max \left \{ \zeta_{\ell+1} - \bar{r}_k, 0 \right \} - \max \left \{ \zeta_{\ell} - \bar{r}_k, 0 \right \} \right)  \label{line:back_from_travel}
          \end{align}
The first equality follows from \eqref{line:alt_defn_delta}. The second equality follows from algebra and the definition of $\bar{r}_k$. The third and fourth equalities follow from algebra. The fifth equality follows from \eqref{line:alt_defn_delta}. The sixth equality, the inequality, and the seventh equality follow from algebra. Using the above inequality, we have
\begin{align*}
    &\sum_{\ell=1}^{L_k} \left( \max \left \{r_{i_{k,\ell}} - \zeta_\ell, 0 \right \} - \left( \zeta_{\ell+1} - \zeta_\ell \right) \right)x_{i_{k,\ell}} -  \sum_{\ell=1}^{L_k} \left( \max \left \{r_{i_{k,\ell}} - \delta_\ell, 0 \right \} - \left(\delta_{\ell+1} - \delta_\ell \right) \right) x_{i_{k,\ell}} \\
    &\ge - \sum_{\ell=1}^{L_k} \left(\max \left \{ \zeta_{\ell+1} - \bar{r}_k, 0 \right \} - \max \left \{ \zeta_{\ell} - \bar{r}_k, 0 \right \} \right)x_{i_{k,\ell}}\\
&\ge - \sum_{\ell=1}^{L_k} \left(\max \left \{ \zeta_{\ell+1} - \bar{r}_k, 0 \right \} - \max \left \{ \zeta_{\ell} - \bar{r}_k, 0 \right \} \right) 1\\
&=  - \left( \max \left \{ \zeta_{L_k+1} - \bar{r}_k, 0 \right \} - \max \left \{ \zeta_{1} - \bar{r}_k, 0 \right \} \right)\\
&\ge  -  \max \left \{ \zeta_{L_k+1} - \bar{r}_k, 0 \right \}
\end{align*}
The first inequality follows from \eqref{line:back_from_travel}. The second inequality follows from the fact that $\bx \in \mathcal{X}^c$ (which implies that $x_{i_{k,\ell}} \le 1$) and the fact that $(\balpha,\bbeta,\gamma) \in \mathcal{D}_k$ (which implies that $\zeta_{\ell+1} \ge \zeta_{\ell}$). The equality follows from algebra, and the third inequality follows from algebra. This completes our proof of Claim~\ref{claim:back_from_travel_1}.  
         \halmos  \end{proof}
                   \begin{claim} \label{claim:back_from_travel_2}
\begin{align*}\gamma + \sum_{\ell=L_k+1}^{N+1} \left( \max \left \{r_{i_{k,\ell}} - \zeta_\ell, 0 \right \} - \left( \zeta_{\ell+1} - \zeta_\ell \right) \right)x_{i_{k,\ell}} - \delta_{L_k+1} \ge \max \left \{\zeta_{L_k+1} - \bar{r}_k, 0 \right \} \end{align*}  
          \end{claim}
          \begin{proof}{Proof of Claim~\ref{claim:back_from_travel_2}.}
          We observe that
          \begin{align}
&\left( \max \left \{r_{i_{k,L_k+1}} - \zeta_{L_k+1}, 0 \right \} - \left( \zeta_{L_k+2} - \zeta_{L_k+1} \right) \right)x_{i_{k,L_k+1}} \notag \\
&=\left( \max \left \{0 - \zeta_{L_k+1}, 0 \right \} - \left( \zeta_{L_k+2} - \zeta_{L_k+1} \right) \right)1 \notag \\
&=\max \left \{\zeta_{L_k+1}, 0 \right \} - \zeta_{L_k+2} \label{line:waves}
          \end{align}
          The first equality follows from the definition of $L_k$ and from the fact that $\bx \in \mathcal{X}^c$ (which implies that $x_{i_{k,L_k+1}} = 1$), and the second equality follows from algebra. Therefore, 
\begin{align*}
&\gamma + \sum_{\ell=L_k+1}^{N+1} \left( \max \left \{r_{i_{k,\ell}} - \zeta_\ell, 0 \right \} - \left( \zeta_{\ell+1} - \zeta_\ell \right) \right)x_{i_{k,\ell}} - \delta_{L_k+1} \\
&=\gamma + \left( \max \left \{\zeta_{L_k+1}, 0 \right \} - \zeta_{L_k+2} \right) +  \sum_{\ell=L_k+2}^{N+1} \left( \max \left \{r_{i_{k,\ell}} - \zeta_\ell, 0 \right \} - \left( \zeta_{\ell+1} - \zeta_\ell \right) \right)x_{i_{k,\ell}} - \delta_{L_k+1} \\
&\ge \gamma + \left( \max \left \{\zeta_{L_k+1}, 0 \right \} - \zeta_{L_k+2} \right) - \sum_{\ell=L_k+2}^{N+1}  \left( \zeta_{\ell+1} - \zeta_\ell \right) x_{i_{k,\ell}} - \delta_{L_k+1} \\
&\ge \gamma + \left( \max \left \{\zeta_{L_k+1}, 0 \right \} - \zeta_{L_k+2} \right) - \sum_{\ell=L_k+2}^{N+1}  \left( \zeta_{\ell+1} - \zeta_\ell \right) 1 - \delta_{L_k+1} \\
&= \gamma + \max \left \{\zeta_{L_k+1}, 0 \right \} -  \zeta_{N+2} - \delta_{L_k+1} \\
&=  \max \left \{\zeta_{L_k+1}, 0 \right \} - \delta_{L_k+1} \\
&= \max \left \{\zeta_{L_k+1}, 0 \right \} - \min \left \{ \max \left \{ \zeta_{L_k+1}, 0 \right \}, \bar{r}_k \right \}  \\
&= \max \left \{ \max \left \{ \zeta_{L_k+1}, 0 \right \} - \bar{r}_k, 0 \right \} \\
%&=  \max \left \{\zeta_{L_k+1} - \delta_{L_k+1} , 0 \right \}\\
&=  \max \left \{\zeta_{L_k+1} - \bar{r}_k , 0 \right \}
\end{align*}
The first equality follows from \eqref{line:waves}. The first inequality follows from algebra and from the fact that $\bx \in \mathcal{X}^c$ (which implies that $x_{i_{k,\ell}} \ge 0$). The second inequality follows from the fact that $(\balpha,\bbeta,\gamma) \in \mathcal{D}_k$ (which implies that $\zeta_{\ell+1} - \zeta_{\ell} \ge 0$) and the fact that $\bx \in \mathcal{X}^c$ (which implies that $x_{i_{k,\ell}} \le 1$). The second equality follows from algebra. The third equality follows from the fact that $\zeta_{N+2} = \gamma$. The fourth equality follows from \eqref{line:alt_defn_delta}. The fifth and sixth equalities follow from algebra.  Our proof of Claim~\ref{claim:back_from_travel_2} is complete. 
          \halmos \end{proof}
          Combining Claims~\ref{claim:back_from_travel_1} and \ref{claim:back_from_travel_2}, we have shown that
          \begin{align*}
             &\gamma + \sum_{\ell=1}^{N+1} \left( \max \left \{ r_{i_{k,\ell}} - \zeta_\ell, 0 \right \} - \left( \zeta_{\ell+1} - \zeta_\ell \right) \right) x_{i_{k,\ell}} - J_k(\bx,\bdelta) \\
              &= \left( \gamma + \sum_{\ell=L_k+1}^{N+1} \left( \max \left \{r_{i_{k,\ell}} - \zeta_\ell, 0 \right \} - \left( \zeta_{\ell+1} - \zeta_\ell \right) \right)x_{i_{k,\ell}} - \delta_{L_k+1} \right)\\
&\quad +\left(  \sum_{\ell=1}^{L_k} \left( \max \left \{r_{i_{k,\ell}} - \zeta_\ell, 0 \right \} - \left( \zeta_{\ell+1} - \zeta_{\ell} \right) \right)x_{i_{k,\ell}} -  \sum_{\ell=1}^{L_k} \left( \max \left \{r_{i_{k,\ell}} - \delta_\ell, 0 \right \} - \left(\delta_{\ell+1} - \delta_\ell \right) \right) x_{i_{k,\ell}}\right) \\
&\ge  \max \left \{\zeta_{L_k+1} - \bar{r}_k, 0 \right \}- \max \left \{\zeta_{L_k+1} - \bar{r}_k, 0 \right \} \\
&= 0
          \end{align*}
          where first equality follows from algebra and the definition of $J_k(\bx,\bdelta)$, the inequality follows from Claims~\ref{claim:back_from_travel_1} and \ref{claim:back_from_travel_2}, and the second equality follows from algebra.  We have thus shown that \eqref{line:reform_hard:main_inequality} holds, which completes our proof of Proposition~\ref{prop:reform_hard:obj}.   \halmos \end{proof}

\subsection{Proof of Theorem~\ref{thm:reform}} \label{appx:reform_bringing_together}
For the sake of completeness, we conclude Appendix~\ref{appx:proof_reformulation} by showing that Theorem~\ref{thm:reform} follows from Theorems~\ref{thm:reform_easy} and \ref{thm:reform_hard}. 
\vspace{0.5em}

\begin{proof}{Proof of Theorem~\ref{thm:reform}.}
We first show that every optimal solution for \eqref{prob:dual_reform} is an optimal solution for \eqref{prob:dual}. Indeed, it follows from Theorem~\ref{thm:reform_easy} that every optimal solution for \eqref{prob:dual_reform} can be transformed into a feasible (but possibly suboptimal) solution for \eqref{prob:dual} with the same objective value. This implies that the optimal objective value of \eqref{prob:dual_reform} is greater than or equal to the optimal objective value of \eqref{prob:dual}. Moreover, it follows from Theorem~\ref{thm:reform_hard} that every optimal solution for \eqref{prob:dual} can be transformed into a feasible (but possibly suboptimal) solution for \eqref{prob:dual_reform} with the same or better objective value. This implies that the optimal objective value of \eqref{prob:dual} is greater than or equal to the optimal objective value of \eqref{prob:dual_reform}. We thus conclude that the optimal objective values of \eqref{prob:dual} and \eqref{prob:dual_reform} are equal, and that every optimal solution for \eqref{prob:dual_reform} can be transformed into an optimal solution for \eqref{prob:dual}. 

Consider any $\bdelta \in \Delta_k$ that is a Pareto cut in the sense of Definition~\ref{defn:pareto_reform}. It follows from Theorem~\ref{thm:reform_easy} that there exists $(\balpha,\bbeta,\gamma) \in \mathcal{D}_k$ that satisfies $J_k(\bx,\bdelta) = \gamma + \sum_{\ell=1}^{N+1} (\alpha_\ell - \beta_\ell)x_{i_{k,\ell}}$ for all $\bx \in \mathcal{X}^c$. Now suppose for the sake of developing a contradiction that $(\balpha,\bbeta,\gamma)$ is \emph{not} a Pareto cut in the sense of Definition~\ref{defn:pareto}.  Then it follows from that supposition that there exists $(\balpha',\bbeta',\gamma') \in \mathcal{D}_k$ that dominates $(\balpha,\bbeta,\gamma)$. It  follows from Theorem~\ref{thm:reform_hard} that there exists $\bdelta' \in \Delta_k$ that satisfies that satisfies $J_k(\bx,\bdelta') \le \gamma' + \sum_{\ell=1}^{N+1} (\alpha'_\ell - \beta'_\ell) x_{i_{k,\ell}}$ for all $\bx \in \mathcal{X}^c$. We have thus shown that
\begin{align*}
    J_k(\bx,\bdelta') \le \gamma' + \sum_{\ell=1}^{N+1} (\alpha'_\ell - \beta'_\ell) x_{i_{k,\ell}} \le \gamma + \sum_{\ell=1}^{N+1} (\alpha_\ell - \beta_\ell)x_{i_{k,\ell}} = J_k(\bx,\bdelta) \;\; \text{for all } \bx \in \mathcal{X}^c
\end{align*}
where the second inequality is strict at some $\bar{\bx} \in \mathcal{X}^c$. We have thus shown that $\bdelta'$ dominates $\bdelta$, which contradicts the assumption that $\bdelta$ is a Pareto cut. This contradiction implies that $(\balpha,\bbeta,\gamma)$ is a Pareto cut in the sense of Definition~\ref{defn:pareto}, which completes the proof of Theorem~\ref{thm:reform}. 
\halmos \end{proof}

\section{Proof of Proposition~\ref{prop:phase2:optimal}}\label{appx:proof:phase2:optimal}
\begin{proof}{Proof of Proposition~\ref{prop:phase2:optimal}. }
Let $k \in [K]$ and $\bx \in \mathcal{X}$. We observe that $\bdelta$ is an optimal solution for \eqref{prob:dual_reform} because
\begin{align*}
J_k(\bx,\bdelta) &= \delta_{L_k+1}+ \sum_{\ell=\ell^*}^{L_k}  \left( \max \left \{0, r_{i_{k,\ell}} - \delta_\ell \right \} - \left( \delta_{\ell+1} - \delta_\ell \right) \right) x_{i_{k,\ell}} \\
&=\bar{r}_k+ \sum_{\ell=\ell^*+1}^{L_k}  \left( \max \left \{0, r_{i_{k,\ell}} - \bar{r}_k \right \} - \left( \bar{r}_k - \bar{r}_k \right) \right) x_{i_{k,\ell}} \\
 &\quad + \left( \max \left \{0, r_{i_{k,\ell^*}} - r_{i_{k,\ell^*}} \right \} - \left( \bar{r}_k - r_{i_{k,\ell^*}}\right) \right) x_{i_{k,\ell}} \\
 &=\bar{r}_k+ \sum_{\ell=\ell^*+1}^{L_k}  \left(0 - 0 \right) x_{i_{k,\ell}}  + \left( 0 - \left( \bar{r}_k - r_{i_{k,\ell^*}}\right) \right) x_{i_{k,\ell}} \\
  &=\bar{r}_k- \left( \bar{r}_k - r_{i_{k,\ell^*}}\right)  x_{i_{k,\ell}} \\
  &= r_{i_{k,\ell^*}} 
\end{align*}
The first equality holds because our definition of $\ell^*$ in \eqref{line:phase2:ell_star} implies that $x_{i_{k,1}}= \cdots = x_{i_{k,\ell^*-1}} = 0$. The second equality follows from our construction of  $\bdelta$ in \eqref{line:phase2:delta_ell}. The third equality follows from the fact that $r_{i_{k,\ell}} \le \bar{r}_k$ for all $\ell \in \{1,\ldots,L_k+1\}$. The fourth equality follows from algebra. The fifth equality holds because our definition of $\ell^*$ in \eqref{line:phase2:ell_star} implies that $x_{i_{k,\ell^*}} = 1$. Since the optimal objective value of \eqref{prob:dual_reform} is clearly equal to $r_{i_{k,\ell^*}}$, our proof of Proposition~\ref{prop:phase2:optimal} is complete.  
\halmos  \end{proof}
\section{Illustration for Properties~\ref{property:notspecial}, \ref{property:forward}, \ref{property:gap}, and \ref{property:reverse} } \label{sec:interpretation:four}
In this appendix, we provide intuition for Properties~\ref{property:notspecial}, \ref{property:forward}, \ref{property:gap}, and \ref{property:reverse} by presenting four illustrative examples. Each example   shows that removing any one of those four properties from Section~\ref{sec:characterization_pareto} can allow for solutions for \eqref{prob:dual_reform} that are not Pareto cuts, in the sense of Definition~\ref{defn:pareto_reform}. By showing how removing any of those properties allows for non-Pareto cuts, we motivate  each of the four properties as well as motivate the four subroutines in the transformation algorithm from Section~\ref{sec:transformation}.

\begin{example}[Property~\ref{property:notspecial} is violated] \label{example:property1:violated}
Suppose that there are $N = 2$ products, the revenues are $r_1 = \$10$ and $r_2 = \$5$,  and ranking $k$ satisfies $i_{k,1} = 1$, $i_{k,2} = 2$, and $i_{k,3} = 3$. Consider the two feasible solutions $
\bdelta = (10, 10, 10)$ and $\bdelta' = (5, 5, 5)$, which  correspond to 
\begin{align*}
    J_k(\bx,\bdelta) &= 10 + (\max \{0, 10 - 10\} - (10 - 10)) x_1 + (\max \{ 0, 5 - 10 \} - (10 - 10)) x_2 &&= 10\\
    J_k(\bx,\bdelta') &= 5 + (\max \{0, 10 - 5\} - (5 - 5)) x_1 + (\max \{ 0, 5 - 5 \} - (5 - 5)) x_2 &&= 5 + 5 x_1
\end{align*}
We observe that $\bdelta$ and $\bdelta'$ are optimal solutions for \eqref{prob:dual_reform} in the case of $\bx = (1,0,1)$. We also observe that $T_k(\bdelta) = 1$ (since $\delta_1 = 10 = r_{i_{k,1}}$, $\delta_2 = 10 > r_{i_{k,2}} = 5$, and $\delta_3 = 10 > r_{i_{k,1}} = 0$), and that  $T_k(\bdelta') = 2$ (since $\delta_2' = 5 = r_{i_{k,2}}$ and $\delta_3' = 5 > r_{i_{k,3}} = 0$). 
\begin{itemize}
        \item It follows from the fact that $T_k(\bdelta) = 1$  that $\bdelta$ does not satisfy Property~\ref{property:notspecial}, and  it follows from the fact that $T_k(\bdelta') = 2$ that  $\bdelta'$ satisfies Property~\ref{property:notspecial}.
    \item It follows from the fact that $\delta_1 = 10 = r_{i_{k,1}}$ and $\delta_1' = 5 \le r_{i_{k,1}}$ that $\bdelta$ and $\bdelta'$ satisfy Property~\ref{property:forward}. 
    \item It follows from the fact that $ \delta_2 = \delta_3$ and $\delta_2' = \delta_3'$ that $\bdelta$ and $\bdelta'$ satisfy Property~\ref{property:gap}. 
        \item It follows from the fact that $T_k(\bdelta) = 1$ and $\delta_1 = \delta_2 = \delta_3$ that $\bdelta$ satisfies Property~\ref{property:reverse}, and it follows from the fact that $T_k(\bdelta') = 2$ and $ \delta_2' = \delta_3'$ that $\bdelta'$ satisfies Property~\ref{property:reverse}. 
\end{itemize}
Finally, we observe that $J_k(\bx,\bdelta') = J_k(\bx,\bdelta)$ for all $\bx \in \mathcal{X}^c$ that satisfy $x_1 = 1$ and that $J_k(\bx,\bdelta') < J_k(\bx,\bdelta)$ for all $\bx \in \mathcal{X}^c$ that satisfy $x_1 \in [0,1)$. We have thus shown that  $\bdelta'$ dominates $\bdelta$. \halmos
\end{example}

\begin{example}[Property~\ref{property:forward} is violated] \label{example:property2:violated}
Suppose that there are $N = 2$ products, the revenues are $r_1 = \$5$ and $r_2 = \$10$,  and ranking $k$ satisfies $i_{k,1} = 1$, $i_{k,2} = 2$, and $i_{k,3} = 3$. Consider the two feasible solutions $
\bdelta = (10, 10, 10)$ and $\bdelta' = (5, 10, 10)$, which we observe correspond to 
\begin{align*}
    J_k(\bx,\bdelta) &= 10 + (\max \{0, 5 - 10\} - (10 - 10) x_1 + (\max \{ 0, 10 - 10 \} - (10 - 10)) x_2 &&= 10\\
    J_k(\bx,\bdelta') &= 10 + (\max \{0, 5 - 5\} - (10 - 5) x_1 + (\max \{ 0, 10 - 10 \} - (10 - 10)) x_2 &&= 10 - 5 x_1
\end{align*}
We observe that $\bdelta$ and $\bdelta'$ are both optimal solutions for \eqref{prob:dual_reform} in the case of $\bx = (0,1,1)$. 
\begin{itemize}
        \item We readily observe that $T_k(\bdelta) = T_k(\bdelta') = 2$, which implies that $\bdelta$ and $\bdelta'$ satisfy Property~\ref{property:notspecial}. 
    \item It follows from the fact that $\delta_1 = 10 > r_{i_{k,1}}$  that $\bdelta$ does not satisfy Property~\ref{property:forward}, and it follows from the fact that $\delta_1' = 5 = r_{i_{k,1}}$ that $\bdelta'$ satisfies  Property~\ref{property:forward}. 
    \item It follows from the fact that $ \delta_2 = \delta_3$ and $\delta_2' = \delta_3'$ that $\bdelta$ and $\bdelta'$ satisfy Property~\ref{property:gap}. 
    \item It follows from the fact that $ \delta_2 = \delta_3$ and $\delta_2' = \delta_3'$ that $\bdelta$ and $\bdelta'$ satisfies Property~\ref{property:reverse}.  
\end{itemize}
Finally, we observe that $J_k(\bx,\bdelta') = J_k(\bx,\bdelta)$ for all $\bx \in \mathcal{X}^c$ that satisfy $x_1 = 0$ and that $J_k(\bx,\bdelta') < J_k(\bx,\bdelta)$ for all $\bx \in \mathcal{X}^c$ that satisfy $x_1 \in (0,1]$. We have thus shown that  $\bdelta'$ dominates $\bdelta$. \halmos
\end{example}
\begin{example}[Property~\ref{property:gap} is violated] \label{example:property3:violated}
Suppose that there are $N = 3$ products, the revenues are $r_1 = r_2 = r_3 = \$10$,  and ranking $k$ satisfies $i_{k,1} = 1$, $i_{k,2} = 2$, $i_{k,3} = 3$, and $i_{k,4} = 4$. Consider the two feasible solutions $
\bdelta = (0, 0, 10, 10)$ and $\bdelta' = (0, 10, 10,10)$, which we observe correspond to 
\begin{align*}
    J_k(\bx,\bdelta) &= 10 + (\max \{0, 10 - 0\} - (0 - 0) x_1 \\
    &\quad + (\max \{ 0, 10 - 0 \} - (10 - 0)) x_2  + (\max \{ 0, 10 - 10 \} - (10 - 10)) x_3  &&= 10 + 10 x_1 \\
    J_k(\bx,\bdelta') &= 10 + (\max \{0, 10 - 0\} - (10 - 0) x_1 \\
    &\quad + (\max \{ 0, 10 - 10 \} - (10 - 10)) x_2  + (\max \{ 0, 10 - 10 \} - (10 - 10)) x_3  &&= 10 \end{align*}
We observe that $\bdelta$ and $\bdelta'$ are both optimal solutions for \eqref{prob:dual_reform} in the case of $\bx = (0,1,1,1)$. 
\begin{itemize}
        \item We readily observe that $T_k(\bdelta) = T_k(\bdelta') = 3$, which implies that $\bdelta$ and $\bdelta'$ satisfy Property~\ref{property:notspecial}. 
    \item It follows from the fact that $\delta_1 = \delta'_1 = 0 < 10 = r_{i_{k,1}}$   that $\bdelta$ and $\bdelta'$ satisfy  Property~\ref{property:forward}. 
    \item It follows from the fact that $\delta_2 < r_{i_{k,2}} = 10$ and $\delta_2 < \delta_3$    that $\bdelta$  does not satisfy Property~\ref{property:gap}. It follows from the fact that $\delta_2' = \delta_3' = \delta_4'$ that $\bdelta'$ satisfies Property~\ref{property:gap}. 
    \item It follows from the fact that $ \delta_3 = \delta_4$ and $\delta_3' = \delta_4'$ that $\bdelta$ and $\bdelta'$ satisfies Property~\ref{property:reverse}.  
\end{itemize}
Finally, we observe that $J_k(\bx,\bdelta') = J_k(\bx,\bdelta)$ for all $\bx \in \mathcal{X}^c$ that satisfy $x_1 = 0$ and that $J_k(\bx,\bdelta') < J_k(\bx,\bdelta)$ for all $\bx \in \mathcal{X}^c$ that satisfy $x_1 \in (0,1]$. We have thus shown that  $\bdelta'$ dominates $\bdelta$. \halmos
\end{example}
\begin{example}[Property~\ref{property:reverse} is violated] \label{example:property4:violated}
Suppose that there are $N = 2$ products, the revenues are $r_1 = \$10$ and $r_2 = \$5$,  and ranking $k$ satisfies $i_{k,1} = 1$, $i_{k,2} = 2$, and $i_{k,3} = 3$. Consider the two feasible solutions $
\bdelta = (5, 5, 10)$ and $\bdelta' = (5, 5, 5)$, which we observe correspond to 
\begin{align*}
    J_k(\bx,\bdelta) &= 10 + (\max \{0, 10 - 5\} - (5 - 5) x_1 + (\max \{ 0, 5 - 5 \} - (10 - 5)) x_2 &&= 10 + 5 x_1 - 5 x_2\\
    J_k(\bx,\bdelta') &= 5 + (\max \{0, 10 - 5\} - (5 - 5) x_1 + (\max \{ 0, 5 - 5 \} - (5 - 5)) x_2 &&= 5 + 5 x_1 
\end{align*}
We observe that $\bdelta$ and $\bdelta'$ are both optimal solutions for \eqref{prob:dual_reform} in the case of $\bx = (0,1,1)$. 
\begin{itemize}
        \item We readily observe that $T_k(\bdelta) = T_k(\bdelta') = 2$, which implies that $\bdelta$ and $\bdelta'$ satisfy Property~\ref{property:notspecial}. 
    \item It follows from the fact that $\delta_1 =\delta_1' = 5 < 10  = r_{i_{k,1}}$  that $\bdelta$ and $\bdelta'$ satisfy Property~\ref{property:forward}. 
    \item It follows from the fact that $\delta_2,\delta_2' \ge r_{i_{k,2}}$ that $\bdelta$ and $\bdelta'$  satisfy Property~\ref{property:gap}. 
    \item It follows from the fact that $T_k(\bdelta) = 2$ and $ \delta_2 < \delta_3$ that $\bdelta$ does not satisfy Property~\ref{property:reverse}, and it follows from the fact that  $\delta_2' = \delta_3'$ that $\bdelta'$ satisfies Property~\ref{property:reverse}.  
\end{itemize}
Finally, we observe that $J_k(\bx,\bdelta') = J_k(\bx,\bdelta)$ for all $\bx \in \mathcal{X}^c$ that satisfy $x_2 = 1$ and that $J_k(\bx,\bdelta') < J_k(\bx,\bdelta)$ for all $\bx \in \mathcal{X}^c$ that satisfy $x_2 \in [0,1)$. We have thus shown that  $\bdelta'$ dominates $\bdelta$. \halmos
\end{example}

\section{Proof of Theorem~\ref{thm:sufficient}} \label{appx:thm:sufficient}
Our proof of Theorem~\ref{thm:sufficient} in the present  Appendix~\ref{appx:thm:sufficient} shows that if $\bdelta \in \Delta_k$ satisfies Properties~\ref{property:notspecial}, \ref{property:forward}, \ref{property:gap}, and \ref{property:reverse}, then $\bdelta$ is a Pareto cut. Our proof   makes use of the notions of `Pareto cuts' and `domination' that are given in Definition~\ref{defn:pareto_reform} (Section~\ref{sec:main_alg:reform}) as well as makes use of Theorem~\ref{thm:transformation} (Section~\ref{sec:characterization_pareto}).  

We begin our proof of Theorem~\ref{thm:sufficient} with the following straightforward lemma. 
\begin{lemma} \label{lem:transformation:property_domination}
    Let $\bdelta \in \Delta_k$. We have that $\bdelta$ is a Pareto cut if and only if there does not exist $\bdelta' \in \Delta_k$ that  satisfies Properties~\ref{property:notspecial}, \ref{property:forward}, \ref{property:gap}, and \ref{property:reverse} and  dominates $\bdelta$.
\end{lemma}
\begin{proof}{Proof of Lemma~\ref{lem:transformation:property_domination}.}
Let $\bdelta \in \Delta_k$. If $\bdelta$ is a Pareto cut, then it follows from Definition~\ref{defn:pareto_reform} that there does not exist $\bdelta' \in \Delta_k$ that dominates $\bdelta$, which proves the first direction of Lemma~\ref{lem:transformation:property_domination}. To prove the other direction, suppose that $\bdelta$ is not a Pareto cut. In that case, it follows from Definition~\ref{defn:pareto_reform} that there exists $\bdelta' \in \Delta_k$ that dominates $\bdelta$, in the sense that 
    $J_k(\bx,\bdelta') \le J_k(\bx,\bdelta) \text{ for all }\bx \in \mathcal{X}^c$, 
where that inequality is strict at some $\bar{\bx} \in \mathcal{X}^c$. It follows from Theorem~\ref{thm:transformation} that either $\bdelta'$ satisfies Properties~\ref{property:notspecial}, \ref{property:forward}, \ref{property:gap}, and \ref{property:reverse} or there exists $\bdelta'' \in \Delta_k$ that satisfies Properties~\ref{property:notspecial}, \ref{property:forward}, \ref{property:gap}, and \ref{property:reverse} and dominates $\bdelta'$. Therefore, we conclude in the latter case that  $\bdelta''$ satisfies Properties~\ref{property:notspecial}, \ref{property:forward}, \ref{property:gap}, and \ref{property:reverse} and dominates $\bdelta$, which completes the proof of the other direction of Lemma~\ref{lem:transformation:property_domination}.   \halmos \end{proof}

Equipped with the above lemma, we  now present the organization of the proof of Theorem~\ref{thm:sufficient}. Indeed, consider any $\bdelta,\bdelta' \in \Delta_k$ that satisfy Properties~\ref{property:notspecial}, \ref{property:forward}, \ref{property:gap}, and \ref{property:reverse}.  
Our proof of Theorem~\ref{thm:sufficient} is split into five cases, which are shown in the following table. 
\begin{center}
\begin{tabular}{c l l}
\toprule
\textbf{Case} & \textbf{If\ldots} & \textbf{then $\bdelta'$ does not dominate $\bdelta$, as shown in \ldots} \\
\midrule
1 & $\delta'_{L_k+1} > \delta_{L_k+1}$ & Proposition~\ref{prop:case1} in Appendix~\ref{appx:sufficient:case1} \\
2 & $\delta'_{L_k+1} = \delta_{L_k+1}$ and $\delta_2' < \delta_2$& Proposition~\ref{prop:case2} in Appendix~\ref{appx:sufficient:case2} \\
3 & $\delta'_{L_k+1} = \delta_{L_k+1}$ and $\delta_2' \ge \delta_2$& Proposition~\ref{prop:case3} in Appendix~\ref{appx:sufficient:case3} \\
4 & $\delta_{L_k+1}' < \delta_{L_k+1}$ and $\delta_{T_k(\bdelta)}' < \delta'_{L_k+1}$ &Proposition~\ref{prop:case4} in Appendix~\ref{appx:sufficient:case4}\\
5 & $\delta_{L_k+1}' < \delta_{L_k+1}$ and $\delta_{T_k(\bdelta)}' = \delta'_{L_k+1}$ &Proposition~\ref{prop:case5} in Appendix~\ref{appx:sufficient:case5}\\
\bottomrule
\end{tabular}
\end{center}

\vspace{1em}

\noindent We observe that the five regimes in the second column are mutually exclusive and collectively exhaustive.\footnote{Note that the case where $\delta_{T_k(\bdelta)}' > \delta'_{L_k+1}$ does not need to be considered because  $\bdelta' \in \Delta_k$ implies that $\delta_{T_k(\bdelta)}' \le \delta'_{L_k+1}$. }
Consequently, we conclude that Propositions~\ref{prop:case1}-\ref{prop:case5} in combination with Lemma~\ref{lem:transformation:property_domination} completes the proof of Theorem~\ref{thm:sufficient}. The remainder of Appendix~\ref{appx:thm:sufficient} consists of proving Propositions~\ref{prop:case1}-\ref{prop:case5}.

\subsection{Case 1}\label{appx:sufficient:case1}
 Our first case consists of proving the following proposition. 
\begin{proposition} \label{prop:case1} 
If $\delta_{L_k+1}' > \delta_{L_k+1}$, then $\bdelta'$ does not dominate $\bdelta$.
\end{proposition}
\begin{proof}{Proof of Proposition~\ref{prop:case1}.}
Assume that $\delta_{L_k+1}' > \delta_{L_k+1}$. 
Let $\bx$ be the vector that is defined for each $\ell \in [N+1]$ as
    \begin{align*}
        x_{i_{k,\ell}} &= \begin{cases}
            1,&\text{if } \ell = L_k+1,\\
            0,&\text{otherwise}
        \end{cases}
    \end{align*}We observe that
    \begin{align*}
        J_k(\bx,\bdelta') = \delta_{L_k+1}'+ \sum_{\ell=1}^{L_k}  \left( \max \left \{0, r_{i_{k,\ell}} - \delta_\ell' \right \} - \left( \delta_{\ell+1}' - \delta_\ell' \right) \right) 0 = \delta'_{L_k+1}\\
      J_k(\bx,\bdelta) = \delta_{L_k+1}+ \sum_{\ell=1}^{L_k}  \left( \max \left \{0, r_{i_{k,\ell}} - \delta_\ell \right \} - \left( \delta_{\ell+1} - \delta_\ell \right) \right) 0 = \delta_{L_k+1}
    \end{align*}
    which implies that
    \begin{align*}
        J_k(\bx,\bdelta') > J_k(\bx,\bdelta)
    \end{align*}
   Since $\bx \in \mathcal{X}$, we have shown that $\bdelta'$ does not dominate $\bdelta$. 
\halmos \end{proof}

\subsection{Cases 2 and 3}\label{appx:sufficient:case2andcase3}
Our proofs in Appendices~\ref{appx:sufficient:case2} and \ref{appx:sufficient:case3} will make use of the following lemma. 
\begin{lemma} \label{lem:case2andcase3}
    If $\hat{\bdelta} \in \Delta_k$  satisfies Property~\ref{property:forward}, then for all $\bx \in \mathcal{X}^c$ we have
    \begin{align*}
        J_k(\bx,\hat{\bdelta}) &= \hat{\delta}_{L_k+1} + \left(r_{i_{k,1}} - \hat{\delta}_2 \right)x_{i_{k,1}} + \sum_{\ell=2}^{L_k}  \left( \max \left \{0, r_{i_{k,\ell}} - \hat{\delta}_\ell \right \} - \left( \hat{\delta}_{\ell+1} - \hat{\delta}_\ell \right) \right) x_{i_{k,\ell}}
    \end{align*}
\end{lemma}
\begin{proof}{Proof of Lemma~\ref{lem:case2andcase3}.}
  Note that the inequality $1 < L_k+1$ follows from Assumption~\ref{ass:L_k}  (see  Section~\ref{sec:intro:methods}). Let $\hat{\bdelta} \in \Delta_k$ satisfy Property~\ref{property:forward}.  We observe  for all $\bx \in \mathcal{X}^c$ that
        \begin{align*}
      & \left( \max \left \{0, r_{i_{k,1}} - \hat{\delta}_1 \right \} - \left( \hat{\delta}_{2} - \hat{\delta}_1 \right) \right)x_{i_{k,1}}\\
            &= \left( \max \left \{\hat{\delta}_1, r_{i_{k,1}} \right \} - \hat{\delta}_{2}   \right)x_{i_{k,1}}\\
            &= \left(  r_{i_{k,1}} - \hat{\delta}_{2}   \right)x_{i_{k,1}}
    \end{align*}
    where the first equality follows from algebra and the second equality follows from the fact that $\hat{\bdelta}$ satisfies Property~\ref{property:forward}. The above reasoning together with the definition of $J_k(\bx,\hat{\bdelta})$ completes the proof of Lemma~\ref{lem:case2andcase3}. 
\halmos \end{proof}
Equipped with Lemma~\ref{lem:case2andcase3}, we proceed to prove Case 2 and Case 3. 
\subsubsection{Case 2. }\label{appx:sufficient:case2}
 Our second case consists of proving the following proposition. 
\begin{proposition} \label{prop:case2} 
If $\delta_{L_k+1}' = \delta_{L_k+1}$ and $\delta_2' < \delta_2$, then $\bdelta'$ does not dominate $\bdelta$.
\end{proposition}
\begin{proof}{Proof of Proposition~\ref{prop:case2}.}
Assume that $\delta_{L_k+1}' = \delta_{L_k+1}$ and $\delta_2' < \delta_2$. It follows from the fact that $\delta_{L_k+1}' = \delta_{L_k+1}$ and $\delta_2' < \delta_2$ that $L_k + 1 > 2$. 
       Let $\bx$ be the vector that is defined for each $\ell \in [N+1]$ as
    \begin{align*}
        x_{i_{k,\ell}} &= \begin{cases}
            1,&\text{if } \ell \in \{1,L_k+1\},\\
            0,&\text{otherwise}
        \end{cases}
    \end{align*}
  We observe that
        \begin{align*}
        J_k(\bx,\bdelta') &= \delta_{L_k+1}'+ \left( r_{i_{k,1}} - \delta'_{2} \right) 1 + \sum_{\ell=2}^{L_k}  \left( \max \left \{0, r_{i_{k,\ell}} - \delta_\ell' \right \} - \left( \delta_{\ell+1}' - \delta_\ell' \right) \right) 0 \\
        &= \delta_{L_k+1}' + r_{i_{k,1}}  - \delta_{2}' \\
&> \delta_{L_k+1} + r_{i_{k,1}}  - \delta_{2} \\
&=  \delta_{L_k+1}+ \left( r_{i_{k,1}} - \delta_{2} \right) 1 + \sum_{\ell=2}^{L_k}  \left( \max \left \{0, r_{i_{k,\ell}} - \delta_\ell \right \} - \left( \delta_{\ell+1} - \delta_\ell \right) \right) 0 \\
&= J_k(\bx,\bdelta)
    \end{align*}
       The first equality follows from the fact that $\bdelta'$ satisfies Property~\ref{property:forward} and from Lemma~\ref{lem:case2andcase3} as well as the construction of $\bx$. The second equality follows from algebra. The inequality follows from the fact that $\delta'_{L_k+1} = \delta_{L_k+1}$ and $\delta'_2 < \delta_2$. The third equality follows from algebra. The fourth equality   follows from the fact that $\bdelta$ satisfies Property~\ref{property:forward} and from Lemma~\ref{lem:case2andcase3} as well as the construction of $\bx$.   Since $\bx \in \mathcal{X}$, we have shown that $\bdelta'$ does not dominate $\bdelta$.  \halmos \end{proof}
  \subsubsection{Case 3. }\label{appx:sufficient:case3}
 Our third case consists of proving the following proposition. 
\begin{proposition} \label{prop:case3} 
 If $\delta_{L_k+1}' = \delta_{L_k+1}$ and $\delta_2' \ge \delta_2$, then $\bdelta'$ does not dominate $\bdelta$.
\end{proposition}

Our proof of Proposition~\ref{prop:case3} will utilize the following lemma.
\begin{lemma} \label{lem:idkmaybe}
 If $\delta_{L_k+1}' = \delta_{L_k+1}$ and $\delta_2' \ge \delta_2$,  then either $(\delta_2,\ldots,\delta_{L_k+1}) = (\delta_2',\ldots,\delta_{L_k+1}')$ or there exists $\ell \in \{2,\ldots,L_k\}$ such that $\delta'_{\ell+1} - \delta'_\ell < \delta_{\ell+1} - \delta_{\ell}$.
\end{lemma}
   \begin{proof}{Proof of Lemma~\ref{lem:idkmaybe}.}
  Assume that $\delta_{L_k+1}' = \delta_{L_k+1}$ and $\delta_2' \ge \delta_2$. We recall from Assumption~\ref{ass:L_k}  (see  Section~\ref{sec:intro:methods}) that $L_k \ge 1$, and the lemma clearly holds if $L_k= 1$. We thus assume throughout the rest of the proof of Lemma~\ref{lem:idkmaybe} that $L_k \ge 2$. Suppose that $(\delta_2,\ldots,\delta_{L_k+1}) \neq (\delta_2',\ldots,\delta_{L_k+1}')$ and 
   \begin{align}
      \delta'_{\ell'+1} - \delta'_{\ell'} \ge \delta_{\ell'+1} - \delta_{\ell'}\;  \forall \ell' \in \{2,\ldots,L_k\}.  \label{line:idkmaybe:claim:1}
   \end{align}
   In that case, we observe that
   \begin{align}
       \delta'_{L_k+1} - \delta'_{2} = \sum_{\ell'=2}^{L_k} \left( \delta'_{\ell'+1} - \delta'_{\ell'} \right) \ge \sum_{\ell'=2}^{L_k} \left( \delta_{\ell'+1} - \delta_{\ell'} \right) =  \delta_{L_k+1} - \delta_{2}. \label{line:idkmaybe:claim:2} 
   \end{align}
   We recall that $\delta_{L_k+1}' = \delta_{L_k+1}$ and $\delta'_2 \ge \delta_2$, and so it follows from \eqref{line:idkmaybe:claim:2} that
   \begin{align}
  \delta'_{L_k+1} - \delta'_{2}  =  \delta_{L_k+1} - \delta_{2}.    \label{line:idkmaybe:claim:3}
    \end{align}
       The only way that \eqref{line:idkmaybe:claim:1} and \eqref{line:idkmaybe:claim:3} can be true simultaneously is for the equality $\delta'_{\ell'+1} - \delta'_{\ell'} = \delta_{\ell'+1} - \delta_{\ell'}$ to hold for all $\ell' \in \{2,\ldots,L_k\}$. Since $\delta_{L_k+1}' = \delta_{L_k+1}$, it must be the case that
    $(\delta_2,\ldots,\delta_{L_k+1}) = (\delta_2',\ldots,\delta_{L_k+1}')$, which is a contradiction.  This concludes the proof of Lemma~\ref{lem:idkmaybe}. 
   \halmos \end{proof}

  Equipped with Lemma~\ref{lem:idkmaybe}, we now present our proof of Proposition~\ref{prop:case3}. 
\begin{proof}{Proof of Proposition~\ref{prop:case3}.}
Assume that   $\delta_{L_k+1}' = \delta_{L_k+1}$ and $\delta_2' \ge \delta_2$. If $(\delta_2,\ldots,\delta_{L_k+1}) = (\delta_2',\ldots,\delta_{L_k+1}')$, then it follows from Lemma~\ref{lem:case2andcase3} and from the fact that $\bdelta$ and $\bdelta'$ satisfy Property~\ref{property:forward} that $J_k(\bx,\bdelta) = J_k(\bx,\bdelta')$ for all $\bx \in \mathcal{X}^c$, which implies that $\bdelta'$ does not dominate $\bdelta$. Therefore, we assume henceforth that $(\delta_2,\ldots,\delta_{L_k+1}) \neq (\delta_2',\ldots,\delta_{L_k+1}')$, and so it follows from Lemma~\ref{lem:idkmaybe} that there exists   $\ell^* \in \{2,\ldots,L_k\}$ that satisfies $\delta'_{\ell^*+1} - \delta'_{\ell^*} < \delta_{\ell^*+1} - \delta_{\ell^*}$.    Let $\bx$ be the vector that is defined for each $\ell \in [N+1]$ as
    \begin{align*}
        x_{i_{k,\ell}} &= \begin{cases}
            1,&\text{if } \ell \in \{\ell^*,L_k+1\},\\
            0,&\text{otherwise}
        \end{cases}
    \end{align*}
 We observe that 
       \begin{align*}
        J_k(\bx,\bdelta')  &= {\delta}_{L_k+1}' + \left(r_{i_{k,1}} - {\delta}_2' \right)0 + \sum_{\ell \in \{2,\ldots,L_k\} \setminus \{ \ell^*\}}  \left( \max \left \{0, r_{i_{k,\ell}} - {\delta}_\ell' \right \} - \left( {\delta}_{\ell+1}' - {\delta}_\ell' \right) \right) 0\\
       &\quad + \left( \max \left \{0, r_{i_{k,\ell^*}} - \delta_{\ell^*}' \right \} - \left( \delta_{\ell^*+1}' - \delta_{\ell^*}' \right) \right)1\\
           &= \delta_{L_k+1}' + \max \left \{0, r_{i_{k,\ell^*}} - \delta_{\ell^*}' \right \} - \left( \delta_{\ell^*+1}' - \delta_{\ell^*}' \right) \\
           &\ge \delta_{L_k+1}' + 0 - \left( \delta_{\ell^*+1}' - \delta_{\ell^*}' \right) \\
        &= \delta_{L_k+1} - \left( \delta_{\ell^*+1}' - \delta_{\ell^*}' \right)  \\
     &> \delta_{L_k+1} - \left( \delta_{\ell^*+1} - \delta_{\ell^*} \right) 
    \end{align*}
   The first equality follows from the fact that $\bdelta'$ satisfies Property~\ref{property:forward} and from Lemma~\ref{lem:case2andcase3} as well as the construction of $\bx$. The second equality and the first inequality follow from algebra. The third equality holds because $\delta_{L_k+1} = \delta'_{L_k+1}$. The second inequality  follows from the definition of $\ell^*$. 
    Moreover, we observe that
       \begin{align*}
        J_k(\bx,\bdelta)
  &= {\delta}_{L_k+1} + \left(r_{i_{k,1}} - {\delta}_2 \right)0 + \sum_{\ell \in \{2,\ldots,L_k\} \setminus \{ \ell^*\}}  \left( \max \left \{0, r_{i_{k,\ell}} - {\delta}_\ell \right \} - \left( {\delta}_{\ell+1} - {\delta}_\ell \right) \right) 0\\
       &\quad + \left( \max \left \{0, r_{i_{k,\ell^*}} - \delta_{\ell^*} \right \} - \left( \delta_{\ell^*+1} - \delta_{\ell^*} \right) \right)1\\
           &= \delta_{L_k+1} + \max \left \{0, r_{i_{k,\ell^*}} - \delta_{\ell^*} \right \} - \left( \delta_{\ell^*+1} - \delta_{\ell^*} \right)  \notag \\
    &= \delta_{L_k+1} + 0 - \left( \delta_{\ell^*+1} - \delta_{\ell^*} \right)    
    \end{align*}
   The first equality follows from the fact that $\bdelta$ satisfies Property~\ref{property:forward} and from Lemma~\ref{lem:case2andcase3} as well as the construction of $\bx$. The second equality follows from algebra. The third equality holds because $\bdelta$ satisfies Property~\ref{property:gap} and because $0 \le \delta'_{\ell^*+1} - \delta'_{\ell^*} < \delta_{\ell^*+1} - \delta_{\ell^*}$ and because $\ell^* \in \{2,\ldots,L_k\}$, which together imply that $\delta_{\ell^*} \ge r_{i_{k,\ell^*}}$. Combining the above reasoning, we have shown that
   \begin{align*}
       J_k(\bx,\bdelta') > \delta_{L_k+1} - \left( \delta_{\ell^*+1} - \delta_{\ell^*} \right)  = J_k(\bx,\bdelta) 
   \end{align*}
   Since $\bx \in \mathcal{X}$, we have shown that $\bdelta'$ does not dominate $\bdelta$.
\halmos \end{proof}

\subsection{Case 4}\label{appx:sufficient:case4}
 Our fourth case consists of proving the following proposition.
 
\begin{proposition} \label{prop:case4} 
 If $\delta_{L_k+1}' < \delta_{L_k+1}$ and $\delta_{T_k(\bdelta)}' < \delta'_{L_k+1}$, then $\bdelta'$ does not dominate $\bdelta$.
\end{proposition}
\begin{proof}{Proof of Proposition~\ref{prop:case4}.}
Assume that   $\delta_{L_k+1}' < \delta_{L_k+1}$ and $\delta_{T_k(\bdelta)}' < \delta'_{L_k+1}$. For notational convenience, let $\ell^* \triangleq T_k(\bdelta)$.    We will make use of the following two intermediary claims.
\begin{claim}\label{claim:newslyadded}
$\delta'_{\ell^*} < \delta_{\ell^*} \le r_{i_{k,\ell^*}}$.
\end{claim}
\begin{proof}{Proof of Claim~\ref{claim:newslyadded}.}
We observe that 
\begin{align*}
    \delta'_{\ell^*} &\le \delta'_{L_k+1} < \delta_{L_k+1} = \delta_{\ell^*}\le r_{i_{k,\ell^*}}
\end{align*}
The first inequality follows from the fact that $\bdelta' \in \Delta_k$. The second inequality follows from assumption. The equality follows from the fact that $\bdelta$ satisfies Property~\ref{property:reverse} and the fact that $\ell^* = T_k(\bdelta)$. The third inequality follows from the definition of $T_k(\bdelta)$ and the fact that $\ell^* = T_k(\bdelta)$. That completes the proof of Claim~\ref{claim:newslyadded}. 
\halmos \end{proof}
\begin{claim}\label{claim:ell_star_between_two_and_L_k}
    $\ell^* \in \{2,\ldots,L_k\}$. 
\end{claim}
\begin{proof}{Proof of Claim~\ref{claim:ell_star_between_two_and_L_k}.}
We observe that 
\begin{align*}
    r_{i_{k,L_k+1}} = 0 \le \delta'_{\ell^*} < \delta_{\ell^*}
\end{align*}
where the equality follows from the fact that $i_{k,L_k+1} = N+1$, the first inequality follows from the fact that $\bdelta' \in \Delta_k$, and the second inequality follows from Claim~\ref{claim:newslyadded}. This shows that $\ell^* <L_k+1$. Moreover,  it also follows from the fact that $\bdelta$ satisfies Property~\ref{property:notspecial} that $\ell^* \ge 2$.  That completes the proof of Claim~\ref{claim:ell_star_between_two_and_L_k}. 
\halmos \end{proof}

Equipped with the above intermediary claims, we now proceed to prove Proposition~\ref{prop:case4}. Indeed, let $\bx$ be the vector that is defined for each $\ell \in [N+1]$ as
    \begin{align*}
        x_{i_{k,\ell}} &= \begin{cases}
            1,&\text{if } \ell \in \{\ell^*,L_k+1\},\\
            0,&\text{otherwise}
        \end{cases}
    \end{align*}
 Note that it follows from Claim~\ref{claim:ell_star_between_two_and_L_k} that $\ell^* < L_k+1$. We observe that 
  \begin{align*}
        J_k(\bx,\bdelta')  &= \delta_{L_k+1}'+ \sum_{\ell \in \{1,\ldots,L_k\} \setminus \{\ell^*\}}  \left( \max \left \{0, r_{i_{k,\ell}} - \delta_\ell' \right \} - \left( \delta_{\ell+1}' - \delta_\ell' \right) \right) 0 \\
       &\quad + \left( \max \left \{0, r_{i_{k,\ell^*}} - \delta_{\ell^*}' \right \} - \left( \delta_{\ell^*+1}' - \delta_{\ell^*}' \right) \right)1\\
           &= \delta_{L_k+1}' + \max \left \{0, r_{i_{k,\ell^*}} - \delta_{\ell^*}' \right \} - \left( \delta_{\ell^*+1}' - \delta_{\ell^*}' \right) \\
           &= \delta_{L_k+1}' + \left(  r_{i_{k,\ell^*}} - \delta_{\ell^*}' \right) - \left( \delta_{\ell^*+1}' - \delta_{\ell^*}' \right) \\
           &= \delta_{L_k+1}' + r_{i_{k,\ell^*}}  - \delta_{\ell^*+1}'  \\
           &= \delta_{L_k+1}' + r_{i_{k,\ell^*}}  - \delta_{\ell^*}'  \\
        &> \delta_{\ell^*}' + r_{i_{k,\ell^*}}  - \delta_{\ell^*}'  \\
        &=r_{i_{k,\ell^*}} 
    \end{align*}
    The first equality follows from the construction of $\bx$, and the second equality follows from algebra. The third equality  follows from Claim~\ref{claim:newslyadded}. The fourth equality follows from algebra. The fifth equality follows from the fact that $\delta'_{\ell^*} < r_{i_{k,\ell^*}}$ (Claim~\ref{claim:newslyadded}) and the fact that $\ell^*  \in \{2,\ldots,L_k\}$ (Claim~\ref{claim:ell_star_between_two_and_L_k}), which together with the fact that $\bdelta'$ satisfies Property~\ref{property:gap} implies that $\delta'_{\ell^*+1} = \delta'_{\ell^*}$. The inequality follows from the assumption that $\delta_{\ell^*}' < \delta'_{L_k+1}$, and the sixth equality follows from algebra.    Moreover, we observe that 
       \begin{align}
        J_k(\bx,\bdelta)
       &= \delta_{L_k+1}+ \sum_{\ell \in \{1,\ldots,L_k\} \setminus \{\ell^*\}}  \left( \max \left \{0, r_{i_{k,\ell}} - \delta_\ell \right \} - \left( \delta_{\ell+1} - \delta_\ell \right) \right) 0 \notag \\
       &\quad + \left( \max \left \{0, r_{i_{k,\ell^*}} - \delta_{\ell^*}\right \} - \left( \delta_{\ell^*+1} - \delta_{\ell^*} \right) \right)1 \notag\\
           &= \delta_{L_k+1} + \max \left \{0, r_{i_{k,\ell^*}} - \delta_{\ell^*} \right \} - \left( \delta_{\ell^*+1} - \delta_{\ell^*} \right)  \notag \\
    &= \delta_{L_k+1} + \left( r_{i_{k,\ell^*}} - \delta_{\ell^*} \right) - \left( \delta_{\ell^*+1} - \delta_{\ell^*} \right)  \notag \\
    &= \delta_{L_k+1} + \left( r_{i_{k,\ell^*}} - \delta_{L_k+1}\right) - \left( \delta_{L_k+1} -\delta_{L_k+1} \right)  \notag \\
        &=  r_{i_{k,\ell^*}}  \notag 
    \end{align}
    The first equality follows from the construction of $\bx$, and the second equality follows from algebra. The third equality  follows from Claim~\ref{claim:newslyadded}. The fourth equality follows from the fact that $\bdelta$ satisfies Property~\ref{property:reverse} and the fact that $\ell^* = T_k(\bdelta)$, which imply that $\delta_{\ell^*} = \delta_{\ell^*+1} = \cdots = \delta_{L_k+1}$. The fifth equality follows from algebra.  Combining the above reasoning, we have shown that 
    \begin{align*}
        J_k(\bx,\bdelta') > r_{i_{k,\ell^*}} = J_k(\bx,\bdelta).
    \end{align*}
   Since $\bx \in \mathcal{X}$, we have shown that $\bdelta'$ does not dominate $\bdelta$.
\halmos \end{proof}

\subsection{Case 5} \label{appx:sufficient:case5}
 Our fifth case consists of proving the following proposition. 

\begin{proposition} \label{prop:case5} 
If $\delta_{L_k+1}' < \delta_{L_k+1}$ and $\delta_{T_k(\bdelta)}' = \delta'_{L_k+1}$, then $\bdelta'$ does not dominate $\bdelta$.
\end{proposition}
\begin{proof}{Proof of Proposition~\ref{prop:case5}.}
Assume that   $\delta_{L_k+1}' < \delta_{L_k+1}$ and $\delta_{T_k(\bdelta)}' = \delta'_{L_k+1}$. For notational convenience, let $\ell^* \triangleq T_k(\bdelta)$. 
We will utilize the following three intermediary claims.
\begin{claim}\label{claim:man_this_is_tedious}
    $\delta_{\ell^*}' = \cdots =   \delta_{L_k+1}' < \delta_{L_k+1}  = \cdots =   \delta_{\ell^*}  \le r_{i_{k,\ell^*}} $.
\end{claim}
\begin{proof}{Proof of Claim~\ref{claim:man_this_is_tedious}.}
The fact that $\delta'_{\ell^*} = \cdots = \delta'_{L_k+1}$ follows from the fact that $\bdelta' \in \Delta_k$ and from the assumption that $\delta'_{\ell^*} = \delta'_{L_k+1}$. The fact that $\delta'_{L_k+1} < \delta_{L_k+1}$ is an assumption. The fact that $\delta_{L_k+1} = \cdots = \delta_{\ell^*}$ follows from the fact that $\ell^* = T_k(\bdelta)$ and from the assumption that $\bdelta$ satisfies Property~\ref{property:reverse}. The fact that $\delta_{\ell^*} \le r_{i_{k,\ell^*}}$ follows from the fact that $\ell^* = T_k(\bdelta)$ and from the definition of $T_k(\bdelta)$. This completes our proof of Claim~\ref{claim:man_this_is_tedious}. 
\halmos \end{proof}
\begin{claim}\label{claim:bob_dylan}
$\ell^* \in \{2,\ldots, L_k\}$. 
\end{claim}
\begin{proof}{Proof of Claim~\ref{claim:bob_dylan}.} It follows from the fact that $\bdelta$ satisfies Property~\ref{property:notspecial} and from the fact that $\ell^* = T_k(\bdelta)$ that $\ell^* \ge 2$. Now suppose for the sake of developing a contradiction that $\ell^* = L_k+1$. If that were true, then it would follow from the fact that $\ell^* = T_k(\bdelta)$ and from the definition of $T_k(\bdelta)$ that $\delta_{L_k+1} \le r_{i_{k,L_k+1}} = 0$. However, we assumed that $\delta'_{L_k+1} < \delta_{L_k+1}$, which would imply that $\delta'_{L_k+1} < 0$, which contradicts the assumption that $\bdelta' \in \Delta_k$. Therefore, we have proven by contradiction that $\ell^* \le L_k$, which completes our proof of Claim~\ref{claim:bob_dylan}. 
\halmos \end{proof}
\begin{claim} \label{claim:hello_hassaan}
    If $\bx \in \mathcal{X}$ is a vector that satisfies $x_{i_{k,\ell^*}} = 1$ and $x_{i_{k,\ell}} = 0$ $\forall \ell \in \{\ell^*+1,\ldots,L_k\}$, then 
    \begin{align*}
   &\delta_{L_k+1}+ \sum_{\ell=\ell^*}^{L_k}  \left( \max \left \{0, r_{i_{k,\ell}} - \delta_\ell \right \} - \left( \delta_{\ell+1} - \delta_\ell \right) \right) x_{i_{k,\ell}}\notag  \\
    &=  \delta_{L_k+1}'+ \sum_{\ell=\ell^*}^{L_k}  \left( \max \left \{0, r_{i_{k,\ell}} - \delta_\ell' \right \} - \left( \delta_{\ell+1}' - \delta_\ell' \right) \right) x_{i_{k,\ell}}\\
    &= r_{i_{k,\ell^*}}
    \end{align*}
\end{claim}
\begin{proof}{Proof of Claim~\ref{claim:hello_hassaan}.}
 Let $\bx \in \mathcal{X}$ be a vector that satisfies $x_{i_{k,\ell^*}} = 1$ and $x_{i_{k,\ell}} = 0$ for all $\ell \in \{\ell^*+1,\ldots,L_k\}$. Then
 \begin{align*}
     &\delta_{L_k+1}+ \sum_{\ell=\ell^*}^{L_k}  \left( \max \left \{0, r_{i_{k,\ell}} - \delta_\ell \right \} - \left( \delta_{\ell+1} - \delta_\ell \right) \right) x_{i_{k,\ell}}\\
     &=  \delta_{L_k+1}+ \sum_{\ell = \ell^*+1}^{L_k}  \left( \max \left \{0, r_{i_{k,\ell}} - \delta_\ell \right \} - \left( \delta_{\ell+1} - \delta_\ell \right) \right) 0 \notag \\
       &\quad + \left( \max \left \{0, r_{i_{k,\ell^*}} - \delta_{\ell^*}\right \} - \left( \delta_{\ell^*+1} - \delta_{\ell^*} \right) \right)1 \notag\\
       &= \delta_{L_k+1} + \max \left \{0, r_{i_{k,\ell^*}} - \delta_{\ell^*}\right \} - \left( \delta_{\ell^*+1} - \delta_{\ell^*} \right) \\
    &= \delta_{L_k+1} + \left( r_{i_{k,\ell^*}} - \delta_{\ell^*} \right) - \left( \delta_{\ell^*+1} - \delta_{\ell^*} \right) \\
        &= \delta_{L_k+1} + \left( r_{i_{k,\ell^*}} - \delta_{L_k+1} \right) - \left( \delta_{L_k+1}- \delta_{L_k+1} \right) \\
        &= r_{i_{k,\ell^*}}
 \end{align*}
The first equality follows from the construction of $\bx$ and from  Claim~\ref{claim:bob_dylan} (which implies that $\ell^* \le L_k$).  The second equality follows from algebra. The third equality follows from the fact that $\delta_{\ell^*} \le r_{i_{k,\ell^*}}$ (Claim~\ref{claim:man_this_is_tedious}), and the fourth equality follows from the fact that $\delta_{\ell^*} = \cdots = \delta_{L_k+1}$ (Claim~\ref{claim:man_this_is_tedious}). The fifth equality follows from algebra.  We similarly observe that
 \begin{align*}
     &\delta_{L_k+1}'+ \sum_{\ell=\ell^*}^{L_k}  \left( \max \left \{0, r_{i_{k,\ell}} - \delta_\ell' \right \} - \left( \delta_{\ell+1}' - \delta_\ell' \right) \right) x_{i_{k,\ell}}\\
     &=  \delta_{L_k+1}'+ \sum_{\ell = \ell^*+1}^{L_k}  \left( \max \left \{0, r_{i_{k,\ell}} - \delta_\ell' \right \} - \left( \delta_{\ell+1}' - \delta_\ell' \right) \right) 0 \notag \\
       &\quad + \left( \max \left \{0, r_{i_{k,\ell^*}} - \delta_{\ell^*}'\right \} - \left( \delta_{\ell^*+1}' - \delta_{\ell^*}' \right) \right)1 \notag\\
       &= \delta_{L_k+1}' + \max \left \{0, r_{i_{k,\ell^*}} - \delta_{\ell^*}'\right \} - \left( \delta_{\ell^*+1}' - \delta_{\ell^*}' \right) \\
    &= \delta_{L_k+1}' + \left( r_{i_{k,\ell^*}} - \delta_{\ell^*}' \right) - \left( \delta_{\ell^*+1}' - \delta_{\ell^*}' \right) \\
        &= \delta_{L_k+1}'+ \left( r_{i_{k,\ell^*}} - \delta_{L_k+1}' \right) - \left( \delta_{L_k+1}'- \delta_{L_k+1}' \right) \\
        &= r_{i_{k,\ell^*}}
 \end{align*}
 The first equality follows from the construction of $\bx$ and from  Claim~\ref{claim:bob_dylan} (which implies that $\ell^* \le L_k$).  The second equality follows from algebra. The third equality follows from the fact that $\delta_{\ell^*}' < r_{i_{k,\ell^*}}$ (Claim~\ref{claim:man_this_is_tedious}), and the fourth equality follows from the fact that $\delta_{\ell^*}' = \cdots = \delta_{L_k+1}'$ (Claim~\ref{claim:man_this_is_tedious}). The fifth equality follows from algebra.  
That concludes the proof of Claim~\ref{claim:hello_hassaan}. 
\halmos \end{proof}

In view of the above intermediary claims, the remainder of the proof of Proposition~\ref{prop:case5} is split into two parts.
\begin{itemize}
    \item \textbf{Part 1}: Suppose that $\delta_{2}' < \delta_2$. In that case, let  $\bx$ be the vector that is defined for each $\ell \in [N+1]$ as
    \begin{align*}
        x_{i_{k,\ell}} &= \begin{cases}
            1,&\text{if } \ell \in \{1, \ell^*,L_k+1\},\\
            0,&\text{otherwise}
        \end{cases}
    \end{align*}
We note in the above definition that the set $\{1,\ell^*,L_k+1\}$ is well defined because $\ell^* \in \{2,\ldots,L_k\}$ (Claim~\ref{claim:bob_dylan}).   We observe that $\bx \in \mathcal{X}$ and that
    \begin{align*}
        J_k(\bx,\bdelta') &= r_{i_{k,\ell^*}} + \sum_{\ell =1}^{\ell^*-1}  \left( \max \left \{0, r_{i_{k,\ell}} - \delta_\ell' \right \} - \left( \delta_{\ell+1}' - \delta_\ell' \right) \right) x_{i_{k,\ell}} \\
 &= r_{i_{k,\ell^*}} + \left( \max \left \{0, r_{i_{k,1}} - \delta_1' \right \} - \left( \delta_{2}' - \delta_1' \right) \right)  \\
  &= r_{i_{k,\ell^*}} +  \left( r_{i_{k,1}} - \delta_1' \right) - \left( \delta_{2}' - \delta_1' \right)  \\
    &= r_{i_{k,\ell^*}} +  r_{i_{k,1}} - \delta_2' \\
    &>  r_{i_{k,\ell^*}} + r_{i_{k,1}} - \delta_2 
    \end{align*}
The first equality follows from the definition of $J_k(\bx,\bdelta')$ and from Claim~\ref{claim:hello_hassaan}.  The second equality follows from the construction of $\bx$. The third equality follows from the fact that $\bdelta'$ satisfies Property~\ref{property:forward} (which implies that $\delta_1' \le r_{i_{k,1}}$). The fourth equality follows from algebra.  The inequality follows from the supposition that $\delta_2' < \delta_2$.  Moreover, we observe that 
        \begin{align*}
        J_k(\bx,\bdelta) &= r_{i_{k,\ell^*}} + \sum_{\ell =1}^{\ell^*-1}  \left( \max \left \{0, r_{i_{k,\ell}} - \delta_\ell \right \} - \left( \delta_{\ell+1} - \delta_\ell \right) \right) x_{i_{k,\ell}} \\
 &= r_{i_{k,\ell^*}} + \left( \max \left \{0, r_{i_{k,1}} - \delta_1 \right \} - \left( \delta_{2} - \delta_1 \right) \right)  \\
  &= r_{i_{k,\ell^*}} +  \left( r_{i_{k,1}} - \delta_1 \right) - \left( \delta_{2} - \delta_1\right)  \\
    &=  r_{i_{k,\ell^*}} +  r_{i_{k,1}} - \delta_2
    \end{align*}
The first equality follows from the definition of $J_k(\bx,\bdelta)$ and from Claim~\ref{claim:hello_hassaan}.    The second equality follows from the construction of $\bx$. The third equality follows from the fact that $\bdelta$ satisfies Property~\ref{property:forward} (which implies that $\delta_1 \le r_{i_{k,1}}$). The fourth equality follows from algebra. 

    Combining the above reasoning, we have shown that      \begin{align*}
        J_k(\bx,\bdelta') >  r_{i_{k,\ell^*}} + r_{i_{k,1}} - \delta_2  \ge J_k(\bx,\bdelta).
    \end{align*}
   Since $\bx \in \mathcal{X}$, we have shown that $\bdelta'$ does not dominate $\bdelta$ in Part 1. 

      \item \textbf{Part 2}: Suppose that $\delta_{2}' \ge \delta_2$. In that case, define the set 
      \begin{align*}
          \mathcal{L} \triangleq \left \{\ell \in \{2,\ldots,\ell^*-1\}: \delta_{\ell} < \delta_{\ell+1}\right \} 
      \end{align*}
      We begin by proving that $\mathcal{L}$ is nonempty. Indeed, it follows from Claim~\ref{claim:bob_dylan}  that  $\ell^* \ge 2$. Moreover, it follows from Claim~\ref{claim:man_this_is_tedious} that $\delta_{\ell^*}' < \delta_{\ell^*}$. Combining the inequalities $\ell^* \ge 2$ and $\delta_{\ell^*}' < \delta_{\ell^*}$ together with the supposition that $\delta'_2 \ge \delta_2$, we conclude that $\ell^* \ge 3$. We have thus shown that the set $\{2,\ldots,\ell^*-1\}$ is nonempty. Furthermore, it follows from the fact that $\bdelta' \in \Delta_k$ and from Claim~\ref{claim:man_this_is_tedious} that  $\delta_2' \le  \cdots \le \delta_{\ell^*}' < \delta_{\ell^*}$. It thus follows from the supposition that $\delta_2 \le  \delta_2'$ that there must exist at least one $\ell \in \{2,\ldots,\ell^*-1\}$ that satisfies $\delta_\ell < \delta_{\ell+1}$. We have thus shown that $\mathcal{L}$ is nonempty. 
      
      Let  $\bx$ be the vector that is defined for each $\ell \in [N+1]$ as
    \begin{align*}
        x_{i_{k,\ell}} &= \begin{cases}
            1,&\text{if } \ell \in \mathcal{L} \cup  \{ \ell^*,L_k+1\},\\
            0,&\text{otherwise}
        \end{cases}
    \end{align*}
 We observe that $\bx \in \mathcal{X}$. We also observe that
  \begin{align}
      &\sum_{\ell=1}^{\ell^*-1}  \left( \max \left \{0, r_{i_{k,\ell}} - \delta_\ell' \right \} - \left( \delta_{\ell+1}' - \delta_\ell' \right) \right) x_{i_{k,\ell}} \notag \\
      &=  \sum_{\ell \in \mathcal{L}}  \left( \max \left \{0, r_{i_{k,\ell}} - \delta_\ell' \right \} - \left( \delta_{\ell+1}' - \delta_\ell' \right) \right)  \notag \\     
      &\ge - \sum_{\ell \in \mathcal{L}}   \left( \delta_{\ell+1}' - \delta_\ell'  \right) \notag  \\
&\ge - \sum_{\ell =2}^{\ell^*-1}   \left( \delta_{\ell+1}' - \delta_\ell'  \right) \notag \\
&=  - \left( \delta_{\ell^*}' - \delta_2'\right) \notag \\
&>  - \left( \delta_{\ell^*} - \delta_2\right) \label{line:almost_there_happy}
 \end{align}
The first equality follows from the construction of $\bx$. The first inequality  follows from algebra. The second inequality follows from the fact that $\bdelta' \in \Delta_k$ and  the fact that $\mathcal{L} \subseteq \{2,\ldots,\ell^*-1\}$. The second equality follows from algebra.  The third inequality follows from Claim~\ref{claim:man_this_is_tedious} (which implies that $\delta'_{\ell^*} < \delta_{\ell^*}$) and from the supposition that  $\delta'_2 \ge \delta_2$. We further observe that 
  \begin{align}
      &\sum_{\ell=1}^{\ell^*-1}  \left( \max \left \{0, r_{i_{k,\ell}} - \delta_\ell \right \} - \left( \delta_{\ell+1} - \delta_\ell \right) \right) x_{i_{k,\ell}}\notag  \\
      &=  \sum_{\ell \in \mathcal{L}}  \left( \max \left \{0, r_{i_{k,\ell}} - \delta_\ell \right \} - \left( \delta_{\ell+1} - \delta_\ell \right) \right) \notag   \\     
      &=  -\sum_{\ell \in \mathcal{L}}   \left( \delta_{\ell+1} - \delta_\ell  \right) \notag  \\
&=  -\sum_{\ell =2}^{\ell^*-1}   \left( \delta_{\ell+1} - \delta_\ell  \right) \notag  \\
&= - \left(  \delta_{\ell^*} - \delta_2\right) \label{line:almost_there_happy_2}
 \end{align}
The first equality follows from the construction of $\bx$. The second equality follows from the fact that $\bdelta$ satisfies Property~\ref{property:gap} and from the fact that $\delta_\ell < \delta_{\ell+1}$ for all $\ell \in \mathcal{L} \subseteq \{2,\ldots,\ell^*-1\}$, which implies that $\delta_\ell \ge r_{i_{k,\ell}}$ for all $\ell \in \mathcal{L}$. The third equality follows from the fact that $\delta_{\ell} = \delta_{\ell+1}$ for all $\ell \in \{2,\ldots,\ell^*-1\} \setminus \mathcal{L}$. The fourth equality follows from algebra.

     Combining the above reasoning, we have shown that   
         \begin{align*}
        &J_k(\bx,\bdelta') - J_k(\bx,\bdelta) \\
        &= \sum_{\ell =1}^{\ell^*-1}  \left( \max \left \{0, r_{i_{k,\ell}} - \delta_\ell' \right \} - \left( \delta_{\ell+1}' - \delta_\ell' \right) \right) x_{i_{k,\ell}} - \sum_{\ell =1}^{\ell^*-1}  \left( \max \left \{0, r_{i_{k,\ell}} - \delta_\ell \right \} - \left( \delta_{\ell+1} - \delta_\ell \right) \right) x_{i_{k,\ell}} \\
       &> - \left(\delta_{\ell^*} - \delta_2 \right)  - \left( - \left(\delta_{\ell^*} - \delta_2 \right) \right)\\
       &= 0
    \end{align*}
    where the first equality follows from Claim~\ref{claim:hello_hassaan} and the definitions of $J_k(\bx,\bdelta')$ and $J_k(\bx,\bdelta)$, the inequality follows from lines~\eqref{line:almost_there_happy} and \eqref{line:almost_there_happy_2}, and the second equality follows from algebra.

   Since $\bx \in \mathcal{X}$, we have shown that $\bdelta'$ does not dominate $\bdelta$ in Part 2.

\end{itemize}
Because Parts 1 and 2 are exhaustive, our proof of Proposition~\ref{prop:case5} is complete. 
\halmos \end{proof}

\section{Proof of Proposition~\ref{prop:subroutine:1}}
\label{appx:proof:subroutine:1}
Let $\bdelta \in \Delta_k$. We recall from Section~\ref{sec:main_alg:reform} that $\mathcal{R}_k$ is the set of unique revenues among the products that are not less preferred by ranking $k$ to the no-purchase option, and  that $\bar{r}_k \triangleq \max \{ r \in \mathcal{R}_k\}$ denotes the highest revenue in $\mathcal{R}_k$. Following the notation from Subroutine 1, let  $\bar{r}_k' \triangleq \max \{r \in \mathcal{R}_k: r < \bar{r}_k \}$ denote the second-highest revenue  in $\mathcal{R}_k$. It is shown in Section~\ref{sec:main_alg:reform} that $| \mathcal{R}_k| \ge 2$, which implies that $\bar{r}_k'$ is always finite and that there always exists $\ell \in \{1,\ldots,L_k+1\}$ that satisfies $r_{i_{k,\ell}} = \bar{r}_k'$. 

If $T_k(\bdelta) > 1$, then it follows from Subroutine 1 that $\bdelta' = \bdelta$, which implies that $\bdelta'$ satisfies $\bdelta' \in \Delta_k$ and satisfies Property~\ref{property:notspecial}. Therefore, assume from this point onward that $T_k(\bdelta) = 1$. It follows from that assumption and  from the definition of $T_k(\bdelta)$ that $\delta_1 \le r_{i_{k,1}}$ and that $r_{i_{k,\ell}} < \delta_\ell \le \bar{r}_k$ for all $\ell \in \{2,\ldots,L_k+1\}$. For those inequalities to be true, it must be the case that $r_{i_{k,1}} = \bar{r}_k$ and it must be the case that $r_{i_{k,2}},\ldots,r_{i_{k,L_k+1}} \le \bar{r}'_k$. 

Now let $\bdelta'$ be defined as in Subroutine 1; that is, let $\bdelta'$ be the vector defined by
\begin{align*}
    \delta'_\ell \triangleq \min \{\delta_\ell, \bar{r}_k'\} \quad \forall \ell \in \{1,\ldots,L_k+1\}
\end{align*}
It follows readily from the fact that $\bdelta \in \Delta_k$  that $\bdelta' \in \Delta_k$. Furthermore, it follows from the fact that $r_{i_{k,2}},\ldots,r_{i_{k,L_k+1}} \le \bar{r}'_k$ that there exists  $\ell \in \{2,\ldots,L_k+1\}$ such that $r_{i_{k,\ell}} = \bar{r}_k'$. It thus follows  from the construction of $\bdelta'$ that $\delta'_\ell \le r_{i_{k,\ell}}$ for that $\ell$, which implies that $T_k(\bdelta') > 1$. We have thus shown that $\bdelta'$ satisfies Property~\ref{property:notspecial}. 

The remainder of the proof of Proposition~\ref{prop:subroutine:1} is organized as follows. In Appendix~\ref{appx:proof:subroutine:1:weakdominate}, we show that   $J_k(\bx,\bdelta') \le J_k(\bx,\bdelta)$ for all $\bx \in \mathcal{X}^c$. In Appendix~\ref{appx:proof:subroutine:1:dominate}, we show that there always exists $\bar{\bx} \in \mathcal{X}^c$ that satisfies $J_k(\bar{\bx},\bdelta') < J_k(\bar{\bx},\bdelta)$. The combination of 
Appendices~\ref{appx:proof:subroutine:1:weakdominate} and \ref{appx:proof:subroutine:1:dominate} thus implies that  $\bdelta'$ dominates $\bdelta$.

\subsection[\texorpdfstring{Proof that $J_k(\bx,\bdelta') \le J_k(\bx,\bdelta)$ for all $\bx \in \mathcal{X}^c$}{Text}]{Proof that $J_k(\bx,\bdelta') \le J_k(\bx,\bdelta)$ for all $\bx \in \mathcal{X}^c$} \label{appx:proof:subroutine:1:weakdominate}
In the present Appendix~\ref{appx:proof:subroutine:1:weakdominate}, we show that  
 $J_k(\bx,\bdelta') \le J_k(\bx,\bdelta)$ for all $\bx \in \mathcal{X}^c$. To this end, consider any $\bx \in \mathcal{X}^c$. We first observe that
\begin{align*}
J_k(\bx,\bdelta) &= \delta_{L_k+1}+ \sum_{\ell=1}^{L_k}  \left( \max \left \{0, r_{i_{k,\ell}} - \delta_\ell \right \} - \left( \delta_{\ell+1} - \delta_\ell \right) \right) x_{i_{k,\ell}}\\
&= \delta_{L_k+1}+ \sum_{\ell=2}^{L_k}  \left( \max \left \{0, r_{i_{k,\ell}} - \delta_\ell \right \} - \left( \delta_{\ell+1} - \delta_\ell \right) \right) x_{i_{k,\ell}}  +  \left( \max \left \{0, r_{i_{k,1}} - \delta_1  \right \} - \left( \delta_{2} - \delta_1 \right) \right) x_{i_{k,1}} \\
&= \delta_{L_k+1}+ \sum_{\ell=2}^{L_k}  \left( 0- \left( \delta_{\ell+1} - \delta_\ell \right) \right) x_{i_{k,\ell}}  +  \left( \left( r_{i_{k,1}} - \delta_1  \right) - \left( \delta_{2} - \delta_1 \right) \right) x_{i_{k,1}}\\
&= \delta_{L_k+1}- \sum_{\ell=1}^{L_k}   \left( \delta_{\ell+1} - \delta_\ell \right) x_{i_{k,\ell}}  +  \left( r_{i_{k,1}} - \delta_{1}  \right)  x_{i_{k,1}}
\end{align*}
where the first equality is the definition of $J_k(\bx,\bdelta)$, and the second equality follows from algebra. The third equality follows from the fact that $r_{i_{k,\ell}} < \delta_\ell$ for all $\ell \in \{2,\ldots,L_k+1\}$ and from the fact that $\delta_1 \le r_{i_{k,1}}$. The fourth equality follows from algebra. Moreover, we observe that
\begin{align*}
J_k(\bx,\bdelta') &= \delta_{L_k+1}'+ \sum_{\ell=1}^{L_k}  \left( \max \left \{0, r_{i_{k,\ell}} - \delta_\ell' \right \} - \left( \delta_{\ell+1}' - \delta_\ell' \right) \right) x_{i_{k,\ell}}\\
&= \min \left \{ \bar{r}_k', \delta_{L_k+1}\right \}+ \sum_{\ell=2}^{L_k}  \left( \max \left \{0, r_{i_{k,\ell}} - \min \left \{ \bar{r}_k',  \delta_\ell \right \} \right \} - \left( \min \left \{ \bar{r}_k', \delta_{\ell+1} \right \} - \min \left \{ \bar{r}_k', \delta_\ell \right \} \right) \right) x_{i_{k,\ell}}\\
&\quad +  \left( \max \left \{0, r_{i_{k,1}} - \min \left \{ \bar{r}_k',  \delta_1 \right \} \right \} - \left( \min \left \{ \bar{r}_k', \delta_{2} \right \} - \min \left \{ \bar{r}_k', \delta_1 \right \} \right) \right) x_{i_{k,1}}\\
&= \bar{r}_k'+ \sum_{\ell=2}^{L_k}  \left( \max \left \{0, r_{i_{k,\ell}} - \min \left \{ \bar{r}_k',  \delta_\ell \right \} \right \} - \left( \min \left \{ \bar{r}_k', \delta_{\ell+1} \right \} - \min \left \{ \bar{r}_k', \delta_\ell \right \} \right) \right) x_{i_{k,\ell}}\\
&\quad +  \left( \max \left \{0, r_{i_{k,1}} - \min \left \{ \bar{r}_k',  \delta_1 \right \} \right \} - \left( \min \left \{ \bar{r}_k', \delta_{2} \right \} - \min \left \{ \bar{r}_k', \delta_1 \right \} \right) \right) x_{i_{k,1}}\\
&= \bar{r}_k'+ \sum_{\ell=2}^{L_k}  \left( 0 - \left( \min \left \{ \bar{r}_k', \delta_{\ell+1} \right \} - \min \left \{ \bar{r}_k', \delta_\ell \right \} \right) \right) x_{i_{k,\ell}}\\
&\quad +  \left( \max \left \{0, r_{i_{k,1}} - \min \left \{ \bar{r}_k',  \delta_1 \right \} \right \} - \left( \min \left \{ \bar{r}_k', \delta_{2} \right \} - \min \left \{ \bar{r}_k', \delta_1 \right \} \right) \right) x_{i_{k,1}}\\
&= \bar{r}_k'- \sum_{\ell=2}^{L_k}  \left( \min \left \{ \bar{r}_k', \delta_{\ell+1} \right \} - \min \left \{ \bar{r}_k', \delta_\ell \right \} \right) x_{i_{k,\ell}}\\
&\quad +  \left( \left(  r_{i_{k,1}} -  \min \left \{ \bar{r}_k', \delta_1 \right \} \right)  - \left( \min \left \{ \bar{r}_k', \delta_{2} \right \} - \min \left \{ \bar{r}_k', \delta_1 \right \} \right) \right) x_{i_{k,1}}\\
&=\bar{r}_k'-  \sum_{\ell=1}^{L_k} \left( \min \left \{ \bar{r}_k', \delta_{\ell+1} \right \} - \min \left \{ \bar{r}_k', \delta_\ell \right \} \right)   x_{i_{k,\ell}} +  \left(  r_{i_{k,1}} -  \min \left \{\bar{r}_k', \delta_1 \right \} \right) x_{i_{k,1}}
\end{align*}
The first equality is the definition of $J_k(\bx,\bdelta')$.    The second equality follows from the construction of $\bdelta'$. The third equality holds because it must be the case that $\delta_{L_k+1} > \bar{r}_k'$, since if that inequality did not hold, then it would follow from the fact that $\bdelta \in \Delta_k$ that there would exist $\ell \in \{2,\ldots,L_k+1 \}$ such that $r_{i_{k,\ell}} = \bar{r}_k'$ and $\delta_{\ell} \le r_{i_{k,\ell}}$.  The fourth equality holds because $r_{i_{k,\ell}} < \delta_\ell$  and $r_{i_{k,\ell}} \le \bar{r}_k'$ for all $\ell \in \{2,\ldots,L_k+1\}$.  The fifth equality holds because $r_{i_{k,1}} = \bar{r}_k > \bar{r}_k'$ and because $r_{i_{k,1}} \ge \delta_1$. The sixth equality follows from algebra. 

Therefore,
\begin{align}
    &J_k(\bx,\bdelta') - J_k(\bx,\bdelta) \notag \\
    &= \left( \bar{r}_k' - \delta_{L_k+1} \right) -  \sum_{\ell=1}^{L_k} \left( \min \left \{ \bar{r}_k', \delta_{\ell+1} \right \} - \min \left \{ \bar{r}_k', \delta_\ell \right \} - \left( \delta_{\ell+1} - \delta_\ell \right) \right)   x_{i_{k,\ell}} \notag \\
    &\quad + \left( \left(  r_{i_{k,1}} -  \min \left \{\bar{r}_k', \delta_1 \right \} \right) - \left( r_{i_{k,1}} - \delta_{1}  \right) \right) x_{i_{k,1}} \notag \\
    &=  \bar{r}_k'- \delta_{L_k+1}  -  \sum_{\ell=1}^{L_k} \left( \min \left \{ \bar{r}_k'- \delta_{\ell+1},0 \right \} - \min \left \{ \bar{r}_k' - \delta_\ell, 0 \right \}  \right)   x_{i_{k,\ell}} - \min \left \{ \bar{r}'_k - \delta_1, 0 \right \} x_{i_{k,1}} \label{line:proof:subroutine:1:difference_intermediary}
\end{align}
We have two cases to consider.
\begin{itemize}
    \item First, suppose that $\bar{r}'_k < \delta_1 \le \cdots \le \delta_{L_k+1}$. In that case, we have 
\begin{align*}
   \eqref{line:proof:subroutine:1:difference_intermediary} &=  \bar{r}_k'- \delta_{L_k+1}-  \sum_{\ell=1}^{L_k} \left( \left(  \bar{r}_k'- \delta_{\ell+1}\right) - \left(  \bar{r}_k' - \delta_\ell \right)  \right)   x_{i_{k,\ell}} - \left(  \bar{r}'_k - \delta_1 \right)  x_{i_{k,1}}\\
    &=  \bar{r}_k'- \delta_{L_k+1} +   \sum_{\ell=1}^{L_k} \left(  \delta_{\ell+1} - \delta_\ell   \right)   x_{i_{k,\ell}}  - \left(  \bar{r}'_k - \delta_1 \right)  x_{i_{k,1}}\\
 &\le  \bar{r}_k'- \delta_{L_k+1} +   \sum_{\ell=1}^{L_k} \left(  \delta_{\ell+1} - \delta_\ell   \right)  1  - \left(  \bar{r}'_k - \delta_1 \right) x_{i_{k,1}}\\
  &=  \bar{r}_k'- \delta_1 -  \left(  \bar{r}'_k - \delta_1 \right)  x_{i_{k,1}}\\
  &\le  \bar{r}_k'- \delta_1 -  \left(  \bar{r}'_k - \delta_1 \right) 1\\
  &= 0
\end{align*}
where the first equality follows from the supposition that  $\bar{r}'_k < \delta_\ell$ for all $\ell \in \{1,\ldots,L_k+1\}$, the second equality follows from algebra, the first inequality follows from the fact that $\bx \in \mathcal{X}^c$ (which implies that $x_{i_{k,\ell}} \le 1$ for all $\ell \in \{1,\ldots,L_k+1\}$) and the fact that $\bdelta \in \Delta_k$ (which implies that $\delta_{\ell+1} - \delta_\ell \ge 0$ for all $\ell \in \{1,\ldots,L_k\}$). The third equality follows from algebra. The second inequality follows from the fact that  $\bx \in \mathcal{X}^c$ (which implies that $x_{i_{k,1}} \le 1$) and the supposition that $\bar{r}_k' < \delta_1$. The fourth equality follows from algebra. 

\item Second, suppose that  $\delta_1 \le \bar{r}_k'$. In that case,  it follows from the fact that $ \bar{r}_k' < \delta_{L_k+1}$ that there exists $L \in \{1,\ldots,L_k\}$ such that $\delta_L \le \bar{r}'_k < \delta_{L+1}$. In that case, we observe that
\begin{align*}
\eqref{line:proof:subroutine:1:difference_intermediary} 
&=  \bar{r}_k'- \delta_{L_k+1}- \sum_{\ell=1}^{L-1} \left( \min \left \{ \bar{r}_k'- \delta_{\ell+1},0 \right \} - \min \left \{ \bar{r}_k' - \delta_\ell, 0 \right \}  \right)   x_{i_{k,\ell}} \\
 &\quad -  \left( \min \left \{ \bar{r}_k'- \delta_{L+1},0 \right \} - \min \left \{ \bar{r}_k' - \delta_L, 0 \right \}  \right)   x_{i_{k,L}}\\
& \quad - \sum_{\ell=L+1}^{L_k} \left( \min \left \{ \bar{r}_k'- \delta_{\ell+1},0 \right \} - \min \left \{ \bar{r}_k' - \delta_\ell, 0 \right \}  \right)   x_{i_{k,\ell}} - \min \left \{ \bar{r}'_k - \delta_1, 0 \right \} x_{i_{k,1}}\\
 &=  \bar{r}_k'- \delta_{L_k+1}- \sum_{\ell=1}^{L-1} \left( 0 - 0   \right)   x_{i_{k,\ell}} \\
 &\quad -  \left( \left( \bar{r}_k'- \delta_{L+1} \right) - 0  \right)   x_{i_{k,L}}\\
&\quad - \sum_{\ell=L+1}^{L_k} \left( \left( \bar{r}_k'- \delta_{\ell+1}\right) -  \left( \bar{r}_k' - \delta_\ell\right)  \right)   x_{i_{k,\ell}} - \left(  \bar{r}'_k - \delta_1 \right) x_{i_{k,1}} \\
 &=  \bar{r}_k'- \delta_{L_k+1}  +  \left(\delta_{L+1} -  \bar{r}_k' \right)  x_{i_{k,L}} + \sum_{\ell=L+1}^{L_k} \left( \delta_{\ell+1} - \delta_\ell  \right)   x_{i_{k,\ell}} +  \left( \delta_1 -  \bar{r}'_k  \right)x_{i_{k,1}} \\
  &\le  \bar{r}_k'- \delta_{L_k+1}   +  \left(\delta_{L+1} -  \bar{r}_k' \right)    1 + \sum_{\ell=L+1}^{L_k} \left( \delta_{\ell+1} - \delta_\ell  \right)   1 + \left( \delta_1 -  \bar{r}'_k  \right) 0\\
  &=    \bar{r}_k'- \delta_{L_k+1} +  \left(\delta_{L+1} -  \bar{r}_k' \right)   + \left( \delta_{L_k+1} - \delta_{L+1}  \right)   \\
  &= 0
\end{align*}
where the first equality follows from algebra. The second equality follows from fact that $\bdelta \in \Delta_k$ and the fact that $\delta_{L} \le \bar{r}_k' < \delta_{L+1}$, which  together imply that $\delta_1,\ldots,\delta_L \le \bar{r}_k' < \delta_{L+1},\ldots,\delta_{L_k+1}$. The inequality follows from the fact that $\bx \in \mathcal{X}^c$ (which implies that $x_{i_{k,L}},\ldots,x_{i_{k,L_k+1}} \le 1$ and $x_{i_{k,1}} \ge 0$), the fact that $\bdelta \in \Delta_k$ (which implies that $\delta_{\ell+1} - \delta_\ell \ge 0$ for all $\ell \in \{L+1,\ldots,L_k\}$), and the fact that $\delta_1 \le \bar{r}_k' < \delta_{L+1}$. The fourth and fifth equalities follow from algebra.  
\end{itemize}
Since the above two cases are exhaustive, we conclude that $J_k(\bx,\bdelta') - J_k(\bx,\bdelta) \le 0$.
\subsection[\texorpdfstring{Proof that there exists $\bar{\bx} \in \mathcal{X}^c$ that satisfies $J_k(\bar{\bx},\bdelta') < J_k(\bar{\bx},\bdelta)$}{Text}]{Proof that there exists $\bar{\bx} \in \mathcal{X}^c$ that satisfies $J_k(\bar{\bx},\bdelta') < J_k(\bar{\bx},\bdelta)$} \label{appx:proof:subroutine:1:dominate}
In the present Appendix~\ref{appx:proof:subroutine:1:dominate}, we show that  there exists $\bar{\bx} \in \mathcal{X}^c$ that satisfies $J_k(\bar{\bx},\bdelta') < J_k(\bar{\bx},\bdelta)$. To this end, let $\bar{\bx}$  be the vector that is defined for each $\ell \in [N+1]$ as
\begin{align*}
    \bar{x}_{i_{k,\ell}} &= \begin{cases}
        1,&\text{if } \ell = L_k+1,\\
        0,&\text{otherwise}
    \end{cases}
\end{align*}
It follows from the discussion in the beginning of Appendix~\ref{appx:proof:subroutine:1} that there exists $\ell^* \in \{2,\ldots,L_k+1\}$ that satisfies $r_{i_{k,\ell^*}} = \bar{r}_k'$ and that $\delta_{\ell^*} > r_{i_{k,\ell^*}}$. We observe for any such $\ell^*$ that
\begin{align*}
    J_k(\bar{\bx}, \bdelta') &= \delta'_{L_k+1} + \sum_{\ell=1}^{L_k}   \left( \max \left \{0, r_{i_{k,\ell}} - \delta_\ell' \right \} - \left( \delta_{\ell+1}' - \delta_\ell' \right) \right) 0\\
    &= \min \left \{ \delta_{L_k+1}, \bar{r}_k' \right \} \\
    &< \delta_{L_k+1}\\
    &= \delta_{L_k+1} + \sum_{\ell=1}^{L_k}   \left( \max \left \{0, r_{i_{k,\ell}} - \delta_\ell \right \} - \left( \delta_{\ell+1} - \delta_\ell \right) \right) 0\\
    &= J_k(\bar{\bx}, \bdelta)
\end{align*}
The first equality follows from the construction of $\bar{\bx}$. The second equality follows from algebra and from the construction of $\bdelta'$. The inequality follows from the fact that $\bdelta \in \Delta_k$ (which implies that $\delta_{L_k+1} \ge \delta_{\ell^*}$) and the fact that $T_k(\bdelta) = 1$ and $\ell^* \in \{2,\ldots,L_k+1\}$ (which implies that $\delta_{\ell^*} > r_{i_{k,\ell^*}} = \bar{r}_k'$).  The third equality follows from algebra. The fourth equality follows from the construction of $\bar{\bx}$. This completes our proof that there always exists $\bar{\bx} \in \mathcal{X}^c$ that satisfies $J_k(\bar{\bx}, \bdelta') < J_k(\bar{\bx}, \bdelta)$.

\section{Proof of Proposition~\ref{prop:subroutine:2}}
\label{appx:proof:subroutine:2}
Let $\bdelta \in \Delta_k$ satisfy Property~\ref{property:notspecial}, and let $\bdelta'$ be the output of Subroutine 2. If $\delta_1 \le r_{i_{k,1}}$, then it follows from the construction of $\bdelta'$ that $\bdelta' = \bdelta$, which implies that $\bdelta'$ satisfies $\bdelta' \in \Delta_k$ and satisfies Properties~\ref{property:notspecial} and \ref{property:forward}. Therefore, assume from this point onward that $\delta_1 > r_{i_{k,1}}$, in which case we observe that the output $\bdelta'$ of Subroutine 2 satisfies
\begin{align*}
    \delta'_\ell = \begin{cases}
        \delta_\ell,&\text{if } \ell \in \{2,\ldots,L_k+1\},\\
        r_{i_{k,1}},&\text{if } \ell = 1.
    \end{cases}
\end{align*}
We readily observe that $\bdelta' \in \Delta_k$. Moreover, it follows from the fact that $\bdelta$ satisfies Property~\ref{property:notspecial} that $\bdelta'$ satisfies Property~\ref{property:notspecial}. The vector $\bdelta'$ also obviously satisfies Property~\ref{property:forward}.

The remainder of the proof of Proposition~\ref{prop:subroutine:2} is organized as follows. In Appendix~\ref{appx:proof:subroutine:2:weakdominate}, we show that   $J_k(\bx,\bdelta') \le J_k(\bx,\bdelta)$ for all $\bx \in \mathcal{X}^c$. In Appendix~\ref{appx:proof:subroutine:2:dominate}, we show that there always exists $\bar{\bx} \in \mathcal{X}^c$ that satisfies $J_k(\bar{\bx},\bdelta') < J_k(\bar{\bx},\bdelta)$. The combination of  Appendices~\ref{appx:proof:subroutine:2:weakdominate} and \ref{appx:proof:subroutine:2:dominate} thus implies that  $\bdelta'$ dominates $\bdelta$.

\subsection[\texorpdfstring{Proof that $J_k(\bx,\bdelta') \le J_k(\bx,\bdelta)$ for all $\bx \in \mathcal{X}^c$}{Text}]{Proof that $J_k(\bx,\bdelta') \le J_k(\bx,\bdelta)$ for all $\bx \in \mathcal{X}^c$} \label{appx:proof:subroutine:2:weakdominate}
In the present Appendix~\ref{appx:proof:subroutine:2:weakdominate}, we show that  
 $J_k(\bx,\bdelta') \le J_k(\bx,\bdelta)$ for all $\bx \in \mathcal{X}^c$. To this end, consider any $\bx \in \mathcal{X}^c$. We observe that
\begin{align*}
    J_k(\bx,\bdelta') - J_k(\bx,\bdelta) &= \left( \delta_{L_k+1}'+ \sum_{\ell=1}^{L_k}  \left( \max \left \{0, r_{i_{k,\ell}} - \delta_\ell' \right \} - \left( \delta_{\ell+1}' - \delta_\ell' \right) \right) x_{i_{k,\ell}}  \right) \\
    &\quad - \left( \delta_{L_k+1}+ \sum_{\ell=1}^{L_k}  \left( \max \left \{0, r_{i_{k,\ell}} - \delta_\ell \right \} - \left( \delta_{\ell+1} - \delta_\ell \right) \right) x_{i_{k,\ell}}  \right) \\
    &=   \left( \max \left \{0, r_{i_{k,1}} - r_{i_{k,1}} \right \} - \left( \delta_{2} - r_{i_{k,1}} \right) \right) x_{i_{k,1}} -  \left( \max \left \{0, r_{i_{k,1}} - \delta_{1} \right \} - \left( \delta_{2} - \delta_1 \right) \right) x_{i_{k,1}}  \\
    &=   \left( 0- \left( \delta_{2} - r_{i_{k,1}} \right) \right) x_{i_{k,1}} -  \left( 0 - \left( \delta_{2} - \delta_1 \right) \right) x_{i_{k,1}}  \\
&=  \left(  r_{i_{k,1}}   -   \delta_1  \right) x_{i_{k,1}} \\
&\le 0
\end{align*}
where the first equality follows from the definitions of $    J_k(\bx,\bdelta')$ and $J_k(\bx,\bdelta)$, the second equality follows from the construction of $\bdelta'$, the third equality follows from algebra and from the supposition that $\delta_1 > r_{i_{k,1}}$, the fourth equality follows from algebra, and the inequality follows from the supposition that $\delta_1 > r_{i_{k,1}}$ and the fact that $\bx \in \mathcal{X}^c$. We thus conclude that $J_k(\bx,\bdelta') - J_k(\bx,\bdelta) \le 0$. 
%\end{proof}
\subsection[\texorpdfstring{Proof that there exists $\bar{\bx} \in \mathcal{X}^c$ that satisfies $J_k(\bar{\bx},\bdelta') < J_k(\bar{\bx},\bdelta)$}{Text}]{Proof that there exists $\bar{\bx} \in \mathcal{X}^c$ that satisfies $J_k(\bar{\bx},\bdelta') < J_k(\bar{\bx},\bdelta)$} \label{appx:proof:subroutine:2:dominate}
In the present Appendix~\ref{appx:proof:subroutine:2:dominate}, we show that  there exists $\bar{\bx} \in \mathcal{X}^c$ that satisfies $J_k(\bar{\bx},\bdelta') < J_k(\bar{\bx},\bdelta)$. To this end, let $\bar{\bx}$  be the vector that is defined for each $\ell \in [N+1]$ as
\begin{align*}
    \bar{x}_{i_{k,\ell}} &= \begin{cases}
        1,&\text{if } \ell \in \{1, L_k+1\},\\
        0,&\text{otherwise}
    \end{cases}
\end{align*}
It follows from Assumption~\ref{ass:L_k} from Section~\ref{sec:intro:methods} that $L_k \ge 1$, which implies that $1 \neq L_k+1$. We observe that
\begin{align*}
    J_k(\bar{\bx}, \bdelta') &= \delta'_{L_k+1} + \sum_{\ell=2}^{L_k}   \left( \max \left \{0, r_{i_{k,\ell}} - \delta_\ell' \right \} - \left( \delta_{\ell+1}' - \delta_\ell' \right) \right) 0\\
    &\quad +  \left( \max \left \{0, r_{i_{k,1}} - \delta_1' \right \} - \left( \delta_{2}' - \delta_1' \right) \right) 1\\
    &= \delta_{L_k+1} +  \max \left \{0, r_{i_{k,1}} - \delta_1' \right \} - \left( \delta_{2} - \delta_1' \right) \\
        &= \delta_{L_k+1} +  \max \left \{0, r_{i_{k,1}} - \delta_1 \right \} - \left( \delta_{2} - \delta_1' \right) \\
        &< \delta_{L_k+1} +  \max \left \{0, r_{i_{k,1}} - \delta_1 \right \} - \left( \delta_{2} - \delta_1 \right) \\
    &= \delta_{L_k+1} + \sum_{\ell=1}^{L_k}   \left( \max \left \{0, r_{i_{k,\ell}} - \delta_\ell \right \} - \left( \delta_{\ell+1} - \delta_\ell \right) \right) 0\\
        &\quad +  \left( \max \left \{0, r_{i_{k,1}} - \delta_1 \right \} - \left( \delta_{2} - \delta_1 \right) \right) 1\\
    &= J_k(\bar{\bx}, \bdelta)
\end{align*}
The first equality follows from the construction of $\bar{\bx}$. The second equality follows from algebra and from the fact that $\delta'_\ell = \delta_{\ell}$ for all $\ell \in \{2,\ldots,L_k+1\}$. The third equality follows from the fact that $0 = r_{i_{k,1}} - r_{i_{k,1}} = r_{i_{k,1}} - \delta'_1 > r_{i_{k,1}}  - \delta_1$, which implies that $\max \{0, r_{i_{k,1}} - \delta'_1 \} = \max \{0, r_{i_{k,1}} - \delta_1 \}$. The inequality follows from the fact that $\delta_1' < \delta_1$. The fourth equality follows from algebra, and the fifth equality follows from the construction of $\bar{\bx}$. This completes our proof that there always exists $\bar{\bx} \in \mathcal{X}^c$ that satisfies $J_k(\bar{\bx},\bdelta') < J_k(\bar{\bx},\bdelta)$. 
\section{Proof of Proposition~\ref{prop:subroutine:3}}\label{appx:proof:subroutine:3}

Let $\bdelta \in \Delta_k$ satisfy Properties~\ref{property:notspecial} and \ref{property:forward}, and let $\bdelta'$ be the output of Subroutine 3. We will make use of the following lemma.
\begin{lemma} \label{lem:listening_to_music}
    If $\bdelta$ satisfies Property~\ref{property:gap}, then $\bdelta' = \bdelta$.
    \end{lemma}
    \begin{proof}{Proof of Lemma~\ref{lem:listening_to_music}.}
    Suppose that  $\bdelta$ satisfies Property~\ref{property:gap}.
    It follows from the construction of $\bdelta'$ that $\delta'_1 = \delta_1$.  In what follows, we prove that $\delta'_\ell = \delta_\ell$ for all $\ell \in \{2,\ldots,L_k+1\}$ by backward induction on $\ell$. In the base case where $\ell = L_k+1$, it follows immediately from the construction of $\bdelta'$ that $\delta'_{L_k+1} = \delta_{L_k+1}$. Now consider any arbitrary $\ell \in \{2,\ldots,L_k\}$, and assume by induction that $\delta'_{\ell'} = \delta_{\ell'}$ for all $\ell' \in \{\ell+1,\ldots,L_k+1\}$. In the case of $\ell' = \ell$, we have 
    \begin{align*}
          \delta_{\ell}' &= \begin{cases}
   \delta_\ell,&\text{if } \delta_\ell \ge r_{i_{k,\ell}},\\
   \min \left \{ r_{i_{k,\ell}}, \delta_{\ell+1}' \right \}&\text{if } \delta_\ell < r_{i_{k,\ell}}
   \end{cases} \\
&= \begin{cases}
   \delta_\ell,&\text{if } \delta_\ell \ge r_{i_{k,\ell}},\\
   \min \left \{ r_{i_{k,\ell}}, \delta_{\ell+1} \right \}&\text{if } \delta_\ell < r_{i_{k,\ell}}
   \end{cases} \\
   &= \begin{cases}
   \delta_\ell,&\text{if } \delta_\ell \ge r_{i_{k,\ell}},\\
   \min \left \{ r_{i_{k,\ell}}, \delta_{\ell} \right \}&\text{if } \delta_\ell < r_{i_{k,\ell}}
   \end{cases} \\
   &= \delta_\ell
    \end{align*}
    where the first equality follows from the construction of $\bdelta'$, the second equality follows from the induction hypothesis, the third equality follows from the fact that $\ell \in \{2,\ldots,L_k\}$ and from the supposition that $\bdelta$ satisfies Property~\ref{property:gap}, and the fourth equality follows from algebra. Since $\ell \in \{2,\ldots,L_k\}$ was chosen arbitrarily, our proof of Lemma~\ref{lem:listening_to_music} is complete. 
    \halmos \end{proof}

    It follows from Lemma~\ref{lem:listening_to_music} that if $\bdelta$ satisfies Property~\ref{property:gap}, then $\bdelta' = \bdelta$, which implies that $\bdelta'$ satisfies $\bdelta' \in \Delta_k$ and satisfies Properties~\ref{property:notspecial}, \ref{property:forward}, and \ref{property:gap}.  Therefore, assume from this point onward that $\bdelta$ does not satisfy Property~\ref{property:gap}.   Our proof will also make use of the following intermediary lemma. 
\begin{lemma} \label{lem:subroutine3:induction}
    $\delta'_\ell \ge \delta_\ell$ for all $\ell \in \{1,\ldots,L_k+1\}$. 
\end{lemma}
\begin{proof}{Proof of Lemma~\ref{lem:subroutine3:induction}.}
It follows immediately from the construction of $\bdelta'$ that $\delta'_1 = \delta_1$. In what follows, we prove that $\delta'_\ell \ge \delta_\ell$ for all $\ell \in \{2,\ldots,L_k+1\}$ by backwards induction on $\ell$. In the base case where $\ell = L_k+1$, it follows immediately from construction of $\bdelta'$ that $\delta_{L_k+1}' = \delta_{L_k+1}$. Now consider any arbitrary $\ell \in \{2,\ldots,L_k\}$, and assume by induction that $\delta'_{\ell'} \ge \delta_{\ell'}$ for all $\ell' \in \{\ell+1,\ldots,L_k+1\}$. In the case of $\ell' = \ell$, we have
\begin{align*}
   \delta_{\ell}' &= \begin{cases}
   \delta_\ell,&\text{if } \delta_\ell \ge r_{i_{k,\ell}},\\
   \min \left \{ r_{i_{k,\ell}}, \delta_{\ell+1}' \right \}&\text{if } \delta_\ell < r_{i_{k,\ell}}
   \end{cases} \\
&\ge \begin{cases}
   \delta_\ell,&\text{if } \delta_\ell \ge r_{i_{k,\ell}},\\
   \min \left \{ r_{i_{k,\ell}}, \delta_{\ell+1} \right \}&\text{if } \delta_\ell < r_{i_{k,\ell}}
   \end{cases} \\
   &\ge \begin{cases}
   \delta_\ell,&\text{if } \delta_\ell \ge r_{i_{k,\ell}},\\
   \min \left \{ r_{i_{k,\ell}}, \delta_{\ell} \right \}&\text{if } \delta_\ell < r_{i_{k,\ell}}
   \end{cases} \\
&=  \delta_\ell
   \end{align*}
The  first equality follows from the construction of $\bdelta'$. The first inequality follows from the induction hypothesis. The second inequality follows from the fact that $\bdelta \in \Delta_k$ (which implies that $\delta_{\ell+1} \ge \delta_\ell$). The second equality follows from algebra. 
 Since $\ell \in \{2,\ldots,L_k\}$ was chosen arbitrarily, our proof of Lemma~\ref{lem:subroutine3:induction} is complete. 
\halmos \end{proof}

Equipped with Lemma~\ref{lem:subroutine3:induction}, the remainder of the proof of Proposition~\ref{prop:subroutine:3} is organized as follows. In Appendix~\ref{appx:proof:subroutine:3:delta_prime_feasible}, we prove that $\bdelta' \in \Delta_k$. In Appendix~\ref{appx:proof:subroutine:3:property_1}, we prove that $\bdelta'$ satisfies Property~\ref{property:notspecial}. In Appendix~\ref{appx:proof:subroutine:3:property_2}, we prove that $\bdelta'$ satisfies Property~\ref{property:forward}. In Appendix~\ref{appx:proof:subroutine:3:property_3}, we prove that $\bdelta'$ satisfies Property~\ref{property:gap}. In Appendix~\ref{appx:proof:subroutine:3:weakdominate}, we show that   $J_k(\bx,\bdelta') \le J_k(\bx,\bdelta)$ for all $\bx \in \mathcal{X}^c$. In Appendix~\ref{appx:proof:subroutine:3:dominate}, we show that there always exists $\bar{\bx} \in \mathcal{X}^c$ that satisfies $J_k(\bar{\bx},\bdelta') < J_k(\bar{\bx},\bdelta)$. The combination of  Appendices~\ref{appx:proof:subroutine:3:weakdominate} and \ref{appx:proof:subroutine:3:dominate} thus implies that  $\bdelta'$ dominates $\bdelta$. 

\subsection[\texorpdfstring{Proof that $\bdelta'$ satisfies $\bdelta' \in \Delta_k$}{Text}]{Proof that $\bdelta'$ satisfies $\bdelta' \in \Delta_k$} \label{appx:proof:subroutine:3:delta_prime_feasible}

In the present Appendix~\ref{appx:proof:subroutine:3:delta_prime_feasible}, we show that $\bdelta' \in\Delta_k$.  Indeed, it follows from the construction of $\bdelta'$ and from the fact that $\bdelta \in \Delta_k$ that  $\delta'_{L_k+1} = \delta_{L_k+1} \le  \bar{r}_k$, and it follows from the construction of $\bdelta'$ and from the fact that $\bdelta \in \Delta_k$ that $\delta'_1 = \delta_1 \ge 0$. We observe that  $$\delta'_1 = \delta_1 \le \delta_2 \le \delta_2'$$
where the equality follows from the construction of $\bdelta'$, the first inequality follows from the fact that $\bdelta \in \Delta_k$, and the second inequality follows from Lemma~\ref{lem:subroutine3:induction}.  Finally, for all $\ell \in \{2,\ldots,L_k\}$, we have 
\begin{align*}
    \delta_{\ell}' &= \begin{cases}
   \delta_\ell,&\text{if } \delta_\ell \ge r_{i_{k,\ell}},\\
   \min \left \{ r_{i_{k,\ell}}, \delta'_{\ell+1} \right \},&\text{if } \delta_\ell < r_{i_{k,\ell}} 
   \end{cases}\\
   &\le \begin{cases}
   \delta_{\ell+1},&\text{if } \delta_\ell \ge r_{i_{k,\ell}},\\
   \min \left \{ r_{i_{k,\ell}}, \delta'_{\ell+1} \right \},&\text{if } \delta_\ell < r_{i_{k,\ell}} 
   \end{cases}\\
      &\le \begin{cases}
   \delta_{\ell+1}',&\text{if } \delta_\ell \ge r_{i_{k,\ell}},\\
   \min \left \{ r_{i_{k,\ell}}, \delta'_{\ell+1} \right \},&\text{if } \delta_\ell < r_{i_{k,\ell}} 
   \end{cases}\\
         &\le \begin{cases}
   \delta_{\ell+1}',&\text{if } \delta_\ell \ge r_{i_{k,\ell}},\\
 \delta'_{\ell+1},&\text{if } \delta_\ell < r_{i_{k,\ell}} 
   \end{cases}\\
   &=  \delta'_{\ell+1}
\end{align*}
where the first equality follows from the construction of $\bdelta'$, the first inequality follows from the fact that $\bdelta \in \Delta_k$ (which implies that $\delta_\ell \le \delta_{\ell+1}$), the second inequality follows from Lemma~\ref{lem:subroutine3:induction}, the third inequality follows from algebra, and the second equality follows from algebra. We have thus shown that $\bdelta' \in \Delta_k$.

\subsection[\texorpdfstring{Proof that $\bdelta'$ satisfies Property~\ref{property:notspecial}}{Text}]{Proof that $\bdelta'$ satisfies Property~\ref{property:notspecial}} \label{appx:proof:subroutine:3:property_1}

In the present Appendix~\ref{appx:proof:subroutine:3:property_1}, we show that $\bdelta'$ satisfies Property~\ref{property:notspecial}. Indeed, it follows from the fact that $\bdelta$ satisfies Property~\ref{property:notspecial} that $T_k(\bdelta) > 1$, which implies that there exists $\ell^* \in \{2,\ldots,L_k+1\}$ such that $\delta_{\ell^*} \le r_{i_{k,\ell^*}}$. If $\ell^* = L_k+1$, then it follows from the construction of $\bdelta'$ that $\delta'_{\ell^*} = \delta_{\ell^*} \le r_{i_{k,\ell^*}} = 0$, which implies that $T_k(\bdelta') = L_k+1 > 1$. If $\ell^* < L_k+1$, then it follows from the fact that   $\ell^* \in \{2,\ldots,L_k\}$ that
\begin{align*}
    \delta_{\ell^*}' &= \begin{cases}
         \delta_{\ell^*},&\text{if } \delta_{\ell^*} = r_{i_{k,\ell^*}},\\
            \min \left \{ r_{i_{k,\ell^*}}, \delta'_{\ell^*+1} \right \}, &\text{if } \delta_{\ell^*} < r_{i_{k,\ell^*}}
    \end{cases}\\
    &\le r_{i_{k,\ell^*}}
\end{align*}
The equality follows from the construction of $\bdelta'$ and from the definition of $\ell^*$ (which implies that $\delta_{\ell^*} \le r_{i_{k,\ell^*}})$. The inequality follows from algebra. This shows that $T_k(\bdelta') \ge \ell^* \ge 2$, which implies that $\bdelta'$ satisfies Property~\ref{property:notspecial}. This completes our proof that $\bdelta'$ satisfies Property~\ref{property:notspecial}.

\subsection[\texorpdfstring{Proof that $\bdelta'$ satisfies Property~\ref{property:forward}}{Text}]{Proof that $\bdelta'$ satisfies Property~\ref{property:forward}} \label{appx:proof:subroutine:3:property_2}

In the present Appendix~\ref{appx:proof:subroutine:3:property_2}, we show that $\bdelta'$ satisfies Property~\ref{property:forward}. Indeed, we observe that
        \begin{align*}
            \delta'_1 &=  \delta_1 \le r_{i_{k,1}}
        \end{align*}
        where the equality follows from the construction of $\bdelta'$ and the inequality  follows from the fact that $\bdelta$ satisfies Property~\ref{property:forward}. We have thus shown that $\bdelta'$ satisfies Property~\ref{property:forward}.

\subsection[\texorpdfstring{Proof that $\bdelta'$ satisfies Property~\ref{property:gap}}{Text}]{Proof that $\bdelta'$ satisfies Property~\ref{property:gap}} \label{appx:proof:subroutine:3:property_3}

In the present Appendix~\ref{appx:proof:subroutine:3:property_3}, we show that $\bdelta'$ satisfies Property~\ref{property:gap}. Indeed, for each $\ell \in \{2,\ldots,L_k\}$, we observe that 
        \begin{align}
            \delta'_{\ell} &= \begin{cases}
            \delta_{\ell},&\text{if } \delta_\ell \ge r_{i_{k,\ell}}\\
   \min \left \{ r_{i_{k,\ell}}, \delta'_{\ell+1} \right \},&\text{if } \delta_\ell < r_{i_{k,\ell}}
   \end{cases} \notag \\
   &= \begin{cases}
            \delta_{\ell},&\text{if } \delta_\ell \ge r_{i_{k,\ell}}\\
  r_{i_{k,\ell}},&\text{if } \delta_\ell < r_{i_{k,\ell}} \text{ and }  r_{i_{k,\ell}} < \delta'_{\ell+1} \\
  \delta'_{\ell+1},&\text{if } \delta_\ell < r_{i_{k,\ell}} \text{ and }  r_{i_{k,\ell}} \ge \delta'_{\ell+1} 
   \end{cases} \label{line:happyhappyhappy}
        \end{align}
where the first equality follows from the construction of $\bdelta'$ and the second equality follows from algebra. We now consider each of the above three cases. If $\delta_\ell \ge r_{i_{k,\ell}}$, then it follows from line~\eqref{line:happyhappyhappy} that $\delta'_\ell \ge r_{i_{k,\ell}}$, which complies with Property~\ref{property:gap}. If $\delta_\ell < r_{i_{k,\ell}}$ and $r_{i_{k,\ell}} < \delta'_{\ell+1}$, then it follows from line~\eqref{line:happyhappyhappy} that $\delta_\ell' \ge r_{i_{k,\ell}}$, which complies with Property~\ref{property:gap}. If $\delta_\ell < r_{i_{k,\ell}}$ and $r_{i_{k,\ell}} \ge \delta'_{\ell+1}$, then $\delta'_\ell = \delta'_{\ell+1}$, which complies with Property~\ref{property:gap}. 

 Since this reasoning holds for all $\ell \in \{2,\ldots,L_k\}$, we have shown that $\bdelta'$ satisfies Property~\ref{property:gap}. 

\subsection[\texorpdfstring{Proof that $J_k(\bx,\bdelta') \le J_k(\bx,\bdelta)$ for all $\bx \in \mathcal{X}^c$}{Text}]{Proof that $J_k(\bx,\bdelta') \le J_k(\bx,\bdelta)$ for all $\bx \in \mathcal{X}^c$} \label{appx:proof:subroutine:3:weakdominate}
In the present Appendix~\ref{appx:proof:subroutine:3:weakdominate}, we show that  
 $J_k(\bx,\bdelta') \le J_k(\bx,\bdelta)$ for all $\bx \in \mathcal{X}^c$. We will make use of the following intermediary lemma.
\begin{lemma}\label{lem:subroutine3:new}
For all $\ell \in \{1,\ldots,L_k\}$, we have
\begin{align*}
 \max \left \{0, r_{i_{k,\ell}} - \delta_\ell' \right \} - \left( \delta_{\ell+1}' - \delta_\ell' \right) \le  \max \left \{0, r_{i_{k,\ell}} - \delta_\ell \right \} - \left( \delta_{\ell+1} - \delta_\ell \right) 
\end{align*}
\end{lemma}
\begin{proof}{Proof of Lemma~\ref{lem:subroutine3:new}.}
Consider any $\ell \in \{1,\ldots,L_k\}$. If $\ell = 1$, then we observe that
\begin{align*}
  \max \left \{0, r_{i_{k,1}} - \delta_1' \right \} - \left( \delta_{2}' - \delta_1' \right) 
&= \max \left \{0, r_{i_{k,1}} - \delta_1 \right \} - \left( \delta_{2}' - \delta_1 \right) \\
 &\le \max \left \{0, r_{i_{k,1}} - \delta_1 \right \} - \left( \delta_{2} - \delta_1 \right)
\end{align*}
The equality  follows from the construction of $\bdelta'$ (which implies that $\delta'_1 = \delta_1$). The inequality follows Lemma~\ref{lem:subroutine3:induction}. If $\ell \in \{2,\ldots,L_k\}$, then we observe that
\begin{align*}
   & \max \left \{0, r_{i_{k,\ell}} - \delta_\ell' \right \} - \left( \delta_{\ell+1}' - \delta_\ell' \right) \\
   & \le     \max \left \{0, r_{i_{k,\ell}} - \delta_\ell' \right \} - \left( \delta_{\ell+1} - \delta_\ell' \right) \\
   &=\begin{cases}
    \max \left \{0, r_{i_{k,\ell}} - \delta_\ell \right \} - \left( \delta_{\ell+1} - \delta_\ell \right) ,&\text{if } \delta_\ell \ge r_{i_{k,\ell}},\\
  \max \left \{0, r_{i_{k,\ell}} -  \min \left \{ r_{i_{k,\ell}}, \delta'_{\ell+1} \right \} \right \} - \left( \delta_{\ell+1} -  \min \left \{ r_{i_{k,\ell}}, \delta'_{\ell+1} \right \} \right),&\text{if } \delta_\ell < r_{i_{k,\ell}}
   \end{cases} 
   \end{align*}
 where the inequality follows from Lemma~\ref{lem:subroutine3:induction} and the  equality follows from the construction of $\bdelta'$.   If $\delta_{\ell} \ge r_{i_{k,\ell}}$, then we have shown the desired result.  Otherwise, if $\delta_{\ell} < r_{i_{k,\ell}}$, then 
   \begin{align*}
         &\max \left \{0, r_{i_{k,\ell}} -  \min \left \{ r_{i_{k,\ell}}, \delta'_{\ell+1} \right \} \right \} - \left( \delta_{\ell+1} -  \min \left \{ r_{i_{k,\ell}}, \delta'_{\ell+1} \right \} \right)\\
      &=
\left(  r_{i_{k,\ell}} -  \min \left \{ r_{i_{k,\ell}}, \delta'_{\ell+1} \right \} \right) - \left( \delta_{\ell+1} -  \min \left \{ r_{i_{k,\ell}}, \delta'_{\ell+1} \right \} \right) \\
      &=
  r_{i_{k,\ell}} -   \delta_{\ell+1} \\
         &=
  \left( r_{i_{k,\ell}} - \delta_\ell \right)  -   \left( \delta_{\ell+1} - \delta_\ell \right) \\
          &=
 \max \left \{ 0,  r_{i_{k,\ell}} - \delta_\ell \right \}  -   \left( \delta_{\ell+1} - \delta_\ell \right)
   \end{align*}
where the first, second, and third equalities follow from algebra, and the fourth equality follows from the fact that $\delta_{\ell} < r_{i_{k,\ell}}$. Our proof of Lemma~\ref{lem:subroutine3:new} is thus complete.  \halmos 
\end{proof}

Equipped with Lemma~\ref{lem:subroutine3:new}, we now show that  
 $J_k(\bx,\bdelta') \le J_k(\bx,\bdelta)$ for all $\bx \in \mathcal{X}^c$. Indeed, consider any $\bx \in \mathcal{X}^c$. We observe that 
\begin{align*}
    J_k(\bx,\bdelta') &= \delta_{L_k+1}' + \sum_{\ell=1}^{L_k}\left( \max \left \{0, r_{i_{k,\ell}} - \delta_\ell' \right \} - \left( \delta_{\ell+1}' - \delta_\ell' \right) \right) x_{i_{k,\ell}} \notag \\
&= \delta_{L_k+1} + \sum_{\ell=1}^{L_k}\left( \max \left \{0, r_{i_{k,\ell}} - \delta_\ell' \right \} - \left( \delta_{\ell+1}' - \delta_\ell' \right) \right) x_{i_{k,\ell}}\\
&\le  \delta_{L_k+1} + \sum_{\ell=1}^{L_k}\left( \max \left \{0, r_{i_{k,\ell}} - \delta_\ell \right \} - \left( \delta_{\ell+1} - \delta_\ell \right) \right) x_{i_{k,\ell}}\\
&= J_k(\bx,\bdelta)
\end{align*}
where the first equality is the definition of $J_k(\bx,\bdelta')$, the second equality follows from the construction of $\bdelta'$ (which implies that $\delta'_{L_k+1} = \delta_{L_k+1}$),  the inequality follows from Lemma~\ref{lem:subroutine3:new} and the fact that $\bx \in \mathcal{X}^c$ (which implies that $x_{i_{k,\ell}} \ge 0$ for all $\ell \in \{1,\ldots,L_k\}$), and the third equality is the definition of $J_k(\bx,\bdelta)$.   We thus conclude that $J_k(\bx,\bdelta') \le J_k(\bx,\bdelta)$.

\subsection[\texorpdfstring{Proof that there exists $\bar{\bx} \in \mathcal{X}^c$ that satisfies $J_k(\bar{\bx},\bdelta') < J_k(\bar{\bx},\bdelta)$}{Text}]{Proof that there exists $\bar{\bx} \in \mathcal{X}^c$ that satisfies $J_k(\bar{\bx},\bdelta') < J_k(\bar{\bx},\bdelta)$} \label{appx:proof:subroutine:3:dominate}
In the present Appendix~\ref{appx:proof:subroutine:3:dominate}, we show that  there exists $\bar{\bx} \in \mathcal{X}^c$ that satisfies $J_k(\bar{\bx},\bdelta') < J_k(\bar{\bx},\bdelta)$. Indeed, it follows from the assumption that $\bdelta$ does not satisfy Property~\ref{property:gap} and from the fact that $\bdelta'$ satisfies Property~\ref{property:gap} (see Appendix~\ref{appx:proof:subroutine:3:property_3}) that $\bdelta' \neq \bdelta$. For that to be the case, it follows from Lemma~\ref{lem:subroutine3:induction}  that there must exist $\ell \in \{1,\ldots,L_k+1\}$ such that $\delta_{\ell}' >  \delta_{\ell}$. 

Let $\ell^*$ denote the minimum $\ell \in \{1,\ldots,L_k+1\}$ such that $\delta_{\ell}' >  \delta_{\ell}$. 
It follows from the  construction of $\bdelta^*$ that $\delta'_1 = \delta_1$  and $\delta'_{L_k+1} = \delta_{L_k+1}$, which imply that $\ell^* \in \{2,\ldots,L_k\}$.  Let $\bar{\bx}$ be the vector that is defined for each $\ell \in [N+1]$ as
\begin{align*}
    \bar{x}_{i_{k,\ell}} &= \begin{cases}
        1,&\text{if } \ell \in \{\ell^*-1, L_k+1\},\\
        0,&\text{otherwise}
    \end{cases}
\end{align*}
 We observe that
\begin{align*}
    J_k(\bar{\bx}, \bdelta') &= \delta'_{L_k+1} + \sum_{\ell \in \{1,\ldots,L_k\} \setminus \{\ell^*-1\}}   \left( \max \left \{0, r_{i_{k,\ell}} - \delta_\ell' \right \} - \left( \delta_{\ell+1}' - \delta_\ell' \right) \right) 0\\
    &\quad + \left( \max \left \{0, r_{i_{k,\ell^*-1}} - \delta_{\ell^*-1}' \right \} - \left( \delta_{\ell^*}' - \delta_{\ell^*-1}' \right) \right) 1\\
    &= \delta'_{L_k+1} + \left( \max \left \{0, r_{i_{k,\ell^*-1}} - \delta_{\ell^*-1}' \right \} - \left( \delta_{\ell^*}' - \delta_{\ell^*-1}' \right) \right) \\
        &= \delta_{L_k+1} + \left( \max \left \{0, r_{i_{k,\ell^*-1}} - \delta_{\ell^*-1} \right \} - \left( \delta_{\ell^*}' - \delta_{\ell^*-1}\right) \right) \\
        &< \delta_{L_k+1} + \left( \max \left \{0, r_{i_{k,\ell^*-1}} - \delta_{\ell^*-1} \right \} - \left( \delta_{\ell^*} - \delta_{\ell^*-1}\right) \right) \\
         &= \delta_{L_k+1} + \sum_{\ell \in \{1,\ldots,L_k\} \setminus \{\ell^*-1\}}   \left( \max \left \{0, r_{i_{k,\ell}} - \delta_\ell \right \} - \left( \delta_{\ell+1} - \delta_\ell \right) \right) 0\\
    &\quad +\left( \max \left \{0, r_{i_{k,\ell^*-1}} - \delta_{\ell^*-1} \right \} - \left( \delta_{\ell^*} - \delta_{\ell^*-1}\right) \right)1\\
    &= J_k(\bar{\bx}, \bdelta)
\end{align*}
The first equality follows from the construction of $\bar{\bx}$. The second equality follows from algebra. The third equality follows the construction of $\bdelta'$ (which implies that $\delta'_{L_k+1} = \delta_{L_k+1}$) and from the definition of $\ell^*$ and Lemma~\ref{lem:subroutine3:induction} (which imply that $\delta'_{\ell} = \delta_{\ell}$ for all $\ell \in \{1,\ldots,\ell^*-1\}$). The inequality follows from the definition of $\ell^*$ (which implies that $\delta'_{\ell^*} > \delta_{\ell^*}$). The fourth equality follows from algebra, and the fifth equality follows from the construction of $\bar{\bx}$.  This completes our proof that there always exists $\bar{\bx} \in \mathcal{X}^c$ that satisfies $J_k(\bar{\bx}, \bdelta') < J_k(\bar{\bx}, \bdelta)$.
\section{Proof of Proposition~\ref{prop:subroutine:4}} \label{appx:proof:subroutine:4}
Let $\bdelta \in \Delta_k$ satisfy Properties~\ref{property:notspecial}, \ref{property:forward}, and \ref{property:gap}. Let $L$, $\hat{r}$, and $\bdelta'$ be as defined in Subroutine 4, which we repeat below for convenience. 
\begin{align*}
L &\triangleq T_k(\bdelta)\\
    \hat{r}&\triangleq  \max \left \{ \delta_L, \max_{\ell \in \{L+1,\ldots,L_k+1\}} r_{i_{k,\ell}} \right\}\\
    \delta'_\ell &\triangleq \min \left \{ \delta_\ell, \hat{r} \right \} \quad \forall \ell \in \{1,\ldots,L_k+1\}
\end{align*}
Our proof of Proposition~\ref{prop:subroutine:4}  will make use of the following intermediary lemma. 
\begin{lemma} \label{lem:subroutine4:helpful}
    We have  $\delta_\ell' = \delta_\ell$ for all $\ell \in \{1,\ldots,L\}$.
\end{lemma}
\begin{proof}{Proof of Lemma~\ref{lem:subroutine4:helpful}.}
For each $\ell \in \{1,\ldots,L\}$, we observe that
\begin{align*}
    \delta_\ell' &= \min \{\delta_\ell, \hat{r} \}  \\
     &= \min \left \{\delta_\ell, \max \left \{ \delta_L, \max_{\ell \in \{L+1,\ldots,L_k+1\}} r_{i_{k,\ell}} \right\} \right \}\\
     &\ge \min \left \{\delta_\ell,  \delta_L  \right \}\\
    &= \delta_\ell 
\end{align*}
The first equality is the definition of $\delta'_\ell$. The second equality is the definition of $\hat{r}$. The inequality follows from algebra.  The third equality follows from the fact that $\bdelta \in \Delta_k$ and the fact that $\ell \in \{1,\ldots,L\}$, which implies that $\delta_L \ge \delta_\ell$. Moreover, we observe that
\begin{align*}
\delta_\ell' &= \min \left \{ \delta_\ell, \hat{r} \right \} \le \delta_\ell
\end{align*}
where the  equality is the definition of $\delta_\ell'$ and the inequality follows from algebra. Combining the above reasoning, we have shown that $\delta_\ell' = \delta_\ell$, which completes our proof of Lemma~\ref{lem:subroutine4:helpful}. \halmos \end{proof}
We will also make use of the following lemma.
\begin{lemma} \label{lem:listening_to_music_subroutine_4}
    If $\bdelta$ satisfies Property~\ref{property:reverse}, then $\bdelta' = \bdelta$.
    \end{lemma}
    \begin{proof}{Proof of Lemma~\ref{lem:listening_to_music_subroutine_4}.}
    Suppose that  $\bdelta$ satisfies Property~\ref{property:reverse}. Then it follows from the definition of Property~\ref{property:reverse} that $\delta_L = \cdots = \delta_{L_k+1}$. Moreover, it follows from the definition of $T_k(\bdelta)$ that $\delta_{\ell} > r_{i_{k,\ell}}$ for all $\ell \in \{L+1,\ldots,L_k+1\}$. Therefore, we conclude for all $\ell \in \{1,\ldots,L_k+1\}$ that
    \begin{align*}
    \delta_\ell' &= \begin{cases}
        \delta_\ell,&\text{if } \ell \in \{1,\ldots,L\},\\
        \min \left \{ \delta_\ell, \hat{r} \right \},&\text{if } \ell \in \{L+1,\ldots,L_k+1\}
    \end{cases}\\
    &= \begin{cases}
        \delta_\ell,&\text{if } \ell \in \{1,\ldots,L\},\\
        \min \left \{ \delta_L, \hat{r} \right \},&\text{if } \ell \in \{L+1,\ldots,L_k+1\}
    \end{cases}\\
        &= \begin{cases}
        \delta_\ell,&\text{if } \ell \in \{1,\ldots,L\},\\
         \delta_L,&\text{if } \ell \in \{L+1,\ldots,L_k+1\}
    \end{cases}\\
    &= \delta_\ell
    \end{align*}
    The first equality follows from the definition of $\delta_\ell'$ and from Lemma~\ref{lem:subroutine4:helpful}. The second equality follows from the fact that  $L = T_k(\bdelta)$ and from the supposition that $\bdelta$ satisfies Property~\ref{property:reverse}, which implies that $\delta_L = \cdots = \delta_{L_k+1}$. The third equality follows from the fact that $\hat{r} = \max \{ \delta_L,\ldots \} \ge \delta_L$. The fourth equality follows from the fact that  $\delta_L = \cdots = \delta_{L_k+1}$.    Our proof of Lemma~\ref{lem:listening_to_music_subroutine_4} is complete. 
    \halmos \end{proof}

    It follows from Lemma~\ref{lem:listening_to_music_subroutine_4} that if $\bdelta$ satisfies Property~\ref{property:reverse}, then $\bdelta' = \bdelta$, which implies that $\bdelta'$ satisfies $\bdelta' \in \Delta_k$ and satisfies Properties~\ref{property:notspecial}, \ref{property:forward},  \ref{property:gap}, and \ref{property:reverse}.  Therefore, assume from this point onward that $\bdelta$ does not satisfy Property~\ref{property:reverse}. 
    
    The remainder of the proof of Proposition~\ref{prop:subroutine:4} is organized as follows.  In Appendix~\ref{appx:proof:subroutine:4:delta_prime_feasible}, we show that $\bdelta' \in \Delta_k$.  In Appendix~\ref{appx:proof:subroutine:4:property_2}, we show that $\bdelta'$ satisfies Property~\ref{property:forward}.  In Appendix~\ref{appx:proof:subroutine:4:case1}, we consider the case where $\delta_L \ge \max_{\ell \in \{L+1,\ldots,L_k+1\}} r_{i_{k,\ell}}$, and for that case we prove that $\bdelta'$ satisfies Properties~\ref{property:notspecial}, \ref{property:gap}, and \ref{property:reverse} as well as  dominates $\bdelta$.  In Appendix~\ref{appx:proof:subroutine:4:case2}, we consider the case where $\delta_L < \max_{\ell \in \{L+1,\ldots,L_k+1\}} r_{i_{k,\ell}}$, and for that case we prove that $\bdelta'$ satisfies Properties~\ref{property:notspecial}, \ref{property:gap}, and \ref{property:reverse} as well as dominates $\bdelta$. 
The combination of these results completes the proof of Proposition~\ref{prop:subroutine:4}. 

\subsection[\texorpdfstring{Proof that $\bdelta'$ satisfies $\bdelta' \in \Delta_k$}{Text}]{Proof that $\bdelta'$ satisfies $\bdelta' \in \Delta_k$} \label{appx:proof:subroutine:4:delta_prime_feasible}
In the present Appendix~\ref{appx:proof:subroutine:4:delta_prime_feasible}, we show that $\bdelta' \in\Delta_k$.  Indeed, it follows from Lemma~\ref{lem:subroutine4:helpful} and from the fact that $\bdelta \in \Delta_k$ that $\delta_1' = \delta_1 \ge 0$. Moreover, we observe that
\begin{align*}
    \delta_{L_k+1}' = \min \{ \delta_{L_k+1}, \hat{r} \} \le \delta_{L_k+1} \le \bar{r}_k
\end{align*}
where the  equality is the definition of $\delta_{L_k+1}'$, the first inequality follows from algebra, and the second inequality follows from the fact that  $\bdelta \in \Delta_k$. Finally, we observe for each $\ell \in \{1,\ldots,L_k\}$ that
\begin{align*}
    \delta_\ell' &=\min \left \{ \delta_\ell, \hat{r} \right \} \le \min \left \{ \delta_{\ell+1}, \hat{r} \right \} = \delta_{\ell+1}'
\end{align*}
where the first equality is the definition of $\delta'_\ell$,   the inequality follows from the fact that $\bdelta \in \Delta_k$ (which implies $\delta_\ell \le \delta_{\ell+1})$, and the second equality is the definition of $\delta'_{\ell+1}$.   We have thus shown that $\bdelta' \in \Delta_k$.\looseness=-1

\subsection[\texorpdfstring{Proof that $\bdelta'$ satisfies Property~\ref{property:forward}}{Text}]{Proof that $\bdelta'$ satisfies Property~\ref{property:forward}} \label{appx:proof:subroutine:4:property_2}

In the present Appendix~\ref{appx:proof:subroutine:4:property_2}, we show that $\bdelta'$ satisfies Property~\ref{property:forward}. Indeed, it follows from the fact that $\bdelta$ satisfies Property~\ref{property:forward} that $\delta_1 \le r_{i_{k,1}}$. Moreover, it follows from Lemma~\ref{lem:subroutine4:helpful} that $\delta_1' = \delta_1$. Therefore, we conclude that $\delta_1' = \delta_1 \le r_{i_{k,1}}$, which shows that $\bdelta'$ satisfies Property~\ref{property:forward}.

\subsection[\texorpdfstring{Case where $\delta_L \ge \max_{\ell \in \{L+1,\ldots,L_k+1\}} r_{i_{k,\ell}}$}{Text}]{Case where $\delta_L \ge \max_{\ell \in \{L+1,\ldots,L_k+1\}} r_{i_{k,\ell}}$} \label{appx:proof:subroutine:4:case1}

In the present Appendix~\ref{appx:proof:subroutine:4:case1}, we consider the case in which $\delta_L \ge \max_{\ell \in \{L+1,\ldots,L_k+1\}} r_{i_{k,\ell}}$, and for this case we prove that $\bdelta'$ satisfies Properties~\ref{property:notspecial}, \ref{property:gap}, and \ref{property:reverse} as well as  dominates $\bdelta$. Indeed, assume throughout the remainder of Appendix~\ref{appx:proof:subroutine:4:case1} that $\delta_L \ge \max_{\ell \in \{L+1,\ldots,L_k+1\}} r_{i_{k,\ell}}$. We  will make use of the following lemma. 
\begin{lemma} \label{lem:subroutine4:ge}
$T_k(\bdelta') \ge L$, and for all $\ell \in \{1,\ldots,L_k+1\}$ we have
\begin{align*}
    \delta_\ell' = \begin{cases}
        \delta_\ell,&\text{if } \ell \in \{1,\ldots,L\},\\
        \delta_L,&\text{if } \ell \in \{L+1,\ldots,L_k+1\}
    \end{cases}
\end{align*}
\end{lemma}
\begin{proof}{Proof of Lemma~\ref{lem:subroutine4:ge}.}
  We observe that 
 \begin{align*}
     \delta'_{L} &= \delta_L \le r_{i_{k,L}}
 \end{align*}
 where the first equality follows from Lemma~\ref{lem:subroutine4:helpful} and the second inequality follows from the definition of $L$. We have thus shown that $T_k(\bdelta') \ge L$. We conclude the proof of Lemma~\ref{lem:subroutine4:ge} by showing that 
 $\delta_{\ell'}' = \delta_{L}$ for all $\ell' \in \{L+1,\ldots,L_k+1\}$. Indeed, we observe for each $\ell' \in \{L+1,\ldots,L_k+1\}$ that 
    \begin{align*}
        \delta_{\ell'}' &= \min \left \{ \delta_{\ell'}, \hat{r} \right \}\\
        &=  \min \left \{ \delta_{\ell'}, \max \left\{ \delta_L, \max_{\ell \in \{L+1,\ldots,L_k+1\}} r_{i_{k,\ell}} \right \} \right \}\\
        &=   \min \left \{ \delta_{\ell'},  \delta_L \right \}\\
     &=    \delta_L
    \end{align*}
     The first equality is the definition of $\delta'_{\ell'}$. The second equality is the definition of $\hat{r}$. The third equality follows from the assumption that $\delta_L \ge \max_{\ell \in \{L+1,\ldots,L_k+1\}} r_{i_{k,\ell}}$.  The fourth equality follows from the fact that $\bdelta \in \Delta_k$ and the fact that $\ell' \in \{L+1,\ldots,L_k+1\}$, which together imply that $\delta_{\ell'} \ge \delta_L$.   That completes our proof of Lemma~\ref{lem:subroutine4:ge}. 
\halmos \end{proof}

Equipped with Lemma~\ref{lem:subroutine4:ge}, the remainder of Appendix~\ref{appx:proof:subroutine:4:case1} is organized as follows. In Appendix~\ref{appx:proof:subroutine:4:case1:property_1}, we prove that $\bdelta'$ satisfies Property~\ref{property:notspecial}. In Appendix~\ref{appx:proof:subroutine:4:case1:property_3}, we prove that $\bdelta'$ satisfies Property~\ref{property:gap}. In Appendix~\ref{appx:proof:subroutine:4:case1:property_4}, we show that  $\bdelta'$ satisfies Property~\ref{property:reverse}. In Appendix~\ref{appx:proof:subroutine:4:case1:weakdominate}, we show that   $J_k(\bx,\bdelta') \le J_k(\bx,\bdelta)$ for all $\bx \in \mathcal{X}^c$. In Appendix~\ref{appx:proof:subroutine:4:case1:dominate}, we show that there always exists $\bar{\bx} \in \mathcal{X}^c$ that satisfies $J_k(\bar{\bx},\bdelta') < J_k(\bar{\bx},\bdelta)$. The combination of  Appendices~\ref{appx:proof:subroutine:4:case1:weakdominate} and \ref{appx:proof:subroutine:4:case1:dominate} thus implies that  $\bdelta'$ dominates $\bdelta$. 

\subsubsection[\texorpdfstring{Proof that $\bdelta'$ satisfies Property~\ref{property:notspecial}}{Text}]{Proof that $\bdelta'$ satisfies Property~\ref{property:notspecial}.} \label{appx:proof:subroutine:4:case1:property_1}
In the present Appendix~\ref{appx:proof:subroutine:4:case1:property_1}, we show that  $\bdelta'$ satisfies Property~\ref{property:notspecial}. Indeed, it follows from the fact that $\bdelta$ satisfies Property~\ref{property:notspecial} and from the definition of $L$ that $L \ge 2$. Moreover,  Lemma~\ref{lem:subroutine4:ge} shows that $T_k(\bdelta') \ge L$. We thus conclude that $T_k(\bdelta') \ge 2$, which shows that $\bdelta'$ satisfies Property~\ref{property:notspecial}. 

\subsubsection[\texorpdfstring{Proof that $\bdelta'$ satisfies Property~\ref{property:gap}}{Text}]{Proof that $\bdelta'$ satisfies Property~\ref{property:gap}.} \label{appx:proof:subroutine:4:case1:property_3} In the present Appendix~\ref{appx:proof:subroutine:4:case1:property_3}, we show that  $\bdelta'$ satisfies Property~\ref{property:gap}. 
Indeed, it follows from Lemma~\ref{lem:subroutine4:ge} that $\delta'_{L} = \cdots = \delta'_{L_k+1}$. Moreover, for each $\ell \in \{2,\ldots,L-1\}$, we observe that if $\delta'_{\ell} < r_{i_{k,\ell}}$, then
    \begin{align*}
        \delta'_{\ell} &= \delta_{\ell} = \delta_{\ell+1}  = \delta'_{\ell+1} 
    \end{align*}
    The first equality follows from Lemma~\ref{lem:subroutine4:ge} and the fact that $\ell \in \{2,\ldots,L-1\}$. The second equality follows from the fact that $\delta_\ell = \delta_\ell' < r_{i_{k,\ell}}$,  the fact that $\ell \in \{2,\ldots,L-1\}$, and the fact that $\bdelta$ satisfies Property~\ref{property:gap}, which together imply that $\delta_{\ell} = \delta_{\ell+1}$. The third equality follows from Lemma~\ref{lem:subroutine4:ge} and the fact that $\ell+1 \in \{3,\ldots,L\}$.    We have thus shown for all $\ell \in \{2,\ldots,L_k\}$  that if $\delta'_{\ell} < r_{i_{k,\ell}},$ then $\delta'_{\ell} = \delta'_{\ell+1}$, which completes our proof that $\bdelta'$ satisfies Property~\ref{property:gap}. 

\subsubsection[\texorpdfstring{Proof that $\bdelta'$ satisfies Property~\ref{property:reverse}}{Text}]{Proof that $\bdelta'$ satisfies Property~\ref{property:reverse}.} \label{appx:proof:subroutine:4:case1:property_4} In the present Appendix~\ref{appx:proof:subroutine:4:case1:property_4}, we show that  $\bdelta'$ satisfies Property~\ref{property:reverse}. Indeed,   it follows from Lemma~\ref{lem:subroutine4:ge} that $T_k(\bdelta') \ge L$ and that $\delta'_{L} = \cdots =   \delta'_{L_k+1} = \delta_L$, which imply that $\delta'_{T_k(\bdelta')} = \cdots =   \delta'_{L_k+1} = \delta_L$. This shows that $\bdelta'$ satisfies Property~\ref{property:reverse}.

\subsubsection[\texorpdfstring{Proof that $J_k(\bx,\bdelta') \le J_k(\bx,\bdelta)$ for all $\bx \in \mathcal{X}^c$}{Text}]{Proof that $J_k(\bx,\bdelta') \le J_k(\bx,\bdelta)$ for all $\bx \in \mathcal{X}^c$.} \label{appx:proof:subroutine:4:case1:weakdominate}
 In the present Appendix~\ref{appx:proof:subroutine:4:case1:weakdominate}, we show that  $J_k(\bx,\bdelta') \le J_k(\bx,\bdelta)$ for all $\bx \in \mathcal{X}^c$. Indeed, consider  any $\bx \in \mathcal{X}^c$. We first observe that 
    \begin{align*}
        J_k(\bx,\bdelta') &= \delta'_{L_k+1} +  \sum_{\ell=1}^{L_k}  \left( \max \left \{0, r_{i_{k,\ell}} - \delta_\ell' \right \} - \left( \delta_{\ell+1}' - \delta_\ell' \right) \right) x_{i_{k,\ell}} \\
        &=\delta_{L_k+1}' +  \sum_{\ell=L}^{L_k}  \left( \max \left \{0, r_{i_{k,\ell}} - \delta_\ell' \right \} - \left( \delta_{\ell+1}' - \delta_\ell' \right) \right) x_{i_{k,\ell}}\\
        &\quad +  \sum_{\ell=1}^{L-1}  \left( \max \left \{0, r_{i_{k,\ell}} - \delta_\ell' \right \} - \left( \delta_{\ell+1}' - \delta_\ell' \right) \right) x_{i_{k,\ell}}\\
     &=\delta_{L} +  \sum_{\ell=L}^{L_k}  \left( \max \left \{0, r_{i_{k,\ell}} - \delta_L \right \} - \left( \delta_{L} - \delta_L \right) \right) x_{i_{k,\ell}}\\
        &\quad +  \sum_{\ell=1}^{L-1}  \left( \max \left \{0, r_{i_{k,\ell}} - \delta_\ell \right \} - \left( \delta_{\ell+1} - \delta_\ell \right) \right) x_{i_{k,\ell}}\\
        &= \delta_{L} +  \sum_{\ell=L+1 }^{L_k}  \max \left \{0, r_{i_{k,\ell}} - \delta_L \right \}  x_{i_{k,\ell}} \\
        &\quad + \max \left \{0, r_{i_{k,L}} - \delta_L \right \}  x_{i_{k,L}}  + \sum_{\ell=1}^{L-1}  \left( \max \left \{0, r_{i_{k,\ell}} - \delta_\ell \right \} - \left( \delta_{\ell+1} - \delta_\ell \right) \right) x_{i_{k,\ell}}\\
        &= \delta_{L} +   \max \left \{0, r_{i_{k,L}} - \delta_L \right \}  x_{i_{k,L}}  + \sum_{\ell=1}^{L-1}  \left( \max \left \{0, r_{i_{k,\ell}} - \delta_\ell \right \} - \left( \delta_{\ell+1} - \delta_\ell \right) \right) x_{i_{k,\ell}}\\
&= \delta_{L} +   \left(  r_{i_{k,L}} - \delta_L \right)  x_{i_{k,L}}  + \sum_{\ell=1}^{L-1}  \left( \max \left \{0, r_{i_{k,\ell}} - \delta_\ell \right \} - \left( \delta_{\ell+1} - \delta_\ell \right) \right) x_{i_{k,\ell}}
    \end{align*}
    The first equality is the definition of $J_k(\bx,\bdelta')$. The second equality follows from algebra. The third equality follows from Lemma~\ref{lem:subroutine4:ge}. The  fourth equality follows from algebra.  The fifth equality follows from the fact  that $\delta_L \ge \max_{\ell \in \{L+1,\ldots,L_k+1\}} r_{i_{k,\ell}}$, which implies that $\max \left \{ 0, r_{i_{k,\ell}} - \delta_L \right \} =  0$ for all $\ell \in \{L+1,\ldots,L_k\}$. The sixth equality follows from the definition of $L$, which implies that $\delta_L \le r_{i_{k,L}}$. Moreover, we observe that
        \begin{align*}
        J_k(\bx,\bdelta) &= \delta_{L_k+1} +  \sum_{\ell=1}^{L_k}  \left( \max \left \{0, r_{i_{k,\ell}} - \delta_\ell \right \} - \left( \delta_{\ell+1} - \delta_\ell \right) \right) x_{i_{k,\ell}} \\
        &=\delta_{L_k+1} +  \sum_{\ell=L}^{L_k}  \left( \max \left \{0, r_{i_{k,\ell}} - \delta_\ell \right \} - \left( \delta_{\ell+1} - \delta_\ell \right) \right) x_{i_{k,\ell}}\\
        &\quad +  \sum_{\ell=1}^{L-1}  \left( \max \left \{0, r_{i_{k,\ell}} - \delta_\ell \right \} - \left( \delta_{\ell+1} - \delta_\ell \right) \right) x_{i_{k,\ell}}\\
     &=\delta_{L_k+1} - \sum_{\ell=L}^{L_k}  \left( \delta_{\ell+1} - \delta_\ell \right) x_{i_{k,\ell}} +  \max \left \{0, r_{i_{k,L}} - \delta_L \right \}  x_{i_{k,L}}\\
        &\quad +  \sum_{\ell=1}^{L-1}  \left( \max \left \{0, r_{i_{k,\ell}} - \delta_\ell \right \} - \left( \delta_{\ell+1} - \delta_\ell \right) \right) x_{i_{k,\ell}}\\
         &=\delta_{L_k+1} - \sum_{\ell=L}^{L_k}  \left( \delta_{\ell+1} - \delta_\ell \right) x_{i_{k,\ell}}+  \left( r_{i_{k,L}} - \delta_L \right)  x_{i_{k,L}}\\
        &\quad +  \sum_{\ell=1}^{L-1}  \left( \max \left \{0, r_{i_{k,\ell}} - \delta_\ell \right \} - \left( \delta_{\ell+1} - \delta_\ell \right) \right) x_{i_{k,\ell}}\\
       &\ge \delta_{L_k+1} - \sum_{\ell=L}^{L_k}  \left( \delta_{\ell+1} - \delta_\ell \right) 1+  \left( r_{i_{k,L}} - \delta_L \right)  x_{i_{k,L}}\\
        &\quad +  \sum_{\ell=1}^{L-1}  \left( \max \left \{0, r_{i_{k,\ell}} - \delta_\ell \right \} - \left( \delta_{\ell+1} - \delta_\ell \right) \right) x_{i_{k,\ell}}\\
       &= \delta_{L} +  \left( r_{i_{k,L}} - \delta_L \right)  x_{i_{k,L}}+  \sum_{\ell=1}^{L-1}  \left( \max \left \{0, r_{i_{k,\ell}} - \delta_\ell \right \} - \left( \delta_{\ell+1} - \delta_\ell \right) \right) x_{i_{k,\ell}}
    \end{align*}
    The first equality is the definition of $J_k(\bx,\bdelta)$. The second equality follows from algebra. The third equality follows from the definition of $L$, which implies that $ \max \left \{ 0, r_{i_{k,\ell}} - \delta_\ell\right \} = 0 $ for all $\ell \in \{L+1,\ldots,L_K\}$. The fourth equality follows from the definition of $L$, which implies that $\delta_L \le r_{i_{k,L}}$. The inequality follows from the fact that $\bdelta \in \Delta_k$ (which implies that $\delta_{\ell+1} - \delta_{\ell} \ge 0$ for all $\ell \in \{L,\ldots,L_k\}$) and the fact that $\bx \in \mathcal{X}^c$ (which implies that $x_{i_{k,\ell}} \le 1$ for all $\ell \in \{L,\ldots,L_k\}$). The fifth equality follows from algebra. Combining the above reasoning, we have shown that 
    \begin{align*}
        J_k(\bx,\bdelta') &=  \delta_{L} +   \left(  r_{i_{k,L}} - \delta_L \right)  x_{i_{k,L}}  + \sum_{\ell=1}^{L-1}  \left( \max \left \{0, r_{i_{k,\ell}} - \delta_\ell \right \} - \left( \delta_{\ell+1} - \delta_\ell \right) \right) x_{i_{k,\ell}}\\
        &\le J_k(\bx,\bdelta)
    \end{align*}
We thus conclude that $J_k(\bx,\bdelta') \le J_k(\bx,\bdelta)$.

\subsubsection[\texorpdfstring{Proof that there exists $\bar{\bx} \in \mathcal{X}^c$ that satisfies $J_k(\bar{\bx},\bdelta') < J_k(\bar{\bx},\bdelta)$}{Text}]{Proof that there exists $\bar{\bx} \in \mathcal{X}^c$ that satisfies $J_k(\bar{\bx},\bdelta') < J_k(\bar{\bx},\bdelta)$.} \label{appx:proof:subroutine:4:case1:dominate}
In the present Appendix~\ref{appx:proof:subroutine:4:case1:dominate}, we show that  there exists $\bar{\bx} \in \mathcal{X}^c$ that satisfies $J_k(\bar{\bx},\bdelta') < J_k(\bar{\bx},\bdelta)$. 
We will make use of the following lemma.
\begin{lemma} \label{lem:subroutine:4:case1:delta_L_k_so_so_big}
    $\delta'_{L_k+1} < \delta_{L_k+1}$.
\end{lemma}
\begin{proof}{Proof of Lemma~\ref{lem:subroutine:4:case1:delta_L_k_so_so_big}.}
It follows from the construction of $\bdelta'$ that $\delta'_\ell \le \delta_\ell$ for all $\ell \in \{1,\ldots,L_k+1\}$. Therefore, it follows from the assumption that $\bdelta \neq \bdelta'$ and from the fact that $\delta'_\ell = \delta_\ell$ for all $\ell \in \{1,\ldots,L\}$ (Lemma~\ref{lem:subroutine4:ge}) that there must exist $\ell^* \in \{L+1,\ldots,L_k+1\}$ such that $\delta_{\ell^*}' < \delta_{\ell^*}$. We thus observe that
\begin{align*}
    \delta_{L_k+1}' = \delta_{\ell^*}' < \delta_{\ell^*} \le \delta_{L_k+1}
\end{align*}
The first equality follows from Lemma~\ref{lem:subroutine4:ge} (which implies that $\delta_{L+1}' = \cdots = \delta_{L_k+1}'$). The inequality follows from the construction of $\ell^*$. The second equality follows from the fact that $\bdelta \in \Delta_k$ (which implies that $\delta_{\ell^*} \le \cdots \le \delta_{L_k+1}$). We have thus shown that $\delta_{L_k+1}' < \delta_{L_k+1}$, which completes the proof of Lemma~\ref{lem:subroutine:4:case1:delta_L_k_so_so_big}. 
\halmos \end{proof}

Equipped with the above lemma, we now show that  there exists $\bar{\bx} \in \mathcal{X}^c$ that satisfies $J_k(\bar{\bx},\bdelta') < J_k(\bar{\bx},\bdelta)$. Indeed, let $\bar{\bx}$ be the vector that is defined for each $\ell \in [N+1]$ as
\begin{align*}
    \bar{x}_{i_{k,\ell}} &= \begin{cases}
        1,&\text{if } \ell = L_k+1,\\
        0,&\text{otherwise}
    \end{cases}
\end{align*}
 We observe that
\begin{align*}
    J_k(\bar{\bx}, \bdelta') &= \delta'_{L_k+1} + \sum_{\ell =1}^{L_k}   \left( \max \left \{0, r_{i_{k,\ell}} - \delta_\ell' \right \} - \left( \delta_{\ell+1}' - \delta_\ell' \right) \right) 0\\
    &<   \delta_{L_k+1} + \sum_{\ell =1}^{L_k}   \left( \max \left \{0, r_{i_{k,\ell}} - \delta_\ell \right \} - \left( \delta_{\ell+1} - \delta_\ell \right) \right) 0\\
    &= J_k(\bar{\bx}, \bdelta)
\end{align*}
The first equality follows from the construction of $\bar{\bx}$. The inequality follows from algebra and from Lemma~\ref{lem:subroutine:4:case1:delta_L_k_so_so_big}. The second equality follows from the construction of $\bar{\bx}$. This completes our proof that there always exists $\bar{\bx} \in \mathcal{X}^c$ that satisfies $J_k(\bar{\bx}, \bdelta') < J_k(\bar{\bx}, \bdelta)$.

\subsection[\texorpdfstring{Case where $\delta_L < \max_{\ell \in \{L+1,\ldots,L_k+1\}} r_{i_{k,\ell}}$}{Text}]{Case where $\delta_L < \max_{\ell \in \{L+1,\ldots,L_k+1\}} r_{i_{k,\ell}}$} \label{appx:proof:subroutine:4:case2}

In the present Appendix~\ref{appx:proof:subroutine:4:case2}, we consider the case in which $\delta_L < \max_{\ell \in \{L+1,\ldots,L_k+1\}} r_{i_{k,\ell}}$, and for this case we prove that $\bdelta'$ satisfies Properties~\ref{property:notspecial}, \ref{property:gap}, and \ref{property:reverse} as well as dominates $\bdelta$. Indeed, assume throughout the remainder of Appendix~\ref{appx:proof:subroutine:4:case2} that $\delta_L < \max_{\ell \in \{L+1,\ldots,L_k+1\}} r_{i_{k,\ell}}$. Moreover, let $\bar{L}$ denote the largest integer that satisfies $\delta_{\bar{L}} \le \max_{\ell \in \{L+1,\ldots,L_k+1\}} r_{i_{k,\ell}}$.  We will make use of the following lemmas.  

\begin{lemma} \label{lem:subroutine4:strict_less:L_bar_in_between}
    $L \le \bar{L} <  T_k(\bdelta')$.
\end{lemma}
\begin{proof}{Proof of Lemma~\ref{lem:subroutine4:strict_less:L_bar_in_between}.}
If follows from the fact that $\delta_L  < \max_{\ell \in \{L+1,\ldots,L_k+1\} }r_{i_{k,\ell}}$ and from the definition of $\bar{L}$ that $\bar{L}$ must be greater than or equal to ${L}$.

To show that $\bar{L} < T_k(\bdelta')$, let $\hat{L} \in \{L+1,\ldots,L_k+1\}$ denote any integer that satisfies $r_{i_{k,\hat{L}}} = \max_{\ell \in \{L+1,\ldots,L_k+1\}} r_{i_{k,\ell}}$. It follows from algebra that such $\hat{L}$ must exist.  We will now show by contradiction that $\bar{L} < \hat{L}$.  Indeed, suppose for the sake of developing a contradiction that $\hat{L} \le \bar{L}$. If that were true, then 
\begin{align*}
r_{i_{k,\hat{L}}} < \delta_{\hat{L}} \le \delta_{\bar{L}} &\le  \max_{\ell \in \{L+1,\ldots,L_k+1\}} r_{i_{k,\ell}} = r_{i_{k,\hat{L}}}
\end{align*}
The first inequality follows from the definition of $L$ and from the fact that $\hat{L} > L$. The second inequality follows from the supposition that $\hat{L} \le \bar{L}$ and the fact that $\bdelta \in \Delta_k$. The third inequality follows from the definition of $\bar{L}$.  The equality follows from the definition of $\hat{L}$. 
Since we cannot have $r_{i_{k,\hat{L}}} < r_{i_{k,\hat{L}}}$, it must be the case that $\bar{L} < \hat{L}$. 

To complete our proof that $\bar{L} < T_k(\bdelta')$, we now show that $T_k(\bdelta') \ge \hat{L}$. Indeed, we observe that 
\begin{align*}
    \delta'_{\hat{L}} &=  \min \left \{ \delta_{\hat{L}}, \hat{r} \right \} \\
     &= \min \left \{  \delta_{\hat{L}}, \max \left \{ \delta_L, \max_{\ell \in \{L+1,\ldots,L_k+1\}} r_{i_{k,\ell}}\right \} \right \}\\
        &=   \min \left \{  \delta_{\hat{L}}, \max_{\ell \in \{L+1,\ldots,L_k+1\}} r_{i_{k,\ell}}\right \} \\
        &=   \min \left \{  \delta_{\hat{L}}, r_{i_{k,\hat{L}}}\right \} \\
        &= r_{i_{k,\hat{L}}}
        \end{align*}
The first equality is the definition of $\delta'_{\hat{L}}$. The second equality is the definition of $\hat{r}$. The third equality follows from the assumption that $\delta_L < \max_{\ell \in \{L+1,\ldots,L_k+1\}} r_{i_{k,\ell}}$. The fourth equality follows from the definition of $\hat{L}$. 
The fifth equality follows from the definition of $L$ and from the fact that $\hat{L} > \bar{L} \ge L$, which together imply that $\delta_{\hat{L}} > r_{i_{k,\hat{L}}}$.  We have thus shown that $\delta'_{\hat{L}} \le r_{i_{k,\hat{L}}}$, which implies $T_k(\bdelta') \ge \hat{L}$. This concludes our proof that $T_k(\bdelta') > \bar{L}$, which completes our proof of Lemma~\ref{lem:subroutine4:strict_less:L_bar_in_between}. 
\halmos \end{proof}
\begin{lemma} \label{lem:subroutine4:strict_less:delta_prime_ell}
    For all $\ell \in \{1,\ldots,L_k+1\}$ we have
    \begin{align*}
        \delta_\ell' &= \begin{cases}
            \delta_\ell,&\text{if } \ell \in \left\{1,\ldots,\bar{L} \right \},\\
            \max_{\ell' \in \{L+1,\ldots,L_k+1\}} r_{i_{k,\ell'}},&\text{if } \ell \in \{\bar{L}+1,\ldots,L_k+1 \}
        \end{cases}
    \end{align*}
\end{lemma}
\begin{proof}{Proof of Lemma~\ref{lem:subroutine4:strict_less:delta_prime_ell}.} 
    We observe for all $\ell \in \{1,\ldots,L_k+1\}$ 
    that
    \begin{align*}
        \delta'_{\ell} &= \min \left \{ \delta_{\ell}, \hat{r} \right \} \\
        &= \min \left \{ \delta_{\ell}, \max \left \{ \delta_L, \max_{\ell' \in \{L+1,\ldots,L_k+1\}} r_{i_{k,\ell'}}\right \}\right \}\\
        &= \min \left \{ \delta_{\ell}, \max_{\ell' \in \{L+1,\ldots,L_k+1\}} r_{i_{k,\ell'}}\right \}\\
            &= \begin{cases}
            \delta_\ell,&\text{if } \ell \in \left\{1,\ldots,\bar{L} \right \},\\
            \max_{\ell' \in \{L+1,\ldots,L_k+1\}} r_{i_{k,\ell'}},&\text{if } \ell \in \{\bar{L}+1,\ldots,L_k+1 \}
        \end{cases}
    \end{align*}
 The first equality is the definition of $\delta'_{\ell}$. The second equality follows from the definition of $\hat{r}$. The third equality follows from our assumption that $\delta_L < \max_{\ell \in \{L+1,\ldots,L_k+1\}} r_{i_{k,\ell}}$. The fourth equality follows from the definition of $\bar{L}$ and the fact that $\bdelta \in \Delta_k$, which imply that $\delta_{\ell} \le \max_{\ell' \in \{L+1,\ldots,L_k+1\}} r_{i_{k,\ell'}}$ for all $\ell \in \{1,\ldots, \bar{L}\}$ and $\delta_{\ell} > \max_{\ell' \in \{L+1,\ldots,L_k+1\}} r_{i_{k,\ell'}}$ for all $\ell \in \{\bar{L}+1,\ldots,L_k+1\}$. This completes our proof of Lemma~\ref{lem:subroutine4:strict_less:delta_prime_ell}. 
\halmos \end{proof}

Equipped with the above intermediary lemmas, the remainder of Appendix~\ref{appx:proof:subroutine:4:case2} is organized as follows. In Appendix~\ref{appx:proof:subroutine:4:case2:property_1}, we prove that $\bdelta'$ satisfies Property~\ref{property:notspecial}. In Appendix~\ref{appx:proof:subroutine:4:case2:property_3}, we prove that $\bdelta'$ satisfies Property~\ref{property:gap}. In Appendix~\ref{appx:proof:subroutine:4:case2:property_4}, we show that  $\bdelta'$ satisfies Property~\ref{property:reverse}. In Appendix~\ref{appx:proof:subroutine:4:case2:weakdominate}, we show that   $J_k(\bx,\bdelta') \le J_k(\bx,\bdelta)$ for all $\bx \in \mathcal{X}^c$. In Appendix~\ref{appx:proof:subroutine:4:case2:dominate}, we show that there always exists $\bar{\bx} \in \mathcal{X}^c$ that satisfies $J_k(\bar{\bx},\bdelta') < J_k(\bar{\bx},\bdelta)$. The combination of  Appendices~\ref{appx:proof:subroutine:4:case2:weakdominate} and \ref{appx:proof:subroutine:4:case2:dominate} thus implies that  $\bdelta'$ dominates $\bdelta$.

\subsubsection[\texorpdfstring{Proof that $\bdelta'$ satisfies Property~\ref{property:notspecial}.}{Text}]{Proof that $\bdelta'$ satisfies Property~\ref{property:notspecial}.} \label{appx:proof:subroutine:4:case2:property_1}
In the present Appendix~\ref{appx:proof:subroutine:4:case2:property_1}, we show that  $\bdelta'$ satisfies Property~\ref{property:notspecial}. Indeed, it follows from the fact that $\bdelta$ satisfies Property~\ref{property:notspecial} and from the definition of $L$ that $L \ge 2$. Moreover, it follows from Lemma~\ref{lem:subroutine4:strict_less:L_bar_in_between} that  $T_k(\bdelta') > L$.  We thus conclude that $T_k(\bdelta') > L \ge 2$, which shows that $\bdelta'$ satisfies Property~\ref{property:notspecial}.

\subsubsection[\texorpdfstring{Proof that $\bdelta'$ satisfies Property~\ref{property:gap}}{Text}]{Proof that $\bdelta'$ satisfies Property~\ref{property:gap}.} \label{appx:proof:subroutine:4:case2:property_3}
 In the present Appendix~\ref{appx:proof:subroutine:4:case2:property_3}, we show that  $\bdelta'$ satisfies Property~\ref{property:gap}.
Indeed, it follows from Lemma~\ref{lem:subroutine4:strict_less:delta_prime_ell} that $\delta'_{\bar{L}+1} = \cdots = \delta'_{L_k+1}$. Moreover, for each $\ell \in \{2,\ldots,\bar{L}-1\}$, we observe that if $\delta'_{\ell} < r_{i_{k,\ell}}$, then
    \begin{align*}
        \delta'_{\ell} &= \delta_{\ell} = \delta_{\ell+1}  = \delta'_{\ell+1} 
    \end{align*}
    The first equality follows from Lemma~\ref{lem:subroutine4:strict_less:delta_prime_ell} and the fact that $\ell \in \{2,\ldots,\bar{L}-1\}$. The second equality follows from the fact that $\delta_\ell = \delta_\ell' < r_{i_{k,\ell}}$ and the fact that $\bdelta$ satisfies Property~\ref{property:gap}, which together imply that $\delta_{\ell} = \delta_{\ell+1}$. The third equality follows from Lemma~\ref{lem:subroutine4:strict_less:delta_prime_ell} and the fact that $\ell+1 \in \{3,\ldots,\bar{L}\}$. 
  Finally, we observe that %$\delta_{\bar{L}}' =  \delta_{\bar{L}}  \ge r_{i_{k,\bar{L}}}$. 
       \begin{align*}
        \delta'_{\bar{L}} =  \delta_{\bar{L}} \ge r_{i_{k,\bar{L}}}
    \end{align*}
     The  equality follows from Lemma~\ref{lem:subroutine4:strict_less:delta_prime_ell}.  The inequality follows from the definition of $\bar{L}$ (which implies that $\delta_{\bar{L}} < \delta_{\bar{L}+1}$ and $\bar{L} \ge 2$) and the fact that  $\bdelta$ satisfies Property~\ref{property:gap}, which imply that $\delta_{\bar{L}} \ge r_{i_{k,\bar{L}}}$. We have thus shown  for each $\ell \in \{2,\ldots,L_k\}$ that if $\delta_\ell' > r_{i_{k,\ell}}$, then $\delta_{\ell}' = \delta_{\ell+1}'$, which shows that $\bdelta'$ satisfies Property~\ref{property:gap}.
\subsubsection[\texorpdfstring{Proof that $\bdelta'$ satisfies Property~\ref{property:reverse}}{Text}]{Proof that $\bdelta'$ satisfies Property~\ref{property:reverse}.} \label{appx:proof:subroutine:4:case2:property_4} In the present Appendix~\ref{appx:proof:subroutine:4:case2:property_4}, we show that  $\bdelta'$ satisfies Property~\ref{property:reverse}. Indeed, Lemma~\ref{lem:subroutine4:strict_less:L_bar_in_between} shows that $T_k(\bdelta') > \bar{L}$. Moreover, it follows from Lemma~\ref{lem:subroutine4:strict_less:delta_prime_ell} that $\delta'_{\bar{L}+1} = \cdots  = \delta_{L_k+1}'$. Therefore, we conclude that  $\delta'_{T_k(\bdelta')} = \cdots = \delta'_{L_k+1}$, which shows that $\bdelta'$ satisfies Property~\ref{property:reverse}.

\subsubsection[\texorpdfstring{Proof that $J_k(\bx,\bdelta') \le J_k(\bx,\bdelta)$ for all $\bx \in \mathcal{X}^c$}{Text}]{Proof that $J_k(\bx,\bdelta') \le J_k(\bx,\bdelta)$ for all $\bx \in \mathcal{X}^c$.} \label{appx:proof:subroutine:4:case2:weakdominate}
 In the present Appendix~\ref{appx:proof:subroutine:4:case2:weakdominate}, we show that  $J_k(\bx,\bdelta') \le J_k(\bx,\bdelta)$ for all $\bx \in \mathcal{X}^c$. Indeed, consider  any $\bx \in \mathcal{X}^c$. We first show that
\begin{align}
    J_k(\bx,\bdelta')  &=  \max_{\ell' \in \{L+1,\ldots,L_k+1\}} r_{i_{k,\ell'}}  + \left( \delta_{\bar{L}} - \max_{\ell' \in \{L+1,\ldots,L_k+1\}} r_{i_{k,\ell'}}  \right)  x_{i_{k,\bar{L}}}\notag \\
    &\quad + \sum_{\ell=1}^{\bar{L}-1}  \left( \max \left \{0, r_{i_{k,\ell}} - \delta_\ell \right \} - \left( \delta_{\ell+1} - \delta_\ell \right) \right) x_{i_{k,\ell}}  \label{line:subroutine4:strict_less:J_k_transform}
\end{align}
Indeed, we first observe that
\begin{align}
    &\delta_{L_k+1}' +  \sum_{\ell=\bar{L}+1}^{L_k}  \left( \max \left \{0, r_{i_{k,\ell}} - \delta_\ell' \right \} - \left( \delta_{\ell+1}' - \delta_\ell' \right) \right) x_{i_{k,\ell}}\notag \\
    &=\max_{\ell' \in \{L+1,\ldots,L_k+1\}} r_{i_{k,\ell'}} \notag \\
    &\quad +  \sum_{\ell=\bar{L}+1}^{L_k}  \left( \max \left \{0, r_{i_{k,\ell}} - \max_{\ell' \in \{L+1,\ldots,L_k+1\}} r_{i_{k,\ell'}} \right \}  - \left( \max_{\ell' \in \{L+1,\ldots,L_k+1\}} r_{i_{k,\ell'}} - \max_{\ell' \in \{L+1,\ldots,L_k+1\}} r_{i_{k,\ell'}} \right) \right) x_{i_{k,\ell}} \notag \\
        &= \max_{\ell' \in \{L+1,\ldots,L_k+1\}} r_{i_{k,\ell'}}  \label{line:subroutine4:strict_less:top_third} 
\end{align}
where the first equality follows from Lemma~\ref{lem:subroutine4:strict_less:delta_prime_ell} and the second   equality follows from algebra. We next observe that
\begin{align}
    & \sum_{\ell=1}^{\bar{L}-1}  \left( \max \left \{0, r_{i_{k,\ell}} - \delta_\ell' \right \} - \left( \delta_{\ell+1}' - \delta_\ell' \right) \right) x_{i_{k,\ell}}= \sum_{\ell=1}^{\bar{L}-1}  \left( \max \left \{0, r_{i_{k,\ell}} - \delta_\ell \right \} - \left( \delta_{\ell+1} - \delta_\ell \right) \right) x_{i_{k,\ell}} \label{line:subroutine4:strict_less:bottom_third} 
\end{align}
where the equality follows from Lemma~\ref{lem:subroutine4:strict_less:delta_prime_ell}. Finally, we observe that 
\begin{align}
    &  \left( \max \left \{0, r_{i_{k,\bar{L}}} - \delta_{\bar{L}}' \right \} - \left( \delta_{\bar{L}+1}' - \delta_{\bar{L}}' \right) \right) x_{i_{k,\bar{L}}}\notag \\
    &=  \left( \max \left \{0, r_{i_{k,\bar{L}}} - \delta_{\bar{L}} \right \} - \left( \max_{\ell' \in \{L+1,\ldots,L_k+1\}} r_{i_{k,\ell'}}  - \delta_{\bar{L}} \right) \right) x_{i_{k,\bar{L}}}\notag \\
&=  \left( 0 - \left( \max_{\ell' \in \{L+1,\ldots,L_k+1\}} r_{i_{k,\ell'}}  - \delta_{\bar{L}} \right) \right) x_{i_{k,\bar{L}}} \notag \\
&= \left( \delta_{\bar{L}} - \max_{\ell' \in \{L+1,\ldots,L_k+1\}} r_{i_{k,\ell'}}  \right)  x_{i_{k,\bar{L}}} \label{line:subroutine4:strict_less:middle_third} 
\end{align}
The first equality follows from Lemma~\ref{lem:subroutine4:strict_less:delta_prime_ell}. To see why the second equality holds, we  observe that Lemma~\ref{lem:subroutine4:strict_less:L_bar_in_between} implies that $\bar{L}+1 \le L_k+1$, and so it follows from  the definition of $\bar{L}$ that  $
\delta_{\bar{L}} < \delta_{\bar{L}+1}$. Therefore, it follows from  the fact that $\bdelta$ satisfies Property~\ref{property:gap} that $\delta_{\bar{L}} \ge r_{i_{k,\bar{L}}}$, which implies that $0 \ge r_{i_{k,\bar{L}}} - \delta_{\bar{L}} $. The third equality follows from algebra.   Combining \eqref{line:subroutine4:strict_less:top_third}, \eqref{line:subroutine4:strict_less:bottom_third}, and \eqref{line:subroutine4:strict_less:middle_third}, we have shown that 
\begin{align*}
    J_k(\bx,\bdelta') 
        &=\delta_{L_k+1}' +  \sum_{\ell=\bar{L}+1}^{L_k}  \left( \max \left \{0, r_{i_{k,\ell}} - \delta_\ell' \right \} - \left( \delta_{\ell+1}' - \delta_\ell' \right) \right) x_{i_{k,\ell}}\\
        &\quad +  \left( \max \left \{0, r_{i_{k,\bar{L}}} - \delta_{\bar{L}}' \right \} - \left( \delta_{\bar{L}+1}' - \delta_{\bar{L}}' \right) \right) x_{i_{k,\bar{L}}}\\
        &\quad +  \sum_{\ell=1}^{\bar{L}-1}  \left( \max \left \{0, r_{i_{k,\ell}} - \delta_\ell' \right \} - \left( \delta_{\ell+1}' - \delta_\ell' \right) \right) x_{i_{k,\ell}}\\
  &= \max_{\ell' \in \{L+1,\ldots,L_k+1\}} r_{i_{k,\ell'}}  + \left( \delta_{\bar{L}} - \max_{\ell' \in \{L+1,\ldots,L_k+1\}} r_{i_{k,\ell'}}  \right)  x_{i_{k,\bar{L}}} \\
    &\quad + \sum_{\ell=1}^{\bar{L}-1}  \left( \max \left \{0, r_{i_{k,\ell}} - \delta_\ell \right \} - \left( \delta_{\ell+1} - \delta_\ell \right) \right) x_{i_{k,\ell}}  
\end{align*}
where the first equality is the definition of $J_k(\bx,\bdelta')$ and the second equality follows from \eqref{line:subroutine4:strict_less:top_third}, \eqref{line:subroutine4:strict_less:bottom_third}, and \eqref{line:subroutine4:strict_less:middle_third}. We have thus shown that \eqref{line:subroutine4:strict_less:J_k_transform} holds.

We next show that
\begin{align}
           J_k(\bx,\bdelta)  &\ge 
           \max_{\ell' \in \{L+1,\ldots,L_k+1\}} r_{i_{k,\ell'}}  + \left( \delta_{\bar{L}} - \max_{\ell' \in \{L+1,\ldots,L_k+1\}} r_{i_{k,\ell'}}  \right)  x_{i_{k,\bar{L}}} \notag \\
        &\quad +  \sum_{\ell=1}^{\bar{L}-1}  \left( \max \left \{0, r_{i_{k,\ell}} - \delta_\ell \right \} - \left( \delta_{\ell+1} - \delta_\ell \right) \right) x_{i_{k,\ell}}\label{line:subroutine4:strict_less:J_k_transform:2}
\end{align}
Indeed, we first observe that
 \begin{align}
    &\delta_{L_k+1} +  \sum_{\ell=\bar{L}+1}^{L_k}  \left( \max \left \{0, r_{i_{k,\ell}} - \delta_\ell \right \} - \left( \delta_{\ell+1} - \delta_\ell \right) \right) x_{i_{k,\ell}}\notag \\
        &=\delta_{L_k+1} +  \sum_{\ell=\bar{L}+1}^{L_k}  \left( 0- \left( \delta_{\ell+1} - \delta_\ell \right) \right) x_{i_{k,\ell}}\notag \\
&\ge \delta_{L_k+1} -  \sum_{\ell=\bar{L}+1}^{L_k}  \left( \delta_{\ell+1} - \delta_\ell \right) 1\notag \\     
&= \delta_{\bar{L}+1}\label{line:subroutine4:strict_less:delta_orig:top_third} 
\end{align}
The first equality follows from Lemma~\ref{lem:subroutine4:strict_less:L_bar_in_between} (which implies that $\bar{L} + 1 > L$) and  the definition of $L$ (which implies that $\delta_\ell > r_{i_{k,\ell}}$ for all $\ell > L$).  The inequality follows from the fact that $\bdelta \in \Delta_k$ (which implies that $\delta_{\ell+1} \ge \delta_\ell$ for all $\ell \in \{\bar{L}+1,\ldots,L_k\}$) and the fact that $\bx \in \mathcal{X}^c$ (which implies that $x_{i_{k,\ell}} \le 1$ for all $\ell \in \{\bar{L}+1,\ldots,L_k \}$). The second equality follows from algebra. Moreover, we observe that 
\begin{align}
          \left( \max \left \{0, r_{i_{k,\bar{L}}} - \delta_{\bar{L}} \right \} - \left( \delta_{\bar{L}+1} - \delta_{\bar{L}} \right) \right) x_{i_{k,\bar{L}}}  & = \left( 0 - \left( \delta_{\bar{L}+1} - \delta_{\bar{L}} \right) \right) x_{i_{k,\bar{L}}} \notag\\
         &= \left( \delta_{\bar{L}}  - \delta_{\bar{L}+1} \right) x_{i_{k,\bar{L}}} 
\label{line:subroutine4:strict_less:delta_orig:middle_third}
\end{align}
To see why the first equality holds, we  observe that Lemma~\ref{lem:subroutine4:strict_less:L_bar_in_between} implies that $\bar{L}+1 \le L_k+1$, and so it follows from  the definition of $\bar{L}$ that  $
\delta_{\bar{L}} < \delta_{\bar{L}+1}$. Therefore, it follows from  the fact that $\bdelta$ satisfies Property~\ref{property:gap} that $\delta_{\bar{L}} \ge r_{i_{k,\bar{L}}}$, which implies that $0 \ge r_{i_{k,\bar{L}}} - \delta_{\bar{L}} $. 
 The second equality follows from algebra.   Combining \eqref{line:subroutine4:strict_less:delta_orig:top_third} and \eqref{line:subroutine4:strict_less:delta_orig:middle_third}, we have shown that 
\begin{align*}
            J_k(\bx,\bdelta) &= 
            \delta_{L_k+1} +  \sum_{\ell=\bar{L}+1}^{L_k}  \left( \max \left \{0, r_{i_{k,\ell}} - \delta_\ell \right \} - \left( \delta_{\ell+1} - \delta_\ell \right) \right) x_{i_{k,\ell}}\\
        &\quad +  \left( \max \left \{0, r_{i_{k,\bar{L}}} - \delta_{\bar{L}} \right \} - \left( \delta_{\bar{L}+1} - \delta_{\bar{L}} \right) \right) x_{i_{k,\bar{L}}}\\
        &\quad +  \sum_{\ell=1}^{\bar{L}-1}  \left( \max \left \{0, r_{i_{k,\ell}} - \delta_\ell \right \} - \left( \delta_{\ell+1} - \delta_\ell \right) \right) x_{i_{k,\ell}}\\
         &= 
            \delta_{\bar{L}+1} +  \left( \delta_{\bar{L}}  - \delta_{\bar{L}+1} \right) x_{i_{k,\bar{L}}} \\
        &\quad +  \sum_{\ell=1}^{\bar{L}-1}  \left( \max \left \{0, r_{i_{k,\ell}} - \delta_\ell \right \} - \left( \delta_{\ell+1} - \delta_\ell \right) \right) x_{i_{k,\ell}}\\
               &\ge 
           \max_{\ell' \in \{L+1,\ldots,L_k+1\}} r_{i_{k,\ell'}}  + \left( \delta_{\bar{L}} - \max_{\ell' \in \{L+1,\ldots,L_k+1\}} r_{i_{k,\ell'}}  \right)  x_{i_{k,\bar{L}}}\\
        &\quad +  \sum_{\ell=1}^{\bar{L}-1}  \left( \max \left \{0, r_{i_{k,\ell}} - \delta_\ell \right \} - \left( \delta_{\ell+1} - \delta_\ell \right) \right) x_{i_{k,\ell}}
\end{align*}
where the first equality is the definition of $ J_k(\bx,\bdelta)$, the second equality follows from \eqref{line:subroutine4:strict_less:delta_orig:top_third} and \eqref{line:subroutine4:strict_less:delta_orig:middle_third}, and the inequality follows from the definition of $\bar{L}$ (which implies that $\delta_{\bar{L}+1} > \max_{\ell' \in \{L+1,\ldots,L_k+1\}} r_{i_{k,\ell'}}$) and the fact that $\bx \in \mathcal{X}^c$ (which implies that $x_{i_{k,\bar{L}}} \le 1$). We have thus shown that \eqref{line:subroutine4:strict_less:J_k_transform:2} holds.

     Combining \eqref{line:subroutine4:strict_less:J_k_transform} and \eqref{line:subroutine4:strict_less:J_k_transform:2}, we conclude that $J_k(\bx,\bdelta') \le J_k(\bx,\bdelta)$

\subsubsection[\texorpdfstring{Proof that there exists $\bar{\bx} \in \mathcal{X}^c$ that satisfies $J_k(\bar{\bx},\bdelta') < J_k(\bar{\bx},\bdelta)$}{Text}]{Proof that there exists $\bar{\bx} \in \mathcal{X}^c$ that satisfies $J_k(\bar{\bx},\bdelta') < J_k(\bar{\bx},\bdelta)$.} \label{appx:proof:subroutine:4:case2:dominate}
In the present Appendix~\ref{appx:proof:subroutine:4:case2:dominate}, we show that  there exists $\bar{\bx} \in \mathcal{X}^c$ that satisfies $J_k(\bar{\bx},\bdelta') < J_k(\bar{\bx},\bdelta)$.

We will make use of the following lemma.
\begin{lemma} \label{lem:subroutine:4:case2:delta_L_k_so_so_big}
    $\delta'_{L_k+1} < \delta_{L_k+1}$.
\end{lemma}
\begin{proof}{Proof of Lemma~\ref{lem:subroutine:4:case2:delta_L_k_so_so_big}.}
It follows from the construction of $\bdelta'$ that $\delta'_\ell \le \delta_\ell$ for all $\ell \in \{1,\ldots,L_k+1\}$. Therefore, it follows from the assumption that $\bdelta \neq \bdelta'$ and from the fact that $\delta'_\ell = \delta_\ell$ for all $\ell \in \{1,\ldots,\bar{L}\}$ (Lemma~\ref{lem:subroutine4:strict_less:delta_prime_ell}) that there must exist $\ell^* \in \{\bar{L}+1,\ldots,L_k+1\}$ such that $\delta_{\ell^*}' < \delta_{\ell^*}$. We thus observe that
\begin{align*}
    \delta_{L_k+1}' = \delta_{\ell^*}' < \delta_{\ell^*} \le \delta_{L_k+1}
\end{align*}
The  equality follows from Lemma~\ref{lem:subroutine4:strict_less:delta_prime_ell} (which implies that $\delta_{\bar{L}+1}' = \cdots = \delta_{L_k+1}'$). The first inequality follows from the construction of $\ell^*$. The second inequality follows from the fact that $\bdelta \in \Delta_k$ (which implies that $\delta_{\ell^*} \le \cdots \le \delta_{L_k+1}$). We have thus shown that $\delta_{L_k+1}' < \delta_{L_k+1}$, which completes the proof of Lemma~\ref{lem:subroutine:4:case2:delta_L_k_so_so_big}. 
\halmos \end{proof}

Equipped with the above lemma, we now show that  there exists $\bar{\bx} \in \mathcal{X}^c$ that satisfies $J_k(\bar{\bx},\bdelta') < J_k(\bar{\bx},\bdelta)$. Indeed, let $\bar{\bx}$ be the vector that is defined for each $\ell \in [N+1]$ as
\begin{align*}
    \bar{x}_{i_{k,\ell}} &= \begin{cases}
        1,&\text{if } \ell = L_k+1,\\
        0,&\text{otherwise}
    \end{cases}
\end{align*}
 We observe that
\begin{align*}
    J_k(\bar{\bx}, \bdelta') &= \delta'_{L_k+1} + \sum_{\ell =1}^{L_k}   \left( \max \left \{0, r_{i_{k,\ell}} - \delta_\ell' \right \} - \left( \delta_{\ell+1}' - \delta_\ell' \right) \right) 0\\
    &<   \delta_{L_k+1} + \sum_{\ell =1}^{L_k}   \left( \max \left \{0, r_{i_{k,\ell}} - \delta_\ell \right \} - \left( \delta_{\ell+1} - \delta_\ell \right) \right) 0\\
    &= J_k(\bar{\bx}, \bdelta)
\end{align*}
The first equality follows from the construction of $\bar{\bx}$. The inequality follows from algebra and from Lemma~\ref{lem:subroutine:4:case2:delta_L_k_so_so_big}. The second equality follows from the construction of $\bar{\bx}$. This completes our proof that there always exists $\bar{\bx} \in \mathcal{X}^c$ that satisfies $J_k(\bar{\bx}, \bdelta') < J_k(\bar{\bx}, \bdelta)$.

\section{Implementation Details and  Additional Experimental Results}\label{appx:extra_experiments}

In this appendix, we complement the numerical experiments from Section~\ref{sec:experiments} by discussing implementation details and presenting additional  experiments. This appendix is organized as follows.
\begin{itemize}
    \item In Appendix~\ref{appx:extra_experiments:implementation}, we discuss in detail our implementations of the various solution methods from Section~\ref{sec:experiments}. 
    \vspace{0.5em} 
    
    \item In Appendix~\ref{appx:M_spread}, we present additional numerical results regarding the synthetic data generation process that is used in Section~\ref{sec:experiments:MNL_Cutoff}. 
      \vspace{0.5em} 
    
    \item In Appendix~\ref{appx:exclusion_set_budget}, we conduct additional experiments to those in Section~\ref{sec:experiments:mnl:xset} to investigate  the practical efficiency of the exclusion set formulation in assortment optimization problems with cardinality constraints.\looseness=-1
      \vspace{0.5em} 
    
    \item In Appendix~\ref{appx:cross_fold}, we conduct additional experiments to show the out-of-sample performance of the sample average approximation in the context of the real-world dataset from Section~\ref{sec:results:toubia}. 
\end{itemize}

%In this appendix, we complement the numerical experiments in Section~\ref{sec:experiments} by discussing implementation details and presenting additional numerical experiments. In Appendix~\ref{appx:extra_experiments:implementation}, we discuss in detail our implementations of the various solution methods from Section~\ref{sec:experiments}, with the aim of transparency as well as reproducibility. In Appendix~\ref{appx:M_spread}, we present additional numerical results regarding the synthetic data generation process that is used in Section~\ref{sec:experiments:MNL_Cutoff}.  In Appendix~\ref{appx:exclusion_set_budget}, we conduct additional numerical experiments to those in Section~\ref{sec:experiments:mnl:xset} to consider assortment optimization problems with cardinality constraints. In Appendix~\ref{appx:cross_fold}, we 
\subsection{Implementation details} \label{appx:extra_experiments:implementation}
In Section~\ref{sec:experiments}, our goal is to perform numerical experiments that reflect an honest comparison between the various solution methods proposed in this paper and from the prior literature. In this appendix, we discuss in detail our implementations of the various solution methods from Section~\ref{sec:experiments}, with the aim of transparency as well as reproducibility. 
\subsubsection{Parameter tuning and optional heuristics.}  

In our experiments in Section~\ref{sec:experiments}, our goal is to conduct a direct comparison between various solution methods. To this end, we conducted all of our experiments without any optional heuristics or parameter tuning, both for our solution methods as well as the solution methods from the literature. This includes, for example, not tuning the Gurobi mixed-integer programming parameters, not changing the default frequency of callbacks, and not applying any optional heuristics such as the divide-and-conquer warm start from \citet[Section 5.3]{bertsimas2019exact}. The only parameter that we specify is the required numerical tolerance in the original two-phase Benders decomposition method from Section~\ref{sec:benders} and our accelerated Benders decomposition method from Section~\ref{sec:main_alg} for  when to add violated cuts. For both of these decomposition methods, we use a tolerance of $10^{-6}$ for detecting a violated cut, as well as a tolerance of $10^{-6}$ for terminating the decomposition method.\looseness=-1

While optional heuristics and parameter tuning  may lead to speedups in the various solution methods,  our rationale for not including them are three-fold. First,  because our goal in Section~\ref{sec:experiments} is to provide an honest comparison of various solution methods, we do not want the reported computation times to reflect a real or perceived disparity in efforts for choosing optional heuristics and tuning parameters more carefully for some solution methods than others. Second,  the spirit of sample average approximation is that it can provide a general approach to solving assortment optimization problems; as such, we believe it would not be ideal for a practitioner to need to select optional heuristics or tune solver parameters separately for different problem instances to obtain good performance. Third, there is a large number of combinations of solver parameters that can be tuned and optional heuristics that could be considered, both for our solutions methods and existing solution methods, making it impractical to conduct a comprehensive comparison of all combinations of optional heuristics and parameter tuning. 
\subsubsection{Implementation of original two-phase Benders decomposition method.} \label{appx:misic_benders_remark}
In Sections~\ref{sec:experiments:mnl:accelerated} and \ref{sec:results:toubia}, we report the computation times of the original two-phase Benders decomposition method from \citet[Section 4]{bertsimas2019exact}, which is reviewed in Section~\ref{sec:benders} and  Appendix~\ref{appx:benders_review}.
 Our implementation of their two-phase Benders decomposition method is based on the companion code for \cite{bertsimas2019exact} that is made publicly available by the authors.\footnote{\url{https://github.com/vvmisic/optimalPLD}, accessed on September 5, 2025.}\looseness=-1 
 
 Upon studying their code, we identified a small typo in their implementation of their two-phase Benders decomposition method. This typo did not affect the correctness of their code; however, this typo made their implementation of their algorithm   considerably \emph{slower} than the algorithm described in the main text of  \cite{bertsimas2019exact}. As we show  in Tables~\ref{tab:total_time_objective_bmoriginal}  and \ref{tab:large_benders_time_phases_bmoriginal} below, making a small fix to this typo in their code results in significant speedups in the two-phase Benders decomposition method for the empirical study conducted by \citet[Section 5.3]{bertsimas2019exact}.  To provide an honest comparison, all numerical results in Section~\ref{sec:experiments} report the computation time of their implementation of the original two-phase Benders decomposition method with our fix to the typo.\looseness=-1 
 
In greater detail, the typo in their implementation concerned their algorithm for generating cuts in Phase 2. Their algorithm for computing optimal cuts for Phase 2 is described in  \citet[Equations (11a)-(11c) in Section 4.1]{bertsimas2019exact}. In particular, their equation (11a) requires calculating  a vector $\bar{\bbeta} \in \R^{N+1}$ where $\bar{\beta}_\ell^k \triangleq \sum_{j: \sigma_k(j) < \ell} \beta^k_j$ for all $\ell \in [N+1]$.  This vector can be calculated in $\mathcal{O}(N)$ time using the recursion $\bar{\beta}_\ell^k = \bar{\beta}^k_{\ell-1} + \beta^k_{\sigma_k^{-1}(\ell-1)}$ for each $\ell$. However, their code for their implementation of this algorithm calculates  $\bar{\beta}_\ell^k$ separately  for each $\ell$, which requires $\mathcal{O}(N)$ time per $\ell$. This increased the runtime of their Phase 2 algorithm for generating cuts  from $\mathcal{O}(KN)$ to $\mathcal{O}(K N^2)$.   For this reason, the computation times of their Benders decomposition method reported in  \citet[Section 5.3]{bertsimas2019exact} made their method appear slower than it should be.

In all of our numerical experiments in Section~\ref{sec:experiments}, we report computation times for our fixed version of their code that computes cuts in Phase 2 in $\mathcal{O}(KN)$ time.  To demonstrate the significance of this fix, Tables~\ref{tab:total_time_objective_bmoriginal}  and \ref{tab:large_benders_time_phases_bmoriginal} repeat the numerical results shown in Tables~\ref{tab:total_time_objective} and \ref{tab:large_benders_time_phases} from  Section \ref{sec:results:toubia}, appended with additional numerical results that show the computation time of their original implementation of the two-phase Benders decomposition made publicly available by \cite{bertsimas2019exact}. These tables show that our small fix to the typo in their code results in 2x-20x speedup in their Phase 2 for the empirical study conducted in \citet[Section 5.3]{bertsimas2019exact}.

\afterpage{%
\null
\vfill
\begin{table}[t]
\TABLE{Experiments from Appendix \ref{appx:misic_benders_remark} - Computation Time.\label{tab:total_time_objective_bmoriginal}}
{\centering
\begin{tabular}{l @{\hspace{10pt}} 
                r @{\hspace{18pt}} 
                r @{\hspace{18pt}} 
                r @{\hspace{24pt}} }
\toprule
Constraints & 
\multicolumn{3}{c}{Computation Time} \\
 \cmidrule(lr){2-4} 
 & \shortstack{ABD} 
& \shortstack{BD${}^1$} 
& \shortstack{BD${}^0$} \\
\midrule
$\sum x_i = 2$ & 3.40 & 106.47 & 180.95 \\ 
$\sum x_i = 3$ & 5.87 & 276.21 & 480.89 \\ 
$\sum x_i = 4$ & 14.05 & 392.50 & 675.41 \\ 
$\sum x_i = 5$ & 18.09 & 387.86 & 903.59 \\ 
$\sum x_i = 6$ & 167.14 & 705.85 & 1845.19 \\ 
$\sum x_i = 7$ & 157.88 & 861.02 & 2386.83 \\ 
$\sum x_i = 8$ & 371.61 & 1396.28 & 3566.60 \\ 
$\sum x_i = 9$ & 508.31 & 2177.62 & 4525.05 \\ 
$\sum x_i = 10$ & 2215.70 & 4240.28 & 9018.24 \\ 
\bottomrule
\end{tabular}}{Results under different cardinality constraints for the computation time in seconds for three solution methods. The first is the accelerated Benders decomposition method from \S\ref{sec:main_alg} (ABD). The second is the implementation of the two-phase Benders decomposition method from \S\ref{sec:benders} given by \cite{bertsimas2019exact} with our improvement to their code (BD${}^1$). The third is the implementation of the two-phase Benders decomposition method from \S\ref{sec:benders} given by \cite{bertsimas2019exact} without our improvement to their code (BD${}^0$). Results are averaged over three replications to remove variability in solve times, and results are rounded to two decimal places. Table~\ref{tab:total_time_objective} in Section~\ref{sec:results:toubia} includes columns ABD and BD${}^1$.}
\end{table}

\begin{table}[t]
\TABLE{Experiments from Appendix \ref{appx:misic_benders_remark} - Computation Time by Phase\label{tab:large_benders_time_phases_bmoriginal}}
{\centering
\begin{tabular}{l
                r r@{\hspace{18pt}}  % Bertsimas and Misic
                r r@{\hspace{18pt}}  % Bertsimas and Misic Improved
                r r@{\hspace{18pt}}  % Our Formulation
                } 
\toprule
& 
\multicolumn{2}{c}{ABD} & 
\multicolumn{2}{c}{BD${}^1$} & \multicolumn{2}{c}{BD${}^0$} 
 \\
\cmidrule(lr){2-3} \cmidrule(lr){4-5} \cmidrule(lr){6-7}
Constraints  & \multicolumn{1}{c}{Phase 1} & \multicolumn{1}{c}{Phase 2} 
& \multicolumn{1}{c}{Phase 1} & \multicolumn{1}{c}{Phase 2} 
& \multicolumn{1}{c}{Phase 1} & \multicolumn{1}{c}{Phase 2} \\
\midrule
$\sum x_i = 2$ & 2.78 & 0.62 & 102.66 & 3.80 & 102.64 & 78.31 \\ 
$\sum x_i = 3$ & 3.31 & 2.55 & 263.99 & 12.22 & 273.05 & 207.84 \\ 
$\sum x_i = 4$ & 3.40 & 10.65 & 332.14 & 60.36 & 291.12 & 384.29 \\ 
$\sum x_i = 5$ & 4.22 & 13.87 & 333.92 & 53.94 & 340.17 & 563.42 \\ 
$\sum x_i = 6$ & 4.05 & 163.09 & 404.12 & 301.73 & 411.84 & 1433.35 \\ 
$\sum x_i = 7$ & 4.75 & 153.12 & 464.12 & 396.90 & 480.03 & 1906.80 \\ 
$\sum x_i = 8$ & 4.94 & 366.68 & 497.50 & 898.79 & 463.73 & 3102.87 \\ 
$\sum x_i = 9$ & 5.23 & 503.08 & 590.07 & 1587.55 & 611.46 & 3913.59 \\ 
$\sum x_i = 10$ & 5.39 & 2210.31 & 651.65 & 3588.63 & 686.38 & 8331.87 \\ 
\bottomrule
\end{tabular}}
{Results under different cardinality constraints for the computation time in seconds for three solution methods, split by Phase 1 and Phase 2. The first is the accelerated Benders decomposition method from \S\ref{sec:main_alg} (ABD). The second is the implementation of the two-phase Benders decomposition method from \S\ref{sec:benders} given by \cite{bertsimas2019exact} with our improvement to their code (BD${}^1$). The third is the implementation of the two-phase Benders decomposition method from \S\ref{sec:benders} given by \cite{bertsimas2019exact} without our improvement to their code (BD${}^0$). Results are averaged over three replications to remove variability in solve times, and results are rounded to two decimal places. Table~\ref{tab:large_benders_time_phases} in Section~\ref{sec:results:toubia} includes columns ABD and BD${}^1$.}
\end{table}
\null
\vfill
\clearpage
}%

\subsubsection{Initial cuts in Benders decomposition.} \label{appx:extra_experiments:initial_cuts}
To avoid unboundedness of the outer problem~\eqref{prob:outer}, both our accelerated Benders decomposition method from Section~\ref{sec:main_alg} and the original two-phase Benders decomposition method from \cite{bertsimas2019exact} described in Section~\ref{sec:benders} initialize the outer problem~\eqref{prob:outer} in Phase 1 with a single constraint for each ranking.  Consistent with the implementation given by \cite{bertsimas2019exact}, the original two-phase Benders decomposition method is initialized with the constraint $q_k \le \max_{i \in [N+1]} r_i$ for each ranking $k$. Our accelerated Benders decomposition method is initialized with the constraint $q_k \le \max_{\ell \in [L_k+1]} r_{i_{k,\ell}}$, which corresponds to the cut that would be obtained by \eqref{prob:dual_reform} from Section~\ref{sec:main_alg:reform} with  $\delta_1 = \cdots = \delta_{L_k+1} = \bar{r}_k$.   
\subsection{Discussion of the data generation process from Section~\ref{sec:data_generation}} \label{appx:M_spread}
The data generation process described in Section~\ref{sec:data_generation} consists of generating a "ground truth" ranking-based choice model with $M$ rankings (Step 1) and fitting the attraction parameters of the multinomial logit model to transaction data generated from that ranking-based choice model (Step 2).  In this appendix,  we perform additional experiments to demonstrate the effect of the parameter $M$ on the  multinomial logit models with rank cutoffs obtained from this data generation process.\looseness=-1

The results from these additional  experiments  are shown in Figure~\ref{fig:product_dist_by_M}. Similarly as in Section~\ref{sec:experiments:mnl:xset}, Figure~\ref{fig:product_dist_by_M} considers a setting with $N = 50$ products, a rank cutoff of $L = 5$, and $M \in \{5,15,25\}$. For each $M \in \{5,15,25\}$, we perform the data generation process described in Steps 1 and 2 of  Section~\ref{sec:data_generation} to construct a multinomial logit model with rank cutoffs, and we then generate $\tilde{K} = 50000$ samples from the constructed multinomial logit model with rank cutoffs  by following Steps 3 and 4 of Section~\ref{sec:data_generation}. In Figure~\ref{fig:product_dist_by_M}, we report for each product the proportion of samples in which that product is preferred to the no-purchase option, averaged across ten replications for each $M \in \{5,15,25\}$.\footnote{Because the rank cutoff is $L = 5$, we observe that at most $5$ products are preferred to the no-purchase option in any given sample.  As such, for each $M \in \{5,15,25\}$, the sum of the proportions in Figure~\ref{fig:product_dist_by_M} is at most 5.0. } Consistent with the data generation process described in Section~\ref{sec:data_generation}, the products  in Figure~\ref{fig:product_dist_by_M} are sorted within each replication so that the attraction parameters satisfy $\nu_1 \ge \cdots \ge \nu_{N}$.\looseness=-1

There are two main takeaways from Figure~\ref{fig:product_dist_by_M}. First, we observe that as the number of rankings $M$ in the ground truth model increases, the proportion of samples in which each product is preferred to the no-purchase option becomes more evenly distributed across products. This reflects the fact that larger values of $M$ leads to a smaller spread of the attraction parameters in the estimated multinomial logit model. Second, we observe that even for the smallest choice of $M$ ($M = 5$),  each product is preferred to the no-purchase option in at least some of the samples.

\begin{figure}[t] 
\FIGURE{\centering \includegraphics[width=0.5\textwidth]{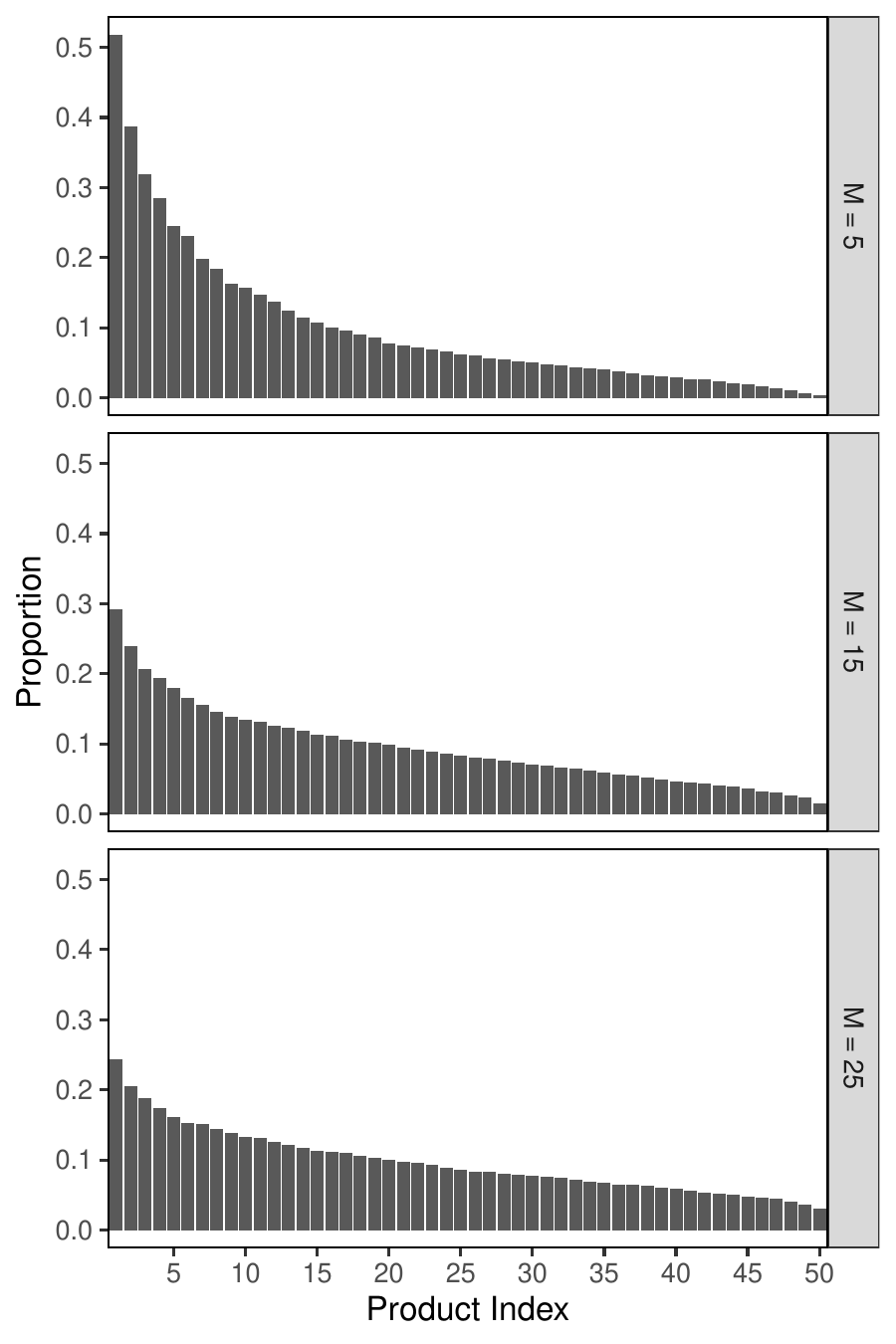}}
{Experiments from Appendix \ref{appx:M_spread} - Visualization of samples.\label{fig:product_dist_by_M}}{Proportion of samples in which each product is preferred to no-purchase option.}
\end{figure}

\subsection{Additional Experiments from Section~\ref{sec:experiments:mnl:xset} }\label{appx:exclusion_set_budget}
In this appendix, we present additional numerical experiments to those in Section~\ref{sec:experiments:mnl:xset} to explore the impact of cardinality constraints on the computation time of the exclusion set formulation~\eqref{prob:exclusionset} and the original mixed-integer programming formulation~\eqref{prob:misic_no_extra}. 

Our additional experiments follow a similar setup as  Section~\ref{sec:experiments:mnl:xset}. Specifically, we use the same data generation process from  Section~\ref{sec:data_generation} with   $N = 50$, $M=5$,  $\tilde{K} = 10000$, and $L=5$. We then compare the performance of the exclusion set formulation~\eqref{prob:exclusionset} and the original mixed-integer programming formulation~\eqref{prob:misic_no_extra} with  budget constraints of the form $\sum_{i=1}^N x_i \le 5$, $\sum_{i=1}^N x_i \le 10$, \ldots, $\sum_{i=1}^N x_i \le 45$,  $\sum_{i=1}^N x_i \le 50$.  Table~\ref{tab:es_cardinality_constraint}  shows the computation time for solving \eqref{prob:exclusionset}  and \eqref{prob:misic_no_extra}, and it also shows the gap of each method of the root relaxation, i.e., the ratio of the optimal objective values of the linear programming relaxations to the true optimal objective values.\looseness=-1 

\begin{table}[t]
\TABLE{Experiments Showing Exclusion Set Performance Under Cardinality Constraints. \label{tab:es_cardinality_constraint}}
{
\begin{tabular}{crrrrrr}
\toprule
{Budget} & 
\multicolumn{3}{c}{\shortstack{Computation Time}} & 
\multicolumn{3}{c}{Root Relaxation Gap} \\
\cmidrule(lr){2-4} \cmidrule(lr){5-7}
& \shortstack{XSET} & \shortstack{MIP}  & \shortstack{Ratio}
& \shortstack{XSET} & \shortstack{MIP} & \shortstack{Ratio}\\
\midrule
5 & 48.09 & 97.80 & 2.03 & 6.44\% & 6.99\% & 1.09  \\ 
10 & 281.00 & 549.89 & 1.96 & 10.29\% & 11.32\% & 1.10  \\ 
15 & 486.80 & 887.64 & 1.82 & 8.40\% & 10.75\% & 1.28  \\ 
20 & 207.06 & 550.55 & 2.66 & 4.11\% & 7.13\% & 1.74  \\ 
25 & 86.67 & 273.10 & 3.15 & 1.56\% & 3.87\% & 2.48  \\ 
30 & 11.59 & 32.83 & 2.83 & 0.44\% & 1.78\% & 4.07  \\ 
35 & 2.34 & 22.80 & 9.74 & 0.05\% & 1.02\% & 18.75  \\ 
40 & 1.51 & 18.12 & 11.96 & 0.02\% & 0.95\% & 58.16  \\ 
45 & 1.59 & 15.35 & 9.64 & 0.02\% & 0.95\% & 58.16  \\ 
Inf & 1.48 & 15.83 & 10.73 & 0.02\% & 0.95\% & 58.16  \\  
\bottomrule
\end{tabular}
}{Results for two solution methods under different cardinality constraints. The two solution methods are our exclusion set formulation~\eqref{prob:exclusionset} (XSET) and the original mixed-integer programming formulation~\eqref{prob:misic_no_extra} (MIP).  The Computation Time columns report the solution time (in seconds) required to solve each formulation under integrality constraints, along with the ratio of these times. The Root Relaxation Gap columns report the average percentage deviation of the root relaxation from the optimal integral solution for both formulations, as well as the ratio of these gaps. Results are averaged over five replications, and results are rounded to two decimal places.}
\end{table}

There are three main takeaways from Table~\ref{tab:es_cardinality_constraint}. First, we observe that the root relaxation gap for both methods decreases as the budget constraint decreases from 10 to 5. This decrease in the root relaxation gaps is consistent with Proposition~\ref{prop:integrality_budget_one} from Section~\ref{sec:exclusion_set:comparison}, which shows that the root relaxation gap of  \eqref{prob:exclusionset} and \eqref{prob:misic_no_extra} are both 0\% if we have a budget constraint of $\sum_{i=1}^N x_i \le 1$. Second, we observe that the computation times for solving \eqref{prob:exclusionset} and \eqref{prob:misic_no_extra} increase as the budget decreases from $50$ to $15$, and the computation times then decrease as the budget decreases from 15 to 5.  We attribute this to the fact that the root relaxation gaps for both formulations become larger as we decrease from $50$ to $15$, after which the root relaxation gaps for both formulations becomes smaller. Third, we observe that the ratio of computation times between \eqref{prob:exclusionset} and \eqref{prob:misic_no_extra} is closely correlated with the ratio of their root relaxation gaps. This demonstrates the value of the strength of the exclusion set formulation~\eqref{prob:exclusionset} compared to the original mixed-integer programming formulation~\eqref{prob:misic_no_extra} from the perspective of reducing the computation time for obtaining an optimal solution.

\subsection{Approximation Gap for Section~\ref{sec:results:toubia}} \label{appx:cross_fold}
In Section~\ref{sec:results:toubia}, we solve the sample average approximation~\eqref{prob:saa} using the dataset from \cite{toubia2003fast}, which consisted of $N = 3584$ products and $\tilde{K} = 330$ samples. This set of experiments is notably different than the experiments in Section~\ref{sec:experiments:MNL_Cutoff}  for two reasons. First, the dataset from \cite{toubia2003fast} contains far more products than samples, raising potential concerns  about the out-of-sample performance of assortments obtained from solving the sample average approximation. Second, the dataset from \cite{toubia2003fast} was  obtained using a conjoint analysis rather than Monte-Carlo simulation. As such,  it is not possible to obtain more samples to improve the approximation gap of the sample average approximation, nor is it possible to generate an out-of-sample validation set of samples to estimate the approximation gap of the assortments obtained from solving sample average approximation, as was done in Section~\ref{sec:experiments:mnl:accelerated}. 

To address these potential concerns, Table~\ref{tab:toubia_cross_validation}  presents the results of a five-fold cross-validation conducted on the dataset of \cite{toubia2003fast}  from Section \ref{sec:results:toubia}.   Specifically, after randomizing the dataset, we partitioned it into five equal subsets. For each budget constraint, we performed five iterations, using four subsets for the sample average approximation~\eqref{prob:saa} and evaluating the revenue performance of the optimal assortment on the remaining hold-out subset. This process allows us to assess how well the optimal assortments generalize to unseen data across different cardinality constraints.  To the best of our knowledge, Table~\ref{tab:toubia_cross_validation} offers  the first empirical analysis of the approximation gap from using the dataset of \cite{toubia2003fast}.\looseness=-1 

\begin{table}[t]
\TABLE{Experiments from Appendix~\ref{appx:cross_fold} - Estimated Approximation Gap.\label{tab:toubia_cross_validation}
}{
\centering
\begin{tabular}{l c}
\toprule
Constraints & \shortstack{Approximation Gap} \\
\midrule
$\sum x_i = 2$ & 94.96\% \\ 
$\sum x_i = 3$ & 95.55\% \\ 
$\sum x_i = 4$ & 96.99\% \\ 
$\sum x_i = 5$ & 96.14\% \\ 
$\sum x_i = 6$ & 97.26\% \\ 
$\sum x_i = 7$ & 97.46\% \\ 
$\sum x_i = 8$ & 95.82\% \\ 
\bottomrule
\end{tabular}}
{Relative accuracy (100\% - MAPE) of optimal assortments from the sample average approximation~\eqref{prob:saa}, solved using the accelerated Benders decomposition method from \S\ref{sec:main_alg}. MAPE is based on the revenue of each optimal assortment on the training data versus its performance on the out-of-sample data. Results are averaged over five-fold cross-validation (Appendix~\ref{appx:cross_fold}) and reported to two decimal places.}
\end{table}

Table~\ref{tab:toubia_cross_validation} reveals that the approximation gap of the sample average approximation~\eqref{prob:saa} on the dataset from \cite{toubia2003fast} is surprisingly small, and substantially smaller than the one reported in Table \ref{tab:saa-convergence}. This indicates that this dataset may possess structural properties, potentially stemming from the data collection methodology employed by \cite{toubia2003fast}, that facilitate strong generalization. The consistently low approximation gaps observed on this real-world dataset of $N = 3584$ products provide additional evidence of the viability of the sample average approximation approach for solving the assortment optimization problems with realistic numbers of products.

\end{APPENDICES}
\end{document}